\documentclass[reqno,12pt,letterpaper]{amsart}
\usepackage[utf8]{inputenc}
\usepackage{amsmath}
\usepackage{amsfonts}
\usepackage{amssymb}
\usepackage{amsthm}
\usepackage[dvipsnames]{xcolor}
\usepackage{enumitem}
\usepackage{mathtools}
\usepackage{mathrsfs}
\usepackage{hyperref}
\usepackage{subcaption}
\usepackage{graphicx}
\usepackage{slashed}

\hypersetup{colorlinks = true,
            citecolor=NavyBlue,
            linkcolor=Green}

\numberwithin{equation}{section}

\let\oldtocsection=\tocsection
\let\oldtocsubsection=\tocsubsection
\renewcommand{\tocsection}[2]{\hspace{0em}\oldtocsection{#1}{#2}}
\renewcommand{\tocsubsection}[2]{\hspace{1em}\oldtocsubsection{#1}{#2}}

\usepackage{cases}

\usepackage{comment}

\newcommand{\RR}{{\mathbb R}}

\newcommand\cD{{\mathcal  D}}
\newcommand\cU{{\mathcal  U}}

\newcommand\cG{{\mathcal  G}}

\newcommand\cO{{\mathcal O}}

\newcommand\cT{{\mathcal T}}
\newcommand\cS{{\mathcal S}}

\newcommand\eps{{\epsilon}}

\newcommand{\mr}{{\mathrm{r } }}
\newcommand{\mi}{{\mathrm{i } }}
\newcommand{\md}{{\mathrm{d } }}

\newcommand{\ma}{{\mathrm{a } }}
\newcommand{\snabla}{\slashed{\nabla}}

\newcommand{\bR}{\mathbb{R}}

\newcommand{\bN}{\mathbb{N}}
\newcommand{\bT}{\mathbb{T}}

\newcommand\adots{\mathinner{\mkern2mu\raise1pt\hbox{.}
\mkern3mu\raise4pt\hbox{.}\mkern1mu\raise7pt\hbox{.}}}

\newtheorem{proposition}{Proposition}[section]
\newtheorem{theo}{Theorem}
\newtheorem{lem}[proposition]{Lemma}

\newtheorem{prop}[proposition]{Proposition}
\newtheorem{ass}[proposition]{Assumption}

\newtheorem{exam}[proposition]{Example}

\theoremstyle{definition}
\newtheorem{defn}[proposition]{Definition}

\theoremstyle{remark}
\newtheorem{rem}[proposition]{Remark}

\newtheorem{nota}[proposition]{Notations}

\numberwithin{equation}{section}
\begin{document}

\title{Propagation of nonlinear pulses near diffractive points of any order}

\author{Jian Wang}
\email{wangjian@ihes.fr}
\address{Institut des Hautes Etudes Scientifiques, 91893 Bures-sur-Yvette, France}
\author{Mark Williams}
\email{williams@math.unc.edu}
\address{Department of Mathematics, University of North Carolina, Chapel Hill, NC 27514}

\begin{abstract}
    We construct pulse-type approximate solutions to nonlinear hyperbolic equations near diffractive points, allowing arbitrary (even infinite) order of grazing. We show that in low regularity spaces and the high frequency limit, such solutions can be approximated by a sum of incoming and reflected pulses constructed using incoming and reflected phases and profiles that satisfy transport equations. New low-regularity estimates comparing the size of pulses to the size of their profiles are required. Earlier geometric optics results for pulses assumed much higher regularity, and considered only propagation in free space or transversal reflection at boundaries.
\end{abstract}

\maketitle

{\hypersetup{linkcolor=NavyBlue}
\vspace{-.6cm}
\tableofcontents}

\newpage

\section{Introduction}

In this paper we consider the diffraction of low regularity pulses traveling on rays that graze boundaries to any order, including infinite order; see Definition \ref{q2}.\footnote{When we speak of the ``order" of a grazing ray, we are referring to its order of tangency with the boundary.   A ``grazing point of order $k$" is a diffractive point of order $2k$ in the language of Definition \ref{q2}.}   

\subsection{Main results for the exterior domain of a convex obstacle}
To explain our results  with minimal preparation, let us first consider special case of the nonlinear wave equation in the exterior domain of a convex obstacle. Let $\mathcal O\subset\RR^n$ be a convex open domain such that $\partial\mathcal O$ is smooth and does not contain straight line segments. For $\epsilon>0$ and $T>0$, we consider the following problem
\begin{equation}\label{eq:obstacle}
\begin{cases}
    (\partial_t^2-\Delta)u^\eps=f(t,x,u^\eps,\partial_t u^\eps, \nabla u^\eps), \ & (t,x)\in [-T,T]\times (\RR^n\setminus \overline{\mathcal O}), \\
    u^\eps(t,x)=0, \ & (t,x)\in [-T,T]\times \partial\mathcal O, \\
    u^\eps(-T, x) = \sqrt{\epsilon} a_\mi(x)V_\mi( \frac{-T+\langle \theta_0,x \rangle }{\epsilon} ), \ & x\in \RR^n\setminus \overline{\mathcal O}, \\
    \partial_t u^\eps(-T,x) = -\frac{1}{\sqrt{\epsilon}} a_\mi(x)V_\mi'(\frac{-T+\langle \theta_0,x \rangle}{\epsilon}), \ & x\in \RR^n\setminus \overline{\mathcal O}.
\end{cases}
\end{equation}
Here we assume that
\begin{itemize}
    \item $f\in C^{\infty}$, $f(t,x,0,0,0)=0$, and $f(t,x,\cdot, \cdot, \cdot)$ and $\nabla_{t,x}f(t,x,\cdot,\cdot,\cdot)$ are Lipschitz;
    \item both $a_\mi\in H^2(\RR^n\setminus \overline{\mathcal O}; \RR)$ and $V_\mi\in H^2(\RR;\RR)$ are compactly supported;
    \item $\theta_0\in \mathbb S^{n-1}$ and there exists $x_0\in \RR^n\setminus \overline{\mathcal O}$ such that $a_\mi(x_0)>0$ and the ray $x_0+\theta_0\mathbb R_+$ intersects $\partial\mathcal O$ tangentially at a point $x_\md$; 
    \item the {\em glancing} set $G_{\phi_\mi}\coloneqq \{ (t,x_\md)\mid (x_\md,\theta_0)\in T^*\partial\mathcal O \}$ is a codimension two $C^1$ submanifold of $\RR^{1+n}$. Here $T^*\partial\mathcal O$ is the cotangent bundle of $\partial\mathcal O$, regarded as a subset of $T^*\RR^n$.
\end{itemize}
Equation \eqref{eq:obstacle} describes the propagation of a plane wave pulse in the exterior of a convex obstacle. At the initial time, the pulse propagates along characteristic lines of the incoming phase $\phi_{\mi}(t,x)\coloneqq -t+\langle \theta_0, x \rangle$, which are straight lines $x+\theta_0 \mathbb R$, $x\in \mathrm{supp } \, a_\mi$. Upon hitting the boundary of the obstacle, reflected waves are created. The propagation of the reflected waves is described by a reflected phase $\phi_\mr=\phi_\mr(t,x)$. Roughly speaking, the ``wave fronts'' of the reflected waves are given by the level sets of the reflected phase.

Our main result approximates the exact solution to \eqref{eq:obstacle}. In particular, it shows that in the $H^1([-T,T]\times (\RR^n\setminus \overline{\mathcal O}))$ norm and in the high frequency limit $\epsilon\to 0+$, the pulse near the diffractive point can be approximated as a sum of  incoming  and reflected pulses.
\begin{theo}\label{thm:obstacle}
    Let $\phi_\mi$ and $\phi_\mr$ be the incoming and reflected phase functions respectively. Then there exist sequences of $C^\infty$ functions $U^{\ell, \epsilon}_{\bullet}(t,x,\theta)$, $\bullet=\mi, \mr$, $\ell\in \mathbb N$, such that 
    \[ \lim_{\ell\to +\infty}\limsup_{\epsilon\to 0+}\left\|u^\eps-\sqrt{\epsilon}\left(U^{\ell,\epsilon}_{\mi}\left(t,x,\frac{\phi_\mi}{\epsilon}\right) + U_{\mr}^{\ell,\epsilon}\left(t,x,\frac{\phi_\mr}{\epsilon}\right)\right)\right\|_{H^1([-T,T]\times (\RR^n\setminus \overline{\mathcal O}))}=0. \]
    Moreover, $U^{\ell,\epsilon}_{\bullet}$ has compact $(t,x)$-support in $[-T,T]\times (\RR^n\setminus \overline{\mathcal O})$, is rapidly decaying in $\theta$ and has moment zero.   The functions $U^{\ell,\epsilon}_{\bullet}$ are truncated and smoothed  approximations to $L^2$ solutions $U_{\bullet}$ of nonlinear transport equations.
\end{theo}

A more precise  and general version of Theorem \ref{thm:obstacle} for  second order strictly hyperbolic differential operators and more general incoming phases is stated in Theorem \ref{mta}. In general, let us consider the space-time domain 
\[ \Omega_T\coloneqq [-T,T]_t\times (\overline{\RR_+})_{x_1}\times \RR^{n-1}_{x'} \ \text{ with } \ T>0 \]
where $\RR_+\coloneqq (0,+\infty)$. We will write $x=(x_1,x')$.
Let $P = P(t,x,\partial_t,\nabla_x)$ be a second-order scalar differential operator with smooth coefficients. We assume $P$ is strictly hyperbolic with respect to $t$ and satisfies Assumption~\ref{amr1} below. We consider the following initial boundary value problem
\begin{equation}\label{in1}
    \begin{cases}
        Pu^\epsilon = f(t,x,u^\epsilon, \partial_tu^\epsilon, \nabla u^\epsilon), \ &  (t,x)\in \Omega_T, \\
        u^\epsilon(t,0,x')=0, \ & (t,0,x')\in \Omega_T,\\
        u^\epsilon(-T,x)=g^\epsilon(x), \ \partial_t u^\epsilon(-T,x)=h^\epsilon(x), \ & x\in \RR_+\times \RR^{n-1}.
    \end{cases}
\end{equation}
As in the model case above, we assume 
\begin{itemize}
    \item $f\in C^{\infty}$, $f(t,x,0,0,0)=0$, and $f(t,x,\cdot, \cdot, \cdot)$ and $\nabla_{t,x}f(t,x,\cdot,\cdot,\cdot)$ are Lipschitz;

    \item the initial data are pulses, that is,
        \begin{align}\label{a1aw}
        \begin{split}
            &g^\eps=\sqrt{\eps}V_0\left(x,\frac{\phi_\mi(-T,x)}{\eps}\right)+o_\eps(1) \text{ in }H^1(\overline{\bR^n_+}),\\
            &h^\eps=\frac{1}{\sqrt{\eps}}V_1\left(x,\frac{\phi_{\mi}(-T,x)}{\eps}\right)+o_\eps(1) \text{ in }L^2(\overline{\bR^n_+}). %\;\phi_0(x)=\phi(x,0).
        \end{split}
        \end{align}
        Here $\phi_\mi=\phi_\mi(t,x)$ is a characteristic incoming phase $\phi_i$ as in Assumption~\ref{amr3}, and 
\begin{subequations}\label{a1cz}
    \begin{align}
        \label{a1cza}
        & V_0\in H^2(\overline{\bR^n_+}\times\bR_\theta), \\
        \label{a1czb}
        & V_1=(\partial_t\phi_\mi(-T,x))\partial_\theta V_0\in H^1(\overline{\bR^n_+}\times\bR_\theta).
    \end{align}
\end{subequations}
Let $\gamma_\mi$ denote an incoming characteristic of $T_{\phi_\mi}$ (see \eqref{ezb} for the definition) that passes through the diffractive point $0$.   The function $V_0(x,\theta)$ is chosen to have compact $(x,\theta)$-support with $x$-support in a small neighborhood of $\gamma_i\cap \{t= -T\}$ contained in $x_1>0$.   
\end{itemize}

In Theorem \ref{mta}, approximate solutions to the initial boundary value problem \eqref{in1} are constructed.  As in the  Theorem \ref{thm:obstacle} for the model case, the approximate solutions are sums of incoming pulses and reflected pulses.

\begin{rem}\label{a1cy}
Condition \eqref{a1czb} is a polarization condition which implies that to leading order a single wave with phase $\phi_\mi$ emerges into $t>-T$.  If we omit that condition, distinct waves associated to $\phi_\mi$ and a second phase $\tilde \phi_\mi$ will emerge. %{\color{red}ref?}
\end{rem}

The hyperbolic boundary problem  \eqref{in1} is similar to the one studied in \cite{wangwil}, but now the 
incoming wave is a \emph{pulse} rather than a wavetrain as in \cite{wangwil}.    The incoming  pulse is constructed  so that it is transported along  spacetime characteristics that include rays which are allowed to graze the boundary to {any} order.   The main goal of the paper is to construct explicit approximate solutions $u^\eps_\ma$ to \eqref{in1} that qualitatively describe the solution near such grazing points and which are close to the exact solutions $u^\eps$ in the $H^1(\Omega_T)$ norm  on a fixed time interval independent of $\eps$ when $\eps$ is small.

\begin{exam}\label{fundex}
The fundamental motivating example  is the case where the spacetime manifold is $M=\mathbb{R}_t\times (\mathbb{R}^n\setminus \mathcal{O})$, 
where $\mathcal{O}\subset \mathbb{R}^n$ is an open convex  
obstacle with $C^\infty$ boundary, and the governing hyperbolic operator is the wave operator $\Box \coloneqq -\partial_t^2+\Delta$, as described in Theorem \ref{thm:obstacle}.   
This example can always be reduced by a local change of variables flattening the boundary to a problem of the form~\eqref{in1}.
\end{exam}

\subsection{Background}

Rigorous nonlinear geometric optics for pulses  has  been considered by several authors --- mainly by Alterman--Rauch \cite{altermanrauch, ar2}, Carles--Rauch \cite{cr}, Coulombel--Williams \cite{CW1,CW2},   Williams \cite{williams5}, and C. Willig \cite{willig}.   In all these cases  the pulses being considered were of \emph{high} regularity, that is, $H^s$ for $s$ big.  
Some of these papers considered propagation, even focusing,  in free space, others considered \emph{transversal} reflection off boundaries; none dealt with diffraction.\footnote{The word ``diffractive" occurs in the title of \cite{altermanrauch}, but there it has a completely different meaning from that used here.}

The diffraction of low regularity \emph{wavetrains} in the case where only \emph{first-order} grazing rays are present was studied by Cheverry \cite{cheverry1996}; see also \cite{dumas2002}.  A diffraction result for low-regularity wavetrains propagated by grazing rays of any order was given in \cite{wangwil}.   We believe that the present paper gives the first diffraction result for low-regularity pulses propagated by grazing rays of arbitrary order,  including first order.  The result for low regularity pulses in free space also appears to be new; earlier work required much more regular initial data.

Our study of \emph{low} regularity pulses is necessitated by our approach to the study of pulse diffraction.  
An outstanding open problem, even in the linear theory, is to construct highly regular approximate solutions (``parametrices") to diffraction problems that involve grazing rays of order higher than one.   In the 1970s Melrose \cite{mel1975} and Taylor \cite{taylor1976cpam} used Fourier-Airy integral operators to construct parametrices for diffraction problems that involved just first-order grazing.   The related problem of constructing 
highly regular wavetrain or pulse approximate  solutions  for  diffraction problems that involve grazing rays of order higher than one is likewise still quite open.   In the low-regularity category a significant portion of the complicated (and interesting) phenomena that make the high-regularity problem inaccessible at present  is invisible.  
This is the price we pay at the moment for a construction of approximate solutions that incorporates grazing rays of any order.

\subsection{Pulses versus wavetrains}

Some of the analysis from \cite{wangwil} carries over from the wavetrain case to the pulse case.  
The geometric set-up is the same.  The grazing set  and reflected flow map assumptions, as well as the procedures for verifying those assumptions, are unchanged; see Assumptions \ref{A02} and \ref{A03}.   These assumptions are needed to construct a reflected phase $\phi_\mr$ that is $C^1$ up to the shadow boundary.  Thus, in the pulse case the incoming and reflected characteristic phases $\phi_\mi$, $\phi_\mr$ are constructed exactly as before.  But pulses, and in particular low regularity pulses, present their own new difficulties, some of which we now describe.

For present purposes we write the approximate pulse solutions in the simplified form 
\begin{align}\label{in2}
u^\eps_\ma(t,x)=\sqrt{\eps}U_{\mi}(t,x,\theta)_{|\theta=\frac{\phi_\mi}{\eps}}+\sqrt{\eps}U_\mr(t,x,\theta)_{|\theta=\frac{\phi_\mr}{\eps}}.
\end{align}
Here  the incoming and reflected profiles  $U_\mi(t,x,\theta)$ and $U_\mr(t,x,\theta)$ are \emph{decaying} in $\theta$ and are defined as the unique decaying primitives in $\theta$ of  (moment zero) profiles $W_\mi(t,x,\theta)$, $W_\mr(t,x,\theta)$ constructed to satisfy certain transport equations.  Appendix~\ref{mz} discusses moment zero approximations.

The ansatz \eqref{in2} should be compared to the one for wavetrains 
\begin{align}\label{in3}
u^\eps_\ma(t,x)=\underline{u}(t,x)+{\eps}U_{\mi}(t,x,\theta)_{|\theta=\frac{\phi_\mi}{\eps}}+{\eps}U_\mr(t,x,\theta)_{|\theta=\frac{\phi_\mr}{\eps}}
\end{align}
where now $U_\mi(t,x,\theta)$ and $U_\mr(t,x,\theta)$ are \emph{periodic} in $\theta\in\bR$ and are defined as the unique mean zero periodic primitives  in $\theta$ of  mean zero profiles $W_\mi(t,x,\theta)$, $W_\mr(t,x,\theta)$ satisfying transport equations.

The absence of the $\underline{u}(t,x)$ term in \eqref{in2} reflects the fact that pulses have no well-defined mean.  The much larger amplitude $\sqrt{\eps}$ in \eqref{in2}, versus $\eps$ in \eqref{in3},  reflects the fact that pulses of a given amplitude are much smaller in Sobolev norms than wavetrains of the same amplitude.  
Indeed, since $W_\mi=\partial_\theta U_\mi$, the main contribution to the $H^1(\Omega_T)$ norm of $\eps U_\mi(t,x,\phi_\mi/\eps)$ is $$\|W_\mi(t,x,\phi_\mi/\eps)\nabla \phi_\mi\|_{L^2(\Omega_T)},$$ which is generally of size $\sim 1$ when $U_\mi$ is a wavetrain, but of size $\sqrt{\eps}$ when $U_\mi$ is a pulse.  In other words we  have 
   $$\frac{1}{\sqrt{\eps}}\|W_\mi(t,x, \phi_\mi/\eps)\nabla \phi_\mi\|_{L^2(\Omega_T)}\sim 1$$
   for pulses.    This estimate is proved  in \S \ref{size}, Proposition \ref{a11z}.  Thus, in order for the $H^1(\Omega_T)$ norm of 
$\eps^{\beta} U_\mi(t,x,\phi_\mi/\eps)$ to be of size $\sim 1$ when $U_\mi$ is a pulse, we need to take {$\beta=\frac12$}.

It would be of no interest to prove an error estimate $\|u^\eps-u^\eps_\ma\|_{H^1(\Omega_T)}\to 0$ as $\eps\to 0$ if each of 
$u^\eps$, $u^\eps_\ma$ \emph{separately} approaches $0$ in $H^1(\Omega_T)$ as $\eps\to 0$.  Thus, in view of the previous paragraph, to get an interesting nonlinear geometric optics result with an error estimate in the $H^1(\Omega_T)$ norm, we must consider pulses such that $( \partial_t u^\eps, \nabla u^\epsilon)$ and 
$(\partial_t u^\eps_\ma, \nabla u^\epsilon_\ma)$ have amplitude $\frac{1}{\sqrt{\eps}}$, which corresponds to size $\sim 1$ in $L^2(\Omega_T)$.

One of the difficulties with low regularity pulses (or low regularity wavetrains) is that 
in most cases $f(t,x,u,\partial_t u,\nabla u)$ fails to lie in any useful Banach space of functions when $f$ is nonlinear and $(u,\partial_t u, \nabla u)$ is not at least locally bounded.  But $f(t,x,u,\partial_t, \nabla u)$ lies in $L^2(\Omega_T)$ \emph{if} $f$ is Lipschitz in $(u,\partial_t u, \nabla u)$ and $(u,\partial_t u, \nabla u)\in L^2(\Omega_T)$.    Thus, as in the wavetrain case of \cite{wangwil}, here we consider Lipschitz nonlinearities in order for the nonlinear problem to make sense.\footnote{As in \cite{wangwil} most of the difficulties for the diffraction problem are already present in the linear case.}

A complicating feature of pulses is revealed in the estimate \eqref{a21}, which shows that when a pulse profile, $Z(t,x,\theta)$ for example,  is $H^1$ and $\phi$ is an appropriate  characteristic phase,
the size of 
$$\frac{1}{\sqrt{\eps}}\left\|Z(t,x,\phi/\eps)\right\|_{L^2(\Omega_T)}$$ 
for $\eps$ small  is determined by the \emph{trace} of $Z(t,x,\theta)$ on the hypersurface in $(t,x)$-space given by $\cS=\{(t,x) \mid \phi(t,x)=0\}$.\footnote{We require that the map $(t,x)\mapsto (\phi(t,x),x)$ is a $C^1$ diffeomorphism on a neighborhood of the $(t,x)$-support of $Z$.}   More precisely,  to control $\frac{1}{\sqrt{\eps}}\|Z(t,x,\phi/\eps)\|_{L^2(\Omega_T)}$ we must control the $L^2$ norm of the \emph{trace} of $Z(t,x,\theta)$ on that hypersurface.   Indeed, for pulses we have the estimates\footnote{The estimates \eqref{in3z} are proved in \S \ref{size}.}
\begin{subequations}\label{in3z}
    \begin{align}
    \label{in3za}
        & \frac{1}{\sqrt{\eps}}\|Z(t,x,\phi/\eps)\|_{L^2(\Omega_T)}\lesssim \|Z(t,x,\theta)\|_{H^1(\Omega_T\times \bR_\theta)},\\
        \label{in3zb}
        & \lim_{\eps\to 0}\frac{1}{\sqrt{\eps}}\|Z(t,x,\phi/\eps)\|_{L^2(\Omega_T)}=\left\|\frac{Z(t,x,\theta)}{|\nabla_{t,x} \phi|^{1/2}}\right\|_{L^2(\cS\times \bR_\theta)} \text{ for }Z\in H^1(\Omega_T\times \bR_\theta).
    \end{align}
\end{subequations}
Compare these to the wavetrain estimate (\cite{jmr1996cpam}) 
\begin{align}\label{in3y}
\limsup_{\eps\to 0}\|Z(t,x,\phi/\eps)\|_{L^2(\Omega_T)}\lesssim C\|Z(t,x,\theta)\|_{L^2(\Omega_T\times  \bR_\theta)}\text{ for }Z\in L^2(\Omega_T,H^1(\bT_\theta)).
\end{align}
The estimates \eqref{in3z} indicate that we need to construct  profiles $W_\mi$, $W_\mr$ with  an extra derivative of $(t,x)$-regularity in the pulse case versus the wavetrain case.  

But in  the pulse case a serious difficulty arises when we try to control the $H^1$ norm of the profile $W_\mr$. The problem is that the bad term $(P_1\phi_\mr)W_\mr$ appears in the transport equation for $W_\mr$; see \eqref{nl3}.  The phase $\phi_\mr$ is only $C^1$ up to the shadow boundary; second derivatives in $(t,x)$ blow up there, so $P_1\phi_\mr$ blows up there.\footnote{The same difficulty arises in the wavetrain case, but in view of \eqref{in3y} we need only  control the $L^2(\Omega_T,H^1(\bT))$ norm of $W_\mr$ in that case.   Observe that $\theta$ derivatives of $P_1\phi_\mr$ are zero.} 
   That  already makes the bad term hard to control when estimating the $L^2$ norm of $W_\mr$,  but a surprising cancellation of the term involving $(P_1\phi_\mr)W_\mr$ in the $L^2$ estimate of $W_\mr$, already observed in \cite{wangwil},  permits one to bound the $L^2$ norm of $W_\mr$.  In order to bound the $H^1$ norm of $W_\mr$ one has to differentiate the transport equation  for $W_\mr$, and hence must \emph{differentiate} the already unbounded factor $P_1\phi_\mr$!\footnote{The factor $P_1\phi_\mi$  appearing in the transport equation \eqref{nl4} for $W_\mi$ remains bounded even after differentiation.}
  The cancellation argument does not work now, so we have a new difficulty, which is addressed in the proof of Theorem~\ref{mta} by new arguments.  See especially steps \textbf{3-14} of the proof of that theorem.  Steps \textbf{11-14} use \eqref{in3za}, while an estimate like \eqref{in3zb} is used in step \textbf{12} to control the term ``$C$". 

    We solve the profile equations using truncated initial data to construct truncated approximations to $W_\mr$ and $W_\mi$, call them $W_\mr^\mu$ and $W^\mu_\mi$, which are both supported away from the shadow boundary.  The supports of $W^\mu_\mr$ and $W^\mu_\mi$ approach the shadow boundary as $\mu\to 0$. We  show that we can control the $H^1$ norms of $W_\mr^\mu$ and $W^\mu_\mi$ for each fixed value of the truncation parameter $\mu$, even though we do not get uniform control as $\mu\to 0$. 
  Along with the fact that $W^\mu_\bullet\to W_\bullet$ in $L^2$ for  $\bullet=\mr, \mi$ as $\mu\to 0$, 
     this turns out to be enough for the error analysis. 
The arguments used to control the $H^1$ norms of $W_\mr^\mu$  and $W_\mi^\mu$ use the Lipschitz assumption, Assumption~\ref{amr1z}.

Another special difficulty associated with pulses, although not a new one, arises from the fact that in order for a decaying profile like $W_\mi(t,x,\theta)$, for example, to have a decaying primitive in $\theta$, $W_\mi$ must have \emph{moment zero}, that is, 
$$\int_{-\infty}^\infty W_\mi(t,x,\theta) \,\mathrm d\theta=0.$$ 
But even if the initial data in the (nonlinear) transport equation for $W_\mi$ is taken to have moment zero, that property is \emph{not} generally inherited by the solution $W_\mi$.   The way around this is to replace $W_\mi$ by a suitable \emph{moment zero approximation} before taking the primitive in $\theta$ to get $U_\mi$.  To define this approximation the Fourier transform in $\theta$ of $W_\mi$, denote it by $\widehat W_\mi(t,x,k)$, is multiplied by a low frequency cutoff function $\chi_\omega(k)$ that vanishes near $k=0$ and is $1$ outside a neighborhood of $0$ of radius $\sim \omega$; see Appendix~\ref{mz}.  This approximation introduces new errors that must be controlled.   In the wavetrain case, where the profiles are periodic with a well-defined mean, it is possible to formulate separate equations for the mean $\underline{u}$ and for $W_\mi$ which have the property  that mean zero initial data for $W_\mi$ yields a solution $W_\mi$ with mean zero.\footnote{Define the mean of the  $2\pi$-periodic function $W_\mi$ to be $\frac{1}{2\pi}\int^{2\pi}_0 W_\mi(t,x,\theta) \,\mathrm d\theta.$}  Thus, moment zero approximations are not needed in the wavetrain case.

\subsection{Organization of the paper} 
In \S \ref{asmr} we describe the classification of boundary points into elliptic, hyperbolic, and glancing sets, and define the class of glancing points of diffractive type $\cG_\md$.  We give a precise statements of the main assumptions and the main result, Theorem \ref{mta}.  In \S \ref{size}  we prove two main estimates on the size of low regularity pulses that are needed for the error analysis, the crude estimate of Proposition \ref{crude} and the more refined estimate of Proposition \ref{a11z}.

In \S \ref{npfs} we construct  a low regularity, large amplitude pulse  propagating in free space.   This construction  does not follow from earlier works (e.g., \cite{ar2}) on geometric optics for high-regularity pulses in free space.   There is no shadow boundary in this problem, so the factor $P_1\phi$ that appears in the profile equation \eqref{pe}  for $W(t,x,\theta)$ is $C^\infty$; derivatives of $P_1\phi$ remain bounded.   There is no need for a truncation argument here, and one can construct $W\in H^1$.

The  proof of the free space result,  Theorem \ref{fnp}, is a good warm-up for the proof of our main result on diffraction, Theorem \ref{mta}, which is proved in \S \ref{diff}.   \S \ref{rphase} shows that Assumption \ref{exv} on the reflected phase $\phi_\mr$ holds in a large class of examples involving grazing points of any order, including the special case treated in Theorem \ref{thm:obstacle}. 
Appendix~\ref{mz} recalls the facts we need about moment-zero approximations, and Appendix~\ref{tools} gathers a few other technical tools for convenient reference.

\begin{nota}
We clarify some notations used in this paper.
\begin{enumerate}
    \item Let $\bN_0 \coloneqq \{0,1,2,\dots\}$.

    \item Write $\overline{\RR_+^n}\coloneqq \overline{\RR_+}\times \RR^{n-1}$.

    \item Let $T>0$ and $\Omega_T=[-T,T]_t\times [0,+\infty)_{x_1}\times \RR^{n-1}_{x'}$, $\tilde\Omega_{T} = [0,T]_t\times \RR_{x_1}\times \RR^{n-1}_{x'}$. We use coordinates $(t,x)=(t,x_1,x')$ for points in $\Omega_{T}$ and $\tilde\Omega_T$. We also use notations $\Omega^\flat_T\coloneqq \Omega_T\cap\{x_1=0\}$ and $\slashed{\Omega}_T\coloneqq [-T,T]_t\times \RR^{n-1}_{x'}$.

    \item Suppose $-T<t_0\leq T$.   Set $\Omega_{[-T,t_0]}\coloneqq [-T,t_0]_t\times [0,\infty)_{x_1}\times \bR^{n-1}_{x'}$, $\Omega^{\flat}_{[-T,t_0]}\coloneqq \Omega_{[-T,t_0]}\cap\{x_1=0\}$ and $\slashed\Omega_{[-T,t_0]}\coloneqq [-T,t_0]\times \RR^{n-1}_{x'}$.

    \item For a function $f=f(t,x)$ where $x=(x_1, x')$, $\nabla f$ always means the gradient with respect to $x$, unless otherwise specified. Moreover, $\slashed\nabla f\coloneqq \nabla_{x'}f$ is the gradient with respect to $x'$.

    \item Many functions in this paper depend on a small parameter $\eps$, and we often suppress $\eps$ in denoting those functions.
If $H$ depends on $\eps$ and we assert that $H\in X$, where $X$ is a function space with norm $\|\cdot\|$, we mean 
$\|H\|\leq C$ with $C$ independent of $\eps$ small.

    \item Suppose $f^\mu$, $g^\mu$ are functions depending on a variable parameter $\mu>0$ and let $\|\cdot\|_1$, $\|\cdot\|_2$ be norms.
The inequality $\|f^\mu\|_1\lesssim \|g^\mu\|_2$ means that there exists $C>0$ independent of $\mu$ such that 
$\|f^\mu\|_1\leq C\|g^\mu\|_2$.  If $C$ depends on $\mu$, we write $\|f^\mu\|_1\lesssim_\mu \|g^\mu\|_2$.
\end{enumerate}
\end{nota}

\section{Assumptions and main result}\label{asmr}

In order to describe and state our main result with a minimum of preparation,  we work now in coordinates $(t,x_1,x')\in \RR\times\RR\times \mathbb{R}^{n}$ and dual coordinates $(\tau,\xi_1,\xi')$ where $t$ is the time variable and $x_1=0$ defines the (noncharacteristic)  boundary.  

\begin{ass}\label{amr1}
 The operator $P$ in \eqref{in1} is a second-order scalar operator $P(t,x,\partial_t, \nabla)$ with coefficients in $C^\infty(\bR^{n+1})$, strictly hyperbolic with respect to $t$,
whose principal symbol has the form 
\begin{align}\label{d1}
p(t,x,\tau,\xi)=\xi_1^2+q(t,x,\tau, \xi'),
\end{align}
where $q(t,x,\cdot,\cdot)$ has signature $(n-1,1)$.   
Moreover, for constants $c\in \RR$, the surfaces $t=c$ are spacelike for $P$; that is 
$p(c,x,1,0_{\RR^n})<0$.   We take the coefficients of $P$ to be constant outside some neighborhood of $0$ in $\bR^{n+1}$.  
\end{ass}

\begin{ass}\label{amr1z}
The function $f(t,x,\zeta)$  in \eqref{in1} is real-valued, lies in $C^\infty(\mathbb{R}^{n+1}_{t,x}\times \mathbb{R}^{n+2}_\zeta)$, and satisfies $f(t,x,0)=0$.  Moreover $f$, $\partial_tf$ and $\nabla f$ are Lipschitz in $\zeta$.   That is, there exists $M>0$ such that 
\begin{align}\label{lips}
|f(t,x,\zeta_1)-f(t,x,\zeta_2)|\leq M|\zeta_1-\zeta_2| \ \text{ for all } \ (t,x) \ \text{ and } \ \zeta_1, \zeta_2,
\end{align}
with a similar estimate for each of $\partial_t f$ and $\nabla f$.  
\end{ass}

\begin{rem}\label{localc}
(i) Assumption \ref{amr1} implies that the boundary $x_1=0$ is noncharacteristic and in fact timelike for $P$; that is
\begin{align*}%\label{amr2}
p(t,0,x',0,1,0_{\RR^{n-1}})>0.
\end{align*}

(ii) Our analysis will be local near $0\in \bR^{n+1}$, so it is no restriction to assume that the coefficients of $P$ are constant outside a neighborhood of $0$. 

(iii) Clearly, Assumption \ref{amr1z} implies $|f_\zeta|\leq M$ for all $(t,x,\zeta)$.

(iv) The problem \eqref{in1} can be stated in coordinate-invariant form where $P$ is a second order differential operator with smooth coefficients on $\bR^{n+1}$, the boundary is a smooth timelike hypersurface defined by  $\beta=0$, the initial surface is a smooth spacelike hypersurface defined by $\alpha=0$, and the two surfaces intersect at $0\in\bR^{n+1}$.   We show in \cite[\S 3]{wangwil} how to choose local ``standard form" coordinates near $0$ so that the principal symbol of $p$ takes the form~\eqref{d1}.
In this formulation $\Omega_T$ is replaced by $\{m\in\bR^{n+1} \mid \beta(m)>0, -T\leq \alpha(m)\leq T\}$.
To shorten this paper we have decided to work throughout in these standard form coordinates.

\end{rem}

Before stating the other assumptions we describe the classification of boundary points.   
Let $\bR^{n+1}_+\coloneqq \{(t,x)\in\bR^{n+1} \mid x_1>0\}$.  We will write points in $T^*\bR^{n+1}$, $\partial T^*\bR^{n+1}$, and $T^*(\partial \bR^{n+1}_+)$ as $(t,x,\tau,\xi)=(t, x_1, x',\tau, \xi_1,\xi')$, $(t,0,x',\tau, \xi_1,\xi')$, and $(t,x',\tau,\xi')$ respectively.  Let 
\begin{align*}
 i^*:\partial T^*\bR^{n+1}\to T^*(\partial \bR^{n+1}_+), \ \ (t,0,x',\tau,\xi_1,\xi') \mapsto (t,x', \tau, \xi').
\end{align*}
We recall from \cite{melsjo1978}  the decomposition 
\begin{align*}%\label{q1}
T^*\partial \bR^{n+1}\setminus 0=E\cup H\cup G
\end{align*}
into \emph{elliptic}, \emph{hyperbolic}, and \emph{glancing} sets.  
If $\sigma=(t,x',\tau,\xi')\in T^*\partial \bR^{n+1}\setminus 0$, we say that $\sigma$ belongs to $E$, $H$, or $G$ if the number of elements in  $(i^*)^{-1}(\sigma)\cap p^{-1}(0)$ is zero, two, or one respectively; or equivalently, if $q(t,0,x', \tau, \xi')$ is positive, negative, or zero respectively.   The sets $E$ and $H$ are conic open subsets of $T^*\partial \bR^{n+1}\setminus 0$, and $G$ is a closed conic hypersurface in  $T^*\partial \bR^{n+1}\setminus 0$.

Recall that the Hamiltonian vector field associated to $p=\xi_1^2+q(t,x,\tau,\xi')$ is the vector field on $T^*\bR^{n+1}$ given by 
\begin{align*}
H_p = \partial_{\tau,\xi}p \cdot\partial_{t,x}-\partial_{t,x}p \cdot\partial_{\tau,\xi}.
\end{align*}
Integral curves of $H_p$ are called \emph{bicharacteristics}.

\begin{defn}\label{q2}
Let $\sigma=(t,x',\tau,\xi')\in G$ and suppose $(i^*)^{-1}(\sigma)\cap p^{-1}(0)=\{\rho\}$, where $\rho\in T^*_{(t,0,x')}\bR^{n+1}$.   
We say $\sigma\in G^\ell$, the glancing set of order at least $\ell\geq 2$, if 
\begin{align}\label{q2z}
(H_p^j x_1)(\rho)=0 \ \text{ for } \ 0\leq j<\ell.   
\end{align}
Thus, $G=G^2\supset G^3\supset \dots\supset G^\infty$.  

We say $\sigma\in G^{\ell}\setminus G^{\ell+1}$, the set of glancing points of exact order $\ell$, if $\sigma\in G^\ell$ and $(H^\ell_px_1)(\rho)\neq 0$.     We will study the transport of pulses near points $\sigma\in G^{2k}\setminus G^{2k+1}$, $k\geq 1$,  such that $(H^{2k}_p x_1)(\rho)>0$.    When $k=1$, such a point $\sigma$ is a  classical \emph{diffractive point} as studied in \cite{mel1975} or \cite{cheverry1996}.   When $k\geq 1$ we refer to $\sigma$ as a \emph{diffractive point of order $2k$}, and we write\footnote{Sometimes, for example earlier in this introduction, we refer to a diffractive point of order $2k$ as a grazing point of order $k$.}
\begin{align}\label{q3}
\sigma\in G^{2k}_\md\setminus G^{2k+1} \ \Leftrightarrow \
\begin{cases}
& (i^*)^{-1}(\sigma)\cap p^{-1}(0)=\{\rho\}, \\ 
& (H_p^j x_1)(\rho)=0 \ \text{ for } \ 0\leq j<2k, \ \text{ and } \ (H^{2k}_p x_1)(\rho)>0.
\end{cases}
\end{align}  
\end{defn}

\begin{rem}
(i) If $\sigma\in G^{2k}_\md\setminus G^{2k+1}$, let $\gamma(s)$ denote the  bicharacteristic of $p$ such that $\gamma(0)=\rho$ for $\rho$ as in \eqref{q2z}.    Then $\gamma$ is tangent to $\partial T^*\bR^{n+1}$ at $\rho$ and lies $T^*\{(t,x) \mid x_1>0\}$ for small $s\neq 0$.

(ii) \emph{Gliding points} of order $2k$, $\sigma\in G^{2k}_{\mathrm g}\setminus G^{2k+1}$, are defined as in \eqref{q3} with the single change $(H^{2k}_p x_1)(\rho)<0$.   If $\sigma\in G^\ell\setminus G^{\ell+1}$ for some odd $\ell$, we call $\sigma$ an \emph{inflection point} of order $\ell$.  
\end{rem}

\begin{defn}[Diffractive points of order $\infty$]\label{q2a}
Let $\sigma\in G^\infty$ and suppose $(i^*)^{-1}(\sigma)\cap p^{-1}(0)=\{\rho\}$.
We say that $\sigma$ is a \emph{diffractive point of order $\infty$} and write $\sigma\in G^\infty_\md$ if the 
bicharacteristic $\gamma(s)$ of $p$ such that $\gamma(0)=\rho$ lies in $T^*\{(t,x) \mid x_1>0\}$ for small $s\neq 0$.
\end{defn}

\begin{defn}[Glancing points of diffractive type]\label{q2b}
We denote by 
\begin{align}\label{gpd}
\mathcal{G}_\md \coloneqq \bigcup_{k=1}^\infty \left(G^{2k}_\md\setminus G^{2k+1}\right)\cup G^\infty_\md
\end{align}
the set of \emph{glancing points of diffractive type}.
\end{defn}

The data $g^\eps$, $h^\eps$ described in \eqref{a1aw} and \eqref{a1cz}
  supplies the incoming pulses.    
The surfaces of constant incoming phase are the spacetime surfaces $\phi_\mi(t,x)=\mathrm{constant}$.  

\begin{ass}\label{amr3}
 The function $\phi_\mi$, called the \emph{incoming phase}, is a \emph{given} $C^\infty$ function that satisfies the \emph{eikonal} equation 
\begin{align*}
p(t,x,\partial_t\phi_\mi, \nabla \phi_\mi)=0 \ \text{ on } \ U \ \text{ with } \ \partial_t\phi_\mi< 0 \ \text{ on } \ U. 
\end{align*}
where $U$ is some $\mathbb{R}^{n+1}$-neighborhood of $0$ that we take to be an open ball centered at $0$.\footnote{Solutions of the eikonal equation can be constructed by the method of characteristics; see, for example, \cite{williams2022}.   The ball $U$ may need to be shrunk a finite number of times in this paper.}   

\end{ass}

We will see that the pulses are transported along characteristics of $p$ associated not only to $\phi_\mi$ but also to an associated \emph{reflected phase} $\phi_\mr$, whose construction is described below.
The characteristics associated to $\phi_\bullet$, $\bullet=\mi,\mr$, are integral curves of the \emph{characteristic vector field of $P$ associated to $\phi_\bullet$}:
\begin{align}\label{ezb}
T_{\phi_\bullet}\coloneqq \left(2\xi_1\partial_{x_1}+\partial_{\tau,\xi'}q(t,x,\tau,\xi')\cdot\partial_{t,x'}\right)_{|(\tau,\xi)=(\partial_t\phi_\bullet, \nabla\phi_\bullet)}.
\end{align}    
The initial pulse profile $V_0(x,\theta)$ in \eqref{a1cz} is  chosen to have compact support in $x_1>0$,  and the incoming phase $\phi_\mi$ is chosen so that {some} of the characteristics of $\phi_\mi$ emerging from points in the $x$-support of  $V_0$   graze the boundary $x_1=0$ to some finite or possibly infinite order.  Each such grazing characteristic is tangent to $x_1=0$ at a single spacetime point, and nearby points on the characteristic lie in $x_1>0$.   The order of tangency is what we mean by the order of ``grazing".  We arrange so that the origin $0\in\Omega_T$ is such a point of tangency.   Near each grazing characteristic there are transversal incoming characteristics that reflect off the boundary; see Figure \ref{rays}. 

\begin{figure}[t]
\begin{center}
    \includegraphics[scale=0.5]{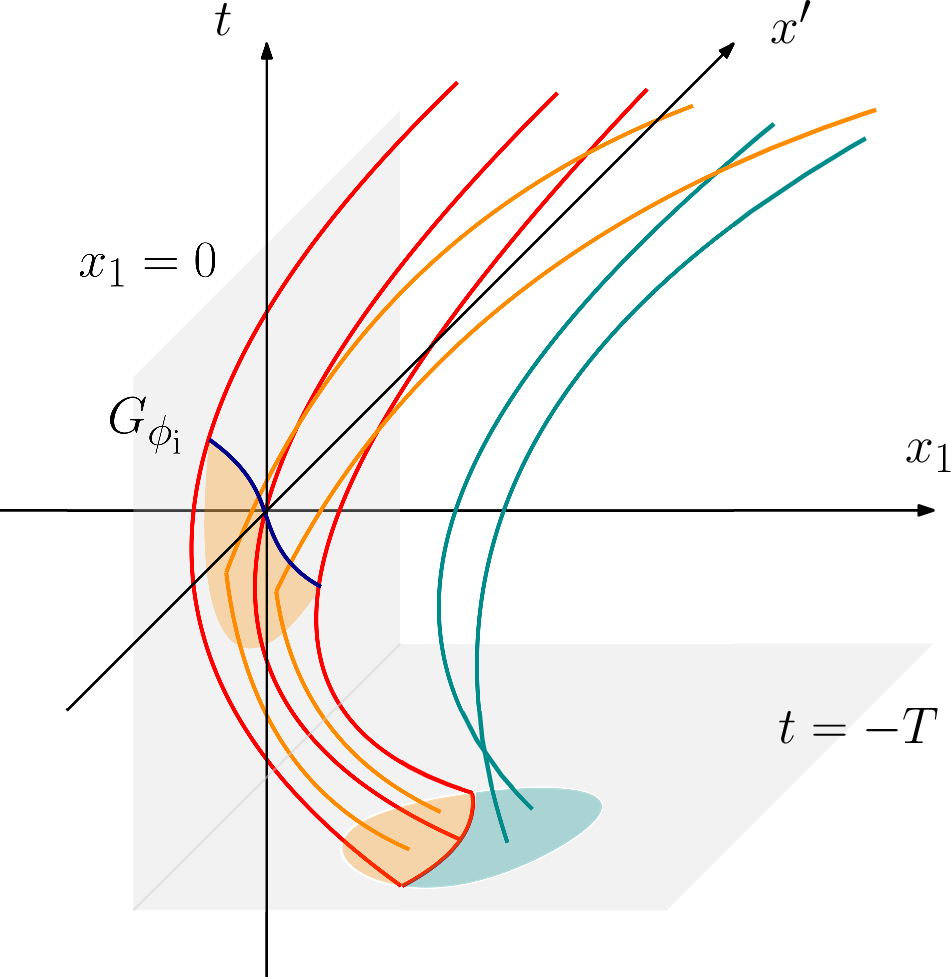}
    \caption{Characteristics associated to $\phi_\mi$ and $\phi_\mr$. The yellow curves reflect off the boundary, the red curves graze the boundary, and the green curves do not touch the boundary. The dark curve on $\{x_1=0\}$ is the grazing set $G_{\phi_\mi}$.}
    \label{rays}
\end{center}
\end{figure}

Let $\phi_0(t,x')=\phi_\mi(t,0,x')$.   Observe that the eikonal equation implies
\begin{align*}
\mathrm{Graph}((\partial_t\phi_0, \slashed{\nabla}\phi_0))\subset H\cup G.
\end{align*}

The following assumption means that a characteristic line of $\phi_\mi$ grazes the boundary $\{x_1=0\}$  at $0$ to some finite or possibly infinite order:
\begin{ass}\label{A1}
With $\mathcal{G}_\md$ as in Definition \ref{gpd}, we have $\underline{\sigma}\coloneqq (0,0,\partial_t\phi_0(0), \slashed{\nabla}\phi_0(0))\in \mathcal{G}_\md$.
Suppose also that $\phi_\mi(0)=0$.
\end{ass}
The condition $\phi_\mi(0)=0$ is needed to insure that the incoming pulse is not negligible near $0$.  

The main theorem is stated in terms of incoming and reflected profiles, $U_\mi$ and $U_\mr$, that describe the transport of oscillations.   Each function $U_\bullet$ for $\bullet=\mr,\mi$  is the unique moment-zero  primitive in $\theta$ of a moment-zero function $W_\bullet(t,x,\theta)\in L^2(\Omega_T\times \mathbb{R})$ that is constructed to satisfy the transport equations \eqref{nl3} and \eqref{nl4}.

We proceed to define particular subsets of $\Omega_T$, $J_\mr$ and $J_\mi$, that {contain} the supports of $W_\mr$ and $W_\mi$. 
From \eqref{ezb} we see that characteristics of $\phi_\mi$ are tangent to $x_1=0$ precisely at points of the \emph{grazing set}
\begin{align*}
G_{\phi_\mi}\coloneqq \{(t,0,x')\in U \mid \partial_{x_1}\phi_\mi(t,0,x')=0\}.
\end{align*}
The first of our two main geometric assumptions is the following \emph{grazing set assumption}.

\begin{ass}[Regularity of the grazing set]\label{A02}
The set $G_{\phi_\mi}$ is a codimension two $C^1$ submanifold of $\mathbb{R}^{n+1}$ near $0\in G_{\phi_\mi}$.   That is, there exists a $C^1$ function $\Xi(t,x)$ defined near $0$ such that $(\partial_t\Xi,\nabla\Xi)(0,0) \neq 0$, 
\begin{align*}%\label{ezd}
G_{\phi_\mi}=\{(t,x)\in U \mid x_1=0 \ \text{ and }\ \Xi=0\},
\end{align*}
and  the vector field $T_{\phi_\mi}$ is transverse to the $n$-dimensional hypersurface $\{\Xi=0\}$ at $0$.
Moreover, we require
\begin{align}\label{A02a}
\{(t,x',\partial_t\phi_0(t,x'), \slashed{\nabla}\phi_0(t,x')) \mid (t,0,x')\in G_{\phi_\mi}\}\subset \cG_\md.
\end{align}
\end{ass}

\begin{rem}
(i) When the origin is a point of \emph{first}-order tangency,  it was shown in \cite{cheverry1996} that Assumption \ref{A02} always holds and that $\Xi$ can be taken to be a $C^\infty$ function. In fact, one can take $\Xi(t,x)=\partial_{x_1}\phi_\mi(t,0,x')$ in that case.\footnote{Examples in \cite{wangwil} show that this choice of $\Xi$ does not work in cases of higher order tangency.}   When the origin is a point of higher than first-order tangency, verifying this assumption can be difficult, and sometimes it fails to hold!  In the papers \cite{wangwil, wangwil2} we present large classes of examples where this assumption is satisfied for all orders of tangency.    However, we also show in \cite{wangwil2} that this assumption sometimes fails even in the case of the basic motivating example, Example \ref{fundex}, and even when the incoming phase is linear or nonlinear with  spherical level sets.   

(ii)  In Example \ref{fundex}, the condition \eqref{A02a} follows easily from the convexity of the obstacle $\cO$.

(iii) An important consequence of Assumption \ref{A02} for our purposes is that the flowout of $G_{\phi_\mi}$ along integral curves of $T_{\phi_\mi}$ is a $C^1$ hypersurface in $\mathbb{R}^{n+1}$ near $0$.
\end{rem}

Let $\mathrm{SB}=\mathrm{SB}_+\cup \mathrm{SB}_-$ be the $C^1$ hypersurface in $\mathbb{R}^{n+1}$  which is the flowout   of $G_{\phi_\mi}$ along characteristics of $\phi_\mi$.   More precisely, $\mathrm{SB}$ is the union of the forward and backward flowouts of $G_{\phi_\mi}$, $\mathrm{SB}_\pm$ respectively,  along integral curves of $T_{\phi_\mi}$.\footnote{By the ``forward flowout" we mean the flowout along integral curves for which $t$ increases as the curve parameter increases.}  We call $\mathrm{SB}_+$ the \emph{shadow boundary}.

\begin{figure}[t]
\begin{center}
    \includegraphics[scale=0.45]{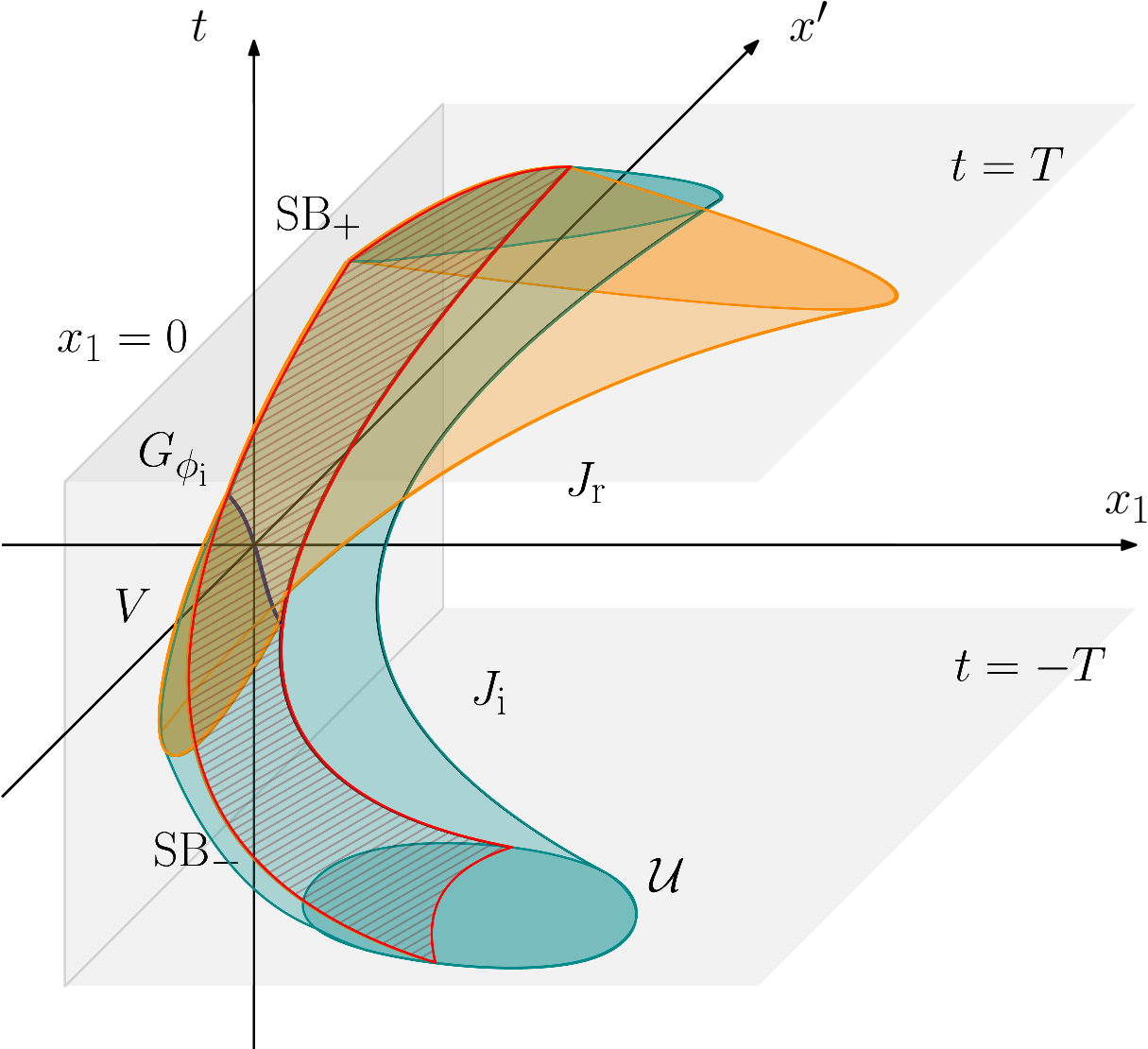}
    \caption{Green domain: forward flowout in $\{x_1\geq 0\}$ of the characteristic vector field $T_{\phi_\mi}$ associated to the incoming phase $\phi_\mi$. Yellow domain: forward flowout of the characteristic vector field $T_{\phi_\mr}$ associated to the reflected phase $\phi_\mr$. Dark curve on the boundary $\{x_1=0\}$: the grazing set $G_{\phi_\mi}$. Red surfaces: $\mathrm{SB}_{\pm}$, forward and backward flowouts of the grazing sets along characteristics of $T_{\phi_\mi}$. }
    \label{shadow}
\end{center}
\end{figure}

With $U_\mi$ as in \eqref{in2}, we take initial data for the incoming profile $W_\mi=\partial_\theta U_\mi$: 
\begin{align*}%\label{ezf}
W_\mi(-T,x,\theta)=\frac{V_1(x,\theta)}{\partial_t\phi_\mi(-T,x)}\eqqcolon a(x,\theta).
\end{align*}
We are interested in the behavior of oscillations transported by rays that reflect off and graze the boundary near $0$, so it is no restriction to assume that the $x$-support of $a$ is small and located near $\mathrm{SB}_-\cap \{t=-T\}.$ 
For $T>0$ small this allows us to choose an $n$-dimensional closed ball $\cU$ such that
\begin{align}\label{ezg}
\begin{split}
\cU\subset \{(-T,x) \mid x_1>0\} \ \text{ and } \ x\text{-support of } \ a\subset \mathring{\cU};
\end{split}
\end{align}
see Figure \ref{shadow}.
For points $(-T,y)\in \cU$ and for $s\geq 0$ let 
\begin{align*}
(t,x)=Z_\mi(s,y), \ \text{ where } \ Z_\mi(0,y)=(-T,y),
\end{align*}
denote the forward flow map determined by $T_{\phi_\mi}$.   We refer to $Z_\mi$ as the 
\emph{the incoming flow map}; it is a $C^\infty$ diffeomorphism onto its range, since $T_{\phi_\mi}$ is transverse to surfaces $t=c$
for $|c|$ small.   Moreover the range of $Z_\mi$ contains an $\mathbb{R}^{n+1}$-neighborhood of $0$.

Now define the flowout of $\cU$ under $T_{\phi_\mi}$ in $\Omega_T$ to be
\begin{align}\label{e2a}
J_\mi=\{Z_\mi(s,y) \mid 0\leq s\leq s(y), \ (-T,y)\in \cU\} \eqqcolon Z_\mi(\cD^\mi)\subset \Omega_T,
\end{align}
where $s(y)$ is the value of $s$ for which the $x$-component of $Z_i(s,y)$ is $0$ when the integral curve leaves $\{x_1\geq 0\}$, and is the value of $s$ for which the $t$-component of $Z_\mi(s,y)$ is $T$ when the integral curve remains inside $\{x_1\geq 0\}$.

Letting $V\coloneqq J_\mi\cap \{x_1=0\}$, we proceed to outline the construction of $\phi_\mr$. 
Let $\gamma_\mr(s;t_0,0,x_0')$ be the bicharacteristic of $p$ passing through $(t_0,0,x_0')\in V$ at $s=0$ with initial conditions
\begin{align*}
\gamma_\mr(t_0, 0,x_0')=(t_0,0, x_0', \partial_t\phi_\mi(t_0,0,x_0'), -\partial_{x_1}\phi_\mi(t_0,0,x_0'), \slashed{\nabla}\phi_{\mi}(t_0,0,x_0') ).
\end{align*}
Write  
\begin{align}\label{e4u}
\gamma_\mr(s;t_0,0,x_0')=(t,x,\tau,\xi)(s,t_0,x_0'),
\end{align}
 and for $(t_0,0,x_0')\in V$ and $s\geq 0$ define the \emph{reflected flow map}
\begin{align}\label{e3}
Z_\mr(s,t_0,x_0') \coloneqq (t,x)(s; t_0, x_0').
\end{align}
Parallel to $J_\mi$  we define the flowout of $V$
\begin{align}\label{e4}
\begin{split}
&J_\mr=\{(t,x)=Z_\mr(s,t_0,x_0') \mid  0\leq s\leq s(t_0,x_0'), \ (t_0,0,x_0')\in V \}\eqqcolon Z_\mr(\cD^\mr)\subset \Omega_T,
\end{split}
\end{align}
where $s(t_0, x_0')$ is the value of $s$ for which the $t$-component of $Z_\mr(s,t_0, x_0')$ is $T$;  see Figure \ref{shadow}.

We now make our second crucial geometric assumption.

\begin{ass}[Reflected flow map $Z_\mr$]\label{A03}
Let 
$$V_\mr \coloneqq \{(t, x') \mid (t,0,x')\in V\} \ \text{ and } \ \mathring{V}_\mr=\{(t,x') \mid (t,0,x')\in V\setminus G_{\phi_\mi}\}$$ for $V$ as above.    The set $\cU$ as in \eqref{ezg} (which determines $V$) and  $s_0>0$ can be chosen so that the map
\begin{align*}
Z_\mr:[0,s_0)\times V_\mr\to \Omega_T
\end{align*}
is a homeomorphism onto its range $J_\mr$,  and so that 
\begin{align*}
Z_\mr:[0,s_0)\times \mathring{V}_\mr\to \Omega_T
\end{align*}
is a  $C^\infty$ diffeomorphism onto its range $\mathring{J_\mr}$.
\end{ass}

\begin{rem}
\textup{
In \cite{wangwil} we showed that Assumption \ref{A03} 
is satisfied by the problem \eqref{in1} for general incoming phases $\phi_\mi$,  when the origin is a point of first-order tangency.\footnote{A result like this was formulated in \cite{cheverry1996}, but the proof there applied to a modified map obtained by truncating the Taylor series of $Z_\mr$ at order two.}   As with Assumption \ref{A02}, when the origin is a point of higher than first-order tangency, verifying this assumption can be difficult.  In \cite{wangwil} we showed that Assumption \ref{A03} always holds in the setting of Example \ref{fundex} 
when the incoming phase $\phi_\mi$ is linear.   In \cite{wangwil2} we verified the assumption in Example \ref{fundex} for general incoming phases with convex level sets, which includes the case of spherical level sets.  
These proofs  apply to all orders of grazing and, in fact, do not depend on Assumption \ref{A02}. } 
\end{rem}

Assumption \ref{A03} allows us to define $Z_\mr^{-1}: J_\mr\to [0,s_0)\times \mathring{V}_\mr$.    By the method of characteristics \cite{williams2022} we obtain the reflected phase $\phi_\mr$ as the solution of $p(t,x,\partial_t\phi_\mr, \nabla\phi_\mr)=0$ on  $\mathring{J_\mr}$ which for $(t,x')\in V_\mr$ satisfies
\begin{align*}%\label{e4z}
\begin{split}
\phi_\mr(t,0,x') & =\phi_\mi(t,0,x'), \\
\partial_{x_1}\phi_\mr(t,0,x') & =-\partial_{x_1}\phi_\mi(t,0,x').
\end{split}
\end{align*}
Define $\mathrm j: [0, s_0)\times V_\mr\to V$ such that $\mathrm j(s,t_0, x_0')\coloneqq (t_0,0,x_0')$.
%If $(s,t_0,x_0')=Z_\mr^{-1}(t,x)$, define $Z_r^{-1}(x,y,t)'\coloneqq (0,y',t')$.  
The method of characteristics yields the formulas\footnote{Here $(\tau,\xi)$ is as in \eqref{e4u}.}
\begin{subequations}\label{e4y}
    \begin{align}
        \label{e4ya}
        \phi_\mr(t,x) & =\phi_\mi \left((\mathrm j\circ Z_\mr^{-1})(t,x)\right), \\
        \label{e4yb}
        (\partial_t\phi_\mr, \nabla\phi_\mr)(t,x) & =(\tau, \xi)\left(Z_\mr^{-1}(t,x)\right),
    \end{align}
\end{subequations}
which show that 
\begin{align*}%\label{e4x}
\phi_\mr\in C^\infty\left(\mathring{J_\mr}\right) \ \text{ but we have only } \ \phi_\mr\in C^1\left({J_\mr}\right).
\end{align*}
In particular,  $\phi_\mr$  is only $C^1$ up to the shadow boundary $\mathrm{SB}_+$.    Moreover, from \eqref{e4u}, \eqref{e3}, and \eqref{e4yb}, we see that for any given $(t,x')\in V_\mr$,  the curve $s\mapsto Z_\mr(s,t,x')$ is an integral curve of the vector field $T_{\phi_\mr}$ defined in \eqref{ezb}.   We call that curve a forward characteristic of $P$ associated to $\phi_\mr$.  

\begin{ass}\label{exv}
The reflected phase $\phi_\mr$ satisfies $\partial_{x_1}\phi_\mr\neq 0$ at all points of $\mathring{J_\mr}$.\footnote{Recall that $\partial_{x_1}\phi_\mr=-\partial_{x_1}\phi_\mi$ on $J_\mr\cap \{x_1=0\}$, so $\partial_{x_1}\phi_\mr\neq 0$ on $V\setminus G_{\phi_\mi}$.}
\end{ass}

\begin{rem}
(i)  Assumption \ref{exv} always holds when the origin is a  first-order grazing point.   
This follows easily from \eqref{e4yb} and the fact that $-\partial_{x_1} q(0,\underline \sigma)>0$ for $q$ as in \eqref{d1}.   In cases of higher order grazing we have $-\partial_{x_1} q(0,\underline \sigma)=0$, so it is harder to check this assumption.   
In \S \ref{rphase} we  prove  that Assumption \ref{exv} holds for grazing of any order in the setting of our motivating example, Example \ref{fundex}, that is, whenever $P$ is the wave operator acting in the exterior of a smooth convex obstacle.   Indeed, the validity of the assumption in such cases is strongly suggested by the ``defocusing" nature of reflected rays illustrated  in Figures 5 and 6 of \cite{wangwil}.

(ii) It was noticed in \cite{cheverry1996} in the case of first-order grazing that the inverse of $Z_\mr$ becomes singular near the grazing set; the jacobian determinant of $Z_\mr^{-1}$ blows up roughly like $1/(\text{distance to } G_{\phi_\mi})$.   Moreover,
$\phi_\mr$ fails to be $C^2$ up to the shadow boundary.  
In \cite{wangwil} we observed that the singularity of the  jacobian determinant of $Z_\mr^{-1}$ worsens and the blow-up becomes more complicated as the order of grazing increases.
\end{rem}

Here is our main result.
\begin{theo}\label{mta}
Consider the problem \eqref{in1} under the structural Assumptions \ref{amr1}, \ref{amr1z}, Assumptions \ref{amr3},  \ref{A1}, and \ref{exv} on the incoming phase $\phi_\mi$,  and  the grazing set and reflected flow map Assumptions \ref{A02} and \ref{A03}.
Suppose that $a=(W_\mi)_{|t=-T}$ lies in $H^1(\{t=-T\})$ and has compact $(x,\theta)$-support satisfying \eqref{ezg}.\footnote{Thus, $W_\mi$ has support near an incoming characteristic that  grazes the boundary to possibly high order at $0\in G_{\phi_\mi}$.}      
Then if $T>0$ is small enough, the exact solution $u^\eps\in H^1(\Omega_T)$ to \eqref{in1} satisfies 
 \begin{align}\label{mtb}
u^\eps(t,x)_{|\Omega_T}\sim_{H^1}\sqrt{\eps} U_\mr\left(t,x,\frac{\phi_\mr(t,x)}{\eps}\right)+\sqrt{\eps} U_\mi\left(t,x,\frac{\phi_\mi(t,x)}{\eps}\right).
\end{align}
Here $U_\bullet(t,x,\theta)$ for $\bullet=\mr, \mi$  is given by $U_\bullet(t,x,\theta)=\int^\theta_{-\infty}W_\bullet(t,x,s) \,\mathrm ds$
and the functions 
$$W_\mr\in L^2(\Omega_T\times \mathbb{R}),\;  W_\mi\in  L^2(\Omega_T\times \mathbb{R})$$
  are constructed to satisfy the profile equations \eqref{nl3}, \eqref{nl4}.  In particular, $W_\bullet$ has support in $J_\bullet$ for $\bullet=\mr, \mi$.   The meaning of $\sim_{H^1}$ in \eqref{mtb} is given in Definition \ref{meaninga}.
\end{theo}

The reader may have noticed that there are problems with making sense of $U_\mr(t,x,\phi_\mr/\eps)$ as an element of $H^1(\Omega_T)$, even for $\eps$ fixed; indeed, this is the restriction of $U_\mr(t,x,\theta)$ to a set of measure zero in 
$\Omega_T\times \bR_\theta$.  
We therefore use the following definition, which is similar to ones used in \cite{jmr1995tams, cheverry1996, jmr1996cpam, jmr2000mams, wangwil} for problem involving wavetrains, to give an asymptotic meaning to \eqref{mtb}.

\begin{defn}\label{meaninga}
The statement \eqref{mtb} means that for any sequence $\delta_\ell\searrow 0$ there exist $C^\infty$ profiles $W_\bullet^\ell(t,x,\theta)$, $\bullet=\mr, \mi$, with compact supports in $\Omega_T\times \bR_\theta$ and $\eps_\ell>0$ such that 
\begin{align}\label{nl10}
\|W_\mr-W^\ell_\mr\|_{L^2(\Omega_T\times\bR_\theta)}+\|W_\mi-W^\ell_\mi\|_{L^2(\Omega_T\times\bR_\theta)} <\delta_\ell,\;\;\partial_\theta (U_\mr,U_\mi)=(W_\mr,W_\mi),
\end{align}
and
\begin{align}\label{nl11}
\left\|u^\eps-\sqrt{\eps}\left(U^\ell_{\mr,\eps^d}\left(t,x,\frac{\phi_\mr}{\eps}\right)+U^\ell_{\mi,\eps^d}\left(t,x,\frac{\phi_\mi}{\eps}\right)\right)\right\|_{H^1(\Omega_T)}<\delta_\ell \ \text{ for } \ \eps\in (0,\eps_\ell].
\end{align}
Here for $\bullet=\mr,\mi$, we let $W^\ell_{\bullet,\eps^d}(t,x,\theta)$ be a moment zero approximation\footnote{Moment zero approximations are defined in Appendix \ref{mz}. In the notation of that appendix we take $\omega=\eps^d$ here.} of $W^\ell_\bullet(t,x,\theta)$ for $0<d<1/2$ and 
$U^\ell_{\bullet, \eps^d}(t,x,\theta)$ is the unique moment zero primitive in $\theta$ of $W^\ell_{\bullet,\eps^d}(t,x,\theta)$.   
Both $W^\ell_{\bullet,\eps^d}$ and $U^\ell_{\bullet, \eps^d}$ are 
$C^\infty$ with compact support in $(t,x)$ and rapidly decaying in $\theta$.
\end{defn}

\section{The size of large amplitude pulses in \texorpdfstring{$L^2$}{l2}.}\label{size}

In this section we prove the basic estimates that allow us to compare the size of a pulse profile like $W(t,x,\theta)$ to the size of the 
pulse given by $W(t,x,\phi/\eps)$.   These estimates are used repeatedly in \S \ref{fnpz} and \S \ref{diff}.
Instead of $\Omega_T$ as in previous sections, we work on $\tilde \Omega_T\coloneqq [0,T]\times \RR^n$ in this section. We still use $(t,x)$ as coordinates for points in $\tilde\Omega_T$ and $(\tau,\xi)$ for the dual coordinates.

\begin{prop}\label{crude}
Let $a(t,x,\theta)\in L^2(\tilde\Omega_T,H^1(\mathbb{R}_\theta))$ and $\phi\in C^1(\tilde \Omega_T; \RR)$.  There exists $C>0$ such that for $\eps\in (0,1]$ we have 
\begin{align*}%\label{a7z}
\|a(t,x,\phi/\eps)\|_{L^2(\tilde\Omega_T)}\leq C\|a(t,x,\theta)\|_{L^2(\tilde\Omega_T,H^1(\bR))}.
\end{align*}
\end{prop}

\begin{proof}
 Letting $\hat a$ be the Fourier transform of $a$ in $\theta$ 
 we have\footnote{Here $\langle k\rangle \coloneqq (1+k^2)^{1/2}$.}
\begin{align*}%\label{a7y}
\begin{split}
\|a(t,x,\phi/\eps)\|_{L^2(\tilde\Omega_T)} & = \left\|\frac{1}{2\pi}\int \hat a(t,x,k)e^{ik\phi/\eps}\,\mathrm d k\right\|_{L^2(\tilde\Omega_T)}\leq \frac{1}{2\pi}\int \|\hat a(t,x,k)\|_{L^2(\tilde\Omega_T)} \,\mathrm d k\\
&\lesssim \left\|\langle k\rangle^{-1}\right\|_{L^2(\bR_k)}\left \| \langle k\rangle \|\hat a(t,x,k)\|_{L^2(\tilde\Omega_T)}\right\|_{L^2(\bR_k)}\lesssim \|a(t,x,\theta)\|_{L^2(\tilde\Omega_T,H^1(\bR))}.
\end{split}
\end{align*}
This is the desired estimate.
\end{proof}

We will also need the following more precise estimates.  If $\phi(t,x)$ is a $C^1$ solution of the eikonal equation $p(t,x,\partial_t\phi, \nabla \phi)=0$ we define the vector field
\begin{align*}%\label{a11u}
T_\phi(t,x)=\partial_{\tau,\xi} p(t,x,\partial_t\phi(t,x), \nabla\phi(x,t))\cdot \partial_{t,x}.
\end{align*}

Let $K\subset \tilde{\Omega}_T$ be the  flowout under a vector field $T_\phi$, defined as in \eqref{ezb},  of a compact ball in the initial surface $t=0$.  
Here we assume $\phi$ is a $C^1$ characteristic phase for which there exists an open set $\mathcal O$ such that $\overline{K}\subset \mathcal O\subset \tilde \Omega_T$ and 
\[ F: \mathcal O \to F(\mathcal O), \ (\phi(t,x),x)\mapsto (t,x) \]
is a $C^1$ diffeomorphism. 
We will use the change of variables 
\begin{align}\label{a11h}
(s,x)=F^{-1}(t,x)\coloneqq (\phi(t,x),x), \quad (t,x)=F(s,x)=(t(s,x),x).
\end{align}

\begin{prop}\label{a11z} 
Assume $W(t,x,\theta)\in H^1(\tilde\Omega_T\times\bR)$ has  compact $(t,x,\theta)$-support with $(t,x)$ support contained in the compact set $\overline{K}$ as above.
Then
\begin{enumerate}
    \item[(i)] There exists a constant $C>0$ such that for $\eps\in (0,1]$:
\begin{align}\label{a11}
\left\|\epsilon^{-\frac12}W(t,x,\phi/\eps)\right\|_{L^2(\tilde\Omega_T)}\leq C\|W\|_{H^1(\tilde\Omega_T\times \bR)}.
\end{align}

\item[(ii)]  If in addition to the above assumptions $W(t,x,\theta)$ is continuous, then
\begin{align}\label{a11x}
\lim_{\eps\to 0+}\left\|\epsilon^{-\frac12}W(t,x,\phi/\eps)\right\|^2_{L^2(\tilde\Omega_T)}=\int |W(t(0,x),x,\theta)|^2|\det F'(0,x)| \,\mathrm dx \mathrm d\theta.
\end{align}
In particular, $\left\|\epsilon^{-\frac12} W(t,x,\phi/\eps)\right\|^2_{L^2(\tilde\Omega_T)}$ is generally of size $\sim 1$ for $\eps$ small and is determined by the restriction of $W(t,x,\theta)$ to the hypersurface $\phi^{-1}(0)\times\bR_\theta$, that is, by $W(t,x,\theta)_{|(t,x)\in \phi^{-1}(0)}$. 
\end{enumerate}
\end{prop}

\begin{proof}

\textbf{1. }
We compute
\begin{align}\label{a12}
\begin{split}
\frac{1}{\eps}\int |W(t,x,\phi/\eps)|^2 \,\mathrm dx \mathrm dt
& =\frac{1}{\eps}\int \left|W\left(t(s,x),x,\frac{s}{\eps}\right)\right|^2|\det F'(s,x)| \,\mathrm dx \mathrm ds\\
& =\int \left|W\left(t(\eps\theta,x),x,\theta\right)\right|^2 |\det F'(\eps\theta,x)| \,\mathrm dx \mathrm d\theta \eqqcolon A_\eps, 
\end{split}
\end{align}
where $\theta=s/\epsilon$.

\noindent To obtain \eqref{a11} we estimate
\begin{align*}%\label{a15}
\begin{split}
&A_\eps
\leq C\int\left|W\left(t(\eps\theta,x),x,\theta\right)\right|^2 \,\mathrm dx \mathrm d\theta\leq C\|W(t,x,\theta)\|^2_{H^1(\tilde\Omega_T\times \bR)}.
\end{split}
\end{align*}
For the last inequality we used the simple trace estimate of Proposition \ref{trace}.  Observe that for each $\eps$, $W\left(t(\eps\theta,x),x,\theta\right)$ is the restriction of $W(t,x,\theta)$ to a hypersurface in $(t,x,\theta)$ space,  which becomes $\kappa=0$ after the change of variable $(\kappa,x,\theta)=H(t,x,\theta)\coloneqq (t-t(\eps\theta,x),x,\theta)$.

\textbf{2. }When $W$ is continuous, we obtain the formula \eqref{a11x} 
by formally taking the limit as $\eps\to 0+$ in the integral defining $A_\eps$ in \eqref{a12}.  To justify this limit, observe that for $\eps\in [0,1]$ the function $W(t(\eps\theta,x),x)$ has $(x,\theta)$-support in a fixed compact set independent of $\eps$, and 
$$\left|W\left(t(\eps\theta,x),x,\theta\right)\right|^2 |\det F'(\eps\theta,x)|\to 
|W(t(0,x),x,\theta)|^2|\det F'(0,x)|\text{ as }\eps\to 0+$$ uniformly on that set.
\end{proof}

\begin{rem}\label{a11kz}
Let $\phi_0(x)=\phi(0,x)$.  For $Y\in H^1(\bR^n_x\times \bR_\theta)$ with compact support will also need an estimate of the form 
\begin{align}\label{a11k}
\left\|\epsilon^{-\frac12}Y(x,\phi_0/\eps)\right\|_{L^2(\bR^n)}\leq C\|Y\|_{H^1(\bR^n_x\times \bR_\theta)}\text{ for }\eps\in (0,1].
\end{align}
%In fact, without loss of generality we can assume $\partial_{x_1}\phi_0\neq 0$.
Writing $x=(x'',x_n)\in \RR^{n-1}\times \RR$, we can assume without loss of generality that $\partial_{x_n}\phi_0\neq 0$.
We can obtain \eqref{a11k}  in place of \eqref{a11h} if
\begin{align*}%\label{a11hk}
(x'',y_n)=F_0^{-1}(x'',x_n) \coloneqq (x'',\phi_0(x)), \quad (x'',x_n)=F_0(x'',y_n)=(x'',x_n(x'',y_n))
\end{align*}
to be a $C^1$ diffeomorphism on an open set $\cO_0\subset \bR^n$ that contains $\{x \mid (0,x)\in K\}$. This can be arranged locally, perhaps after modifying $K$.   
The proof of \eqref{a11} goes through without change to yield \eqref{a11k}.  When in addition $Y$ is continuous, we obtain the analogue of \eqref{a11x} with $\phi_0^{-1}(0)$ in place of $\phi^{-1}(0)$.
\end{rem}

The next proposition rephrases \eqref{a11x} and extends it to functions $W\in H^1(\tilde\Omega_T\times \bR)$ that are not necessarily continuous.

\begin{prop}\label{a20}
Let $K\subset \tilde\Omega_T$, $\phi$, and $F(s,x)$ be just as before Proposition \ref{a11z}.  Assume $W(t,x,\theta)\in H^1(\tilde\Omega_T\times\bR)$ has  $(t,x)$-support contained in the compact set $\overline{K}$.  % and $\theta$ 
Define 
$$\cS=\{(t,x)\in \tilde\Omega_T \mid \phi(t,x)=0\}.$$  
Then 
\begin{align}\label{a21}
\lim_{\eps\to 0}\frac{1}{\sqrt{\eps}} \|W(t,x,\phi/\eps)\|_{L^2(\tilde\Omega_T)}=\left\|\frac{W}{|\nabla_{t,x} \phi|^{1/2}}\right\|_{L^2(\cS\times \bR_\theta)}.
\end{align}

\end{prop}

\begin{proof}
\textbf{1. }When $W$ is continuous, we claim that \eqref{a21} is a rephrasing of \eqref{a11x}.    To see this write
\begin{align}\label{a22}
\left\|\frac{W}{|\nabla_{t,x} \phi|^{1/2}}\right\|^2_{L^2(\cS\times \bR_\theta)}=\int_{\cS\times\bR}\frac{|W(t,x,\theta)|^2}{|\nabla_{t,x}\phi|} \,\mathrm{dS} \mathrm d\theta \text{ where } \mathrm{dS}=\sqrt{1+|t_x|^2}\,\mathrm dx.
\end{align}
 Recall that 
$F(s,x)=(t(s,x),x)$ so $\det F'(0,x)=\partial_s t(0,x)$.  Moreover, since $\phi(t(s,x),x)=s$ it follows that 
\begin{align*}
\nabla\phi+\partial_t \phi \nabla t=0 \ \text{ and } \ \partial_t\phi \partial_st=1.
\end{align*}
Writing $\nabla t$ and $\partial_s t$ in terms of $\nabla \phi$ and $\partial_t\phi$ we see that the right side of \eqref{a22} equals that of~\eqref{a11x}.

\textbf{2. }Fix $\alpha>0$.  We can choose a continuous function $\tilde W$ with compact $(t,x,\theta)$-support and with $(t,x)$-support contained in $\overline{K}$ such that for $C$ as in \eqref{a11} we have
\begin{subequations}\label{a23}
    \begin{align}
        \label{a23a}
        & \|W-\tilde W\|_{H^1(\tilde \Omega_T\times\bR)} < \frac{\alpha}{3C}, \\
        \label{a23b}
        & \left\|\frac{W}{|\nabla_{t,x}\phi|^{1/2}}-\frac{\tilde W}{|\nabla_{t,x} \phi|^{1/2}}\right\|_{L^2(\cS\times \bR_\theta)} < \frac{\alpha}{3}.
    \end{align}
\end{subequations}
To get \eqref{a23b} we used continuity of the trace map $\mathrm{Tr}:H^1(\tilde\Omega_T\times\bR_\theta)\to L^2(\cS\times \bR_\theta)$.
From~\eqref{a11} and \eqref{a23a} we obtain 
\begin{align}\label{a24}
\frac{1}{\sqrt{\eps}} \|(W-\tilde W)(t,x,\phi/\eps)\|_{L^2(\tilde\Omega_T)}<\alpha/3  \ \text{ for } \ \eps\in (0,1].
\end{align}
Now write
\begin{align*}
\begin{split}
&\left|\frac{1}{\sqrt{\eps}} \|W(t,x,\phi/\eps)\|_{L^2(\tilde\Omega_T)}-\left\|\frac{W}{|\nabla_{t,x} \phi|^{1/2}}\right\|_{L^2(\cS\times \bR_\theta)}\right|\\
& \qquad \qquad \leq \left|\frac{1}{\sqrt{\eps}}\|W(t,x,\phi/\eps)\|_{L^2(\tilde\Omega_T)}-\frac{1}{\sqrt{\eps}}\|\tilde W(t,x,\phi/\eps)\|_{L^2(\tilde\Omega_T)}\right|\\
& \qquad \qquad \qquad \qquad + \left|\frac{1}{\sqrt{\eps}}\|\tilde W(t,x,\phi/\eps)\|_{L^2(\tilde\Omega_T)}-\left\|\frac{\tilde W}{|\nabla_{t,x} \phi|^{1/2}}\right\|_{L^2(\cS\times \bR_\theta)}\right|\\
& \qquad \qquad \qquad \qquad \qquad\qquad +\left|\left\|\frac{\tilde W}{|\nabla_{t,x} \phi|^{1/2}}\right\|_{L^2(\cS\times \bR_\theta)}-\left\|\frac{W}{|\nabla_{t,x} \phi|^{1/2}}\right\|_{L^2(\cS\times \bR_\theta)}\right|\\
&\qquad \qquad \qquad \qquad \qquad \qquad \qquad \qquad\qquad \eqqcolon A_\eps+B_\eps+C.
\end{split}
\end{align*}
We have $A_\eps<\alpha/3$ for $\eps\in (0,1]$ by \eqref{a24} and $C<\alpha/3$ by \eqref{a23b}.   Proposition \ref{a11z}(ii) implies $B_\eps<\alpha/3$ for small enough $\eps$. 
\end{proof}

\begin{rem}\label{a25}
\textup{Parallel to Remark \ref{a11kz}, when $Y\in H^1(\bR^n_x\times \bR_\theta)$ with compact support, we obtain the analogue of \eqref{a21} with $\phi_0^{-1}(0)$ in place of $\phi^{-1}(0)$.}

\end{rem}

\section{Nonlinear pulses in free space}\label{npfs}

In this section we construct a low regularity pulse propagating in free space.\footnote{The result of this section is not included in  earlier work on pulses in free space like \cite{ar2}, which deals with high regularity pulses.}  As in \S \ref{size} we denote $\tilde\Omega=[0,T]\times\RR^n$ and points in $T^*(\mathbb R\times \bR^n)$ by $(t,x,\tau,\xi)$.

For $P$ as in \eqref{in1} and $T>0$ consider the Cauchy problem:
\begin{align}\label{a2}
\begin{cases}
Pu^\eps=f(t,x,u^\eps,\partial_t u^\epsilon, \nabla u^\eps), \ & (t,x)\in [0,T]\times\RR^n=\tilde\Omega_T, \\
u^\eps(0,x)=g^\eps(x),  \ & x\in \RR^n, \\ 
\partial_t u^\eps(0,x)=h^\eps(x), \ & x\in \RR^n,
\end{cases}
\end{align}
where $\epsilon>0$ and as in \eqref{in1} we assume 
\begin{itemize}
    \item $f$ satisfies Assumption \ref{amr1z};
    \item pulse initial data:
\begin{align}\label{a1a}
\begin{split}
&g^\eps=\sqrt{\eps}V_0(x,\phi_\mi(0,x)/\eps)+o_\eps(1) \text{ in }H^1(\bR^n)\\
&h^\eps=\frac{1}{\sqrt{\eps}}V_1(x,\phi_\mi(0,x)/\eps)+o_\eps(1) \text{ in }L^2(\bR^n),
\end{split}
\end{align}
where $\phi_\mi$ is a characteristics phase as in Assumption \ref{amr3}. 
The profiles $V_0$, $V_1$ in \eqref{a1a} have compact support in $(x,\theta)$ and satisfy the polarization condition
\begin{subequations}\label{a1c}
\begin{align}
\label{a1ca}
&V_0\in H^2(\bR^n_x\times\bR_\theta),\\
\label{a1cb}
&V_1=(\partial_t\phi_\mi(0,x))\partial_\theta V_0\in H^1(\bR^n_x\times\bR_\theta).  
\end{align}
\end{subequations}
\end{itemize}

The main result is this section gives approximate solutions to \eqref{a2} for pulses in the free space.
\begin{theo}\label{fnp}
Consider the problem \eqref{a2}, \eqref{a1a} under the structural Assumptions \ref{amr1}, \ref{amr1z}, Assumption \ref{amr3} on the incoming phase $\phi_\mi$, and take $V_0$, $V_1$ be as in \eqref{a1c}.  Suppose also that $V_0$ (and hence $V_1)$ has compact $x$-support in $U\cap \{t=0\}$, where $U$ is the open ball in Assumption~\ref{amr3}.
Then there is a $T>0$ and profile $W\in H^1(\tilde\Omega_T\times \bR)$ 
with compact support in $\tilde\Omega_T\times \bR_\theta$ such that 
for $U(t,x,\theta)\coloneqq \int^\theta_{-\infty}W(t,x,s)\,\mathrm ds$ the exact solution $u^\eps$ to \eqref{a2} satisfies
\begin{align}\label{b1zc}
u^\eps\thickapprox_{H^1} \sqrt{\eps}U\left(t,x,\frac{\phi_\mi(t,x)}{\eps}\right) \ \text{ on }\ \tilde\Omega_T.  
\end{align}
Here $\thickapprox_{H^1}$ is defined  in Definition \ref{b}, and the profile $W$ is determined by solving the profile equations \eqref{pe}.
\end{theo}

\begin{rem}\label{fnpa}
(i) The $\epsilon$ dependence in the right side of the profile equations \eqref{pe} implies of course that $W$ has $\epsilon$ dependence; but we suppress the epsilon in the notation and write $W$ instead of $W^\eps$.  When we say, for example in Theorem \ref{fnp}, that $
W\in H^1(\Omega_T\times \bR)$, we mean that $W^\eps\in H^1(\Omega_T\times \bR)$ with an $H^1$ norm that is uniformly bounded for $\eps>0$ small.

(ii) Because of the Lipschitz assumption on $f$, the only restriction on $T$ in Theorem \ref{fnp} comes from the size of the ball $U$ in Assumption \ref{amr3} on which $\phi_\mi$ satisfies the eikonal equation.
\end{rem}

\begin{defn}\label{b}
The statement \eqref{b1zc} means that for any sequence $\delta_\ell\searrow 0$ there exist $C^\infty$ profiles $W^\ell(t,x,\theta)$ with compact supports in $\tilde\Omega_T\times \bR_\theta$ and $\eps_\ell>0$ such that 
\begin{align}\label{b3}
\|W-W^\ell\|_{H^1(\tilde\Omega_T\times\bR_\theta)} <\delta_\ell, \ \text{ where } \ W=\partial_\theta U
\end{align}
and
\begin{align}\label{b4}
\left\|u^\eps-\sqrt{\eps}U^\ell_{\eps^d}\left(t,x,\frac{\phi_\mi}{\eps}\right)\right\|_{H^1(\Omega_T)}\lesssim \delta_\ell \ \text{ for } \ \eps\in (0,\eps_\ell].
\end{align}
Here we let $W^\ell_{\eps^d}(t,x,\theta)$ be a moment zero approximation\footnote{Moment zero approximations are defined in Appendix~\ref{mz}; in the notation of that appendix we take $\omega=\eps^d$ here.} of $W^\ell(t,x,\theta)$ for $0<d<1/2$ and 
$U^\ell_{\eps^d}$ be the unique moment zero primitive in $\theta$ of $W^\ell_{\eps^d}$.
\end{defn}

\begin{rem}
(i) Observe that even if $W$ is $C^\infty$ with compact support in $\tilde\Omega_T\times \bR_\theta$ the function
$$U(t,x,\theta)=\int_{-\infty}^\theta W(t,x,s)\,\mathrm ds,$$ 
although smooth, lies in $L^2$ if and only if the moment of $W$, 
$\int^\infty_{-\infty}W(t,x,s) \,\mathrm ds$, is equal to~$0$.\footnote{In that case $U(t,x,\theta)$ is rapidly decaying in $\theta$.} 
Thus, there is no reason to expect the right side of \eqref{b1zc} to lie in $H^1(\Omega_T)$.  Definition~\ref{b} gives~\eqref{b1zc} an \emph{asymptotic} meaning and is similar to definitions used in \cite{jmr1995tams, jmr1996cpam, jmr2000mams, wangwil}.
For any fixed $\ell$ and $\eps>0$ the profile $U^\ell_{\eps^d}$ in \eqref{b4} is smooth in $(t,x,\theta)$ and rapidly decaying in $\theta$.

We note that the $L^2$ convergence in Condition \eqref{nl10} of Definition \ref{meaninga} has been replaced by $H^1$ convergence in \eqref{b3}. The truncation argument needed in the diffraction problem prevents  $H^1$ convergence.

(ii)  One can check that for $\ell$ large and $\eps$ small, $U^\ell_{\eps^d}$ is close to $U$ in 
the space $C_{\mathrm b}(\bR_\theta,H^1(\tilde\Omega_T))$ with norm $\sup_\theta \|a(t,x,\theta)\|_{H^1(\tilde\Omega_T)}$.
\end{rem}

\subsection{Profile equations}
In the rest of \S \ref{npfs} we write $\phi\coloneqq \phi_\mi$.  
We  look for an approximate solution of the form
\begin{align*}
u^\eps_\ma(t,x)=\sqrt{\eps}U(t,x,\phi/\eps),
\end{align*}
where $U(t,x,\theta)$ is a pulse profile to be determined.   We determine $U(t,x,\theta)$ by plugging
$u^\eps_\ma(t,x)$ into the problem \eqref{a2} and then seeing what conditions $U(t,x,\theta)$ must satisfy in order for $u^\eps_\ma$ to be a useful approximate solution.

When we compute $P(t,x,\partial_t, \nabla)u^\eps_\ma$ we obtain an expression that can be written in the form
\begin{align}\label{a9}
Pu^\eps_\ma(t,x)=\left[\eps^{-3/2}E_1(t,x,\theta)+\eps^{-1/2}E_2(t,x,\theta)+\eps^{1/2}E_3(t,x,\theta)\right]_{|\theta=\frac{\phi(x,t)}{\eps}}.
\end{align}
Here $E_1$, $E_2$, $E_3$ are functions that we now compute explicitly in terms of $\phi$ and $U(t,x,\theta)$.\footnote{We just assume for now that $U$ has enough regularity so that the computations below make sense.}   

The second-order operator $P$ has the form
\begin{align*}
P=p(t,x,\partial_t,\nabla)+B_1(t,x,\partial_t,\nabla)+B_0(t,x)
\end{align*}
where $B_0$, $B_1$ are of orders $0$, $1$ respectively, and we set
\begin{align}\label{b6zz}
    P_1 =p(t,x,\partial_t,\nabla)+B_1(t,x,\partial_t,\nabla).
\end{align}
With $T_\phi(t,x)\coloneqq\partial_{\tau,\xi}p(t,x,\partial_t\phi, \nabla\phi)\cdot\partial_{t,x}$ and $W=\partial_\theta U$, we  obtain
\begin{equation}\label{b6b}
    \begin{split}
        Pu_a^{\epsilon}(t,x) = &
          \epsilon^{-3/2} p(t,x,\partial_t\phi, \nabla \phi)(\partial_{\theta}^2 U) \left(t,x,\theta\right)_{|\theta=\frac{\phi}{\eps}} \\
        & \qquad +\epsilon^{-1/2} \left[(T_{\phi} W)(t,x,\theta ) 
         + (P_1\phi) W\left( t,x,\theta \right) \right]_{|\theta=\frac{\phi}{\eps}} \\
        &\qquad \qquad +\eps^{1/2} PU(t,x,\theta)_{|\theta=\frac{\phi}{\eps}}.
    \end{split}
\end{equation}

On the other hand we have 
\begin{align*}%\label{a10}
f(t,x,u^\eps_\ma,\partial_t u^\epsilon_a, \nabla u^\eps_\ma)=f(t,x,0,\epsilon^{-\frac12} W(t,x,\phi/\eps)\partial_t \phi,\epsilon^{-\frac12} W(t,x,\phi/\eps)\nabla \phi)+r_\eps(t,x),
\end{align*}
where $r_\eps(t,x)$ is an error term we show later is $o_\eps(1)$ in $L^2(\tilde\Omega_T)$; see \eqref{ead}.

In the error analysis  we will need the following quantities to be small for $\eps$ small:
\begin{subequations}\label{a10q}
\begin{align}
\label{a10qa}
& \|Pu^\eps_\ma-f(t,x,u^\eps_\ma,\partial_t u^\epsilon_a, \nabla u^\eps_\ma)\|_{L^2(\tilde\Omega_T)},\\
\label{a10qb}
& \|(u^\eps-u^\eps_\ma)_{|t=0}\|_{H^1(\bR^n)}, \text{ and } \|{\partial_t}(u^\eps-u^a_\eps)_{|t=0}\|_{L^2(\bR^n)}.
\end{align}
\end{subequations}
The eikonal equation satisfied by $\phi$ makes the $E_1$ (as in \eqref{a9}) contribution to \eqref{a10qa}  vanish.
If we choose $W$ to satisfy the profile equations:
 \begin{align}\label{pe}
 \begin{split}
T_\phi W+(P_1\phi)W & =\sqrt{\eps}f(t,x,0,\epsilon^{-\frac12}W(t,x,\theta)\partial_t\phi,\epsilon^{-\frac12}W(t,x,\theta)\nabla \phi),\\
W(0,x,\theta) & =\frac{V_1(x,\theta)}{\partial_t\phi(0,x)},
 \end{split}
 \end{align}
then the main  contribution of $\eps^{-1/2}E_2-f$ to \eqref{a10qa} vanishes.  Moreover, a short computation using the initial conditions \eqref{a1a} indicates that the initial condition in \eqref{pe} is the correct choice for $W(0,x,\theta)$ if we want to make the terms in \eqref{a10qb} small.  We confirm that in the error analysis below.

\begin{nota}
We remark some notations here.
\begin{enumerate}
    \item With  some abuse we write 
    $$f(\epsilon^{-\frac12}W)\coloneqq f(t,x,0,\epsilon^{-\frac12}W(t,x,\theta)\partial_t\phi,\epsilon^{-\frac12}W(t,x,\theta)\nabla\phi).$$

    \item[(2a)] Suppose $f^{\rho,\eps}(z)$ is a function taking values in a Banach space $X$, where  $\rho$ and $\eps$ are small positive parameters.\footnote{The notation in  (2a) is used in \S \ref{fnp}, while that in (2b) will be used in \S \ref{diff}.} The statement
$$\|f^{\rho,\eps}\|_X\leq o_\rho(1)+o_\eps(1)$$
means that $\|f^{\rho,\eps}\|_X\leq A^{\rho,\eps}+B^{\rho,\eps}$, where $A^{\rho,\eps}\to 0$ as $\rho\to 0$ uniformly with respect to  $\eps$.   Moreover, for $\rho$ \emph{fixed},  $B^{\rho,\eps}\to 0$ as $\eps\to 0$.

    \item[(2b)] Similarly, if $\mu$ is another small positive parameter,  
$$\|g^{\mu,\rho,\eps}\|_X\leq o_\mu(1)+o_\rho(1)+o_\eps(1)$$
means that $\|g^{\mu,\rho,\eps}\|_X\leq A^{\mu,\rho,\eps}+B^{\mu,\rho,\eps}+C^{\mu,\rho,\eps}$, where $A^{\mu,\rho,\eps}\to 0$ as $\mu\to 0$ uniformly with respect to $(\rho,\eps)$.  Moreover,  for $\mu$ \emph{fixed}, $B^{\mu,\rho,\eps}\to 0$ as $\rho\to 0$ uniformly with respect to $\eps$, and for $(\mu,\rho)$ \emph{fixed}, $C^{\mu,\rho,\eps}\to 0$ as $\eps\to 0$. In this paper $\mu$, $\rho$, and $\eps$ denote truncation, regularization, and wavelength parameters respectively.   Truncation is used only in \S \ref{diff}.

    \item[(3)] Given a function $u(t,x,\theta)$ we will sometimes write $u(t)$ for the function of $(x,\theta)$ whose value at $(x,\theta)$ is $u(t,x,\theta)$.
\end{enumerate}
\end{nota}

\subsection{Proof of Theorem \ref{fnp}}\label{fnpz}

We first remark that standard energy estimates for the  hyperbolic initial value problem
\begin{align*}%\label{a3z}
Pu=f,  \ \  u_{|t=0}=g, \ \ (u_t)_{|t=0}=h \ \text{ on }\ [0,T]\times \bR^n
\end{align*}
have the form:\footnote{See Chazarain--Piriou \cite[p. 371]{CP}, for example, for a proof of the energy estimate.}
\begin{align}\label{a4}
\begin{split}
&\|u(t,\cdot)\|^2_{H^{k+1}(\bR^n)}+\|\partial_t u(t,\cdot)\|^2_{H^k(\bR^n)} \\
&\qquad \qquad \lesssim \|g\|_{H^{k+1}(\bR^n)}^2+ \|h\|_{H^k(\bR^n)}^2+\int^t_0\|f(s,\cdot)\|_{H^k(\bR^n)}^2 \,\mathrm ds,
\end{split}
\end{align}
where $k\in \bN_0$, $t\in [0,T]$.

\begin{proof}[Proof of Theorem \ref{fnp}]
\textbf{1. Existence of $u^\eps$.} Recall that $V_0\in H^2(\bR^n\times \bR)$.  Using Proposition~\ref{crude} (or rather, its analogue for functions of $(x,\theta)$) and the estimate \eqref{a11k}, we see that $\|g^\eps\|_{H^1(\bR^n)}$ and $\|h^\eps\|_{L^2(\bR^n)}$ are bounded uniformly for small $\eps$.  The Lipschitz property of $f$, \eqref{lips}, thus allows us to use the energy estimate 
\eqref{a4} and a standard Picard iteration and continuation argument to obtain a unique solution $u^\eps\in H^1(\tilde\Omega_T)$ of \eqref{a2}, \eqref{a1a} for appropriate $T>0$ independent of small $\eps>0$; see Remark \ref{fnpa}(ii).

\textbf{2. Linear profile estimates.} Consider the linear problem 
\begin{align}\label{nl1v}
\begin{cases}
T_\phi W+(P_1\phi)W =F & (t,x,\theta)\in\tilde\Omega_T\times \bR,\\
W(0,x,\theta) =a(x,\theta) & (x,\theta)\in \RR^n\times\RR,
\end{cases}
\end{align}
where $F$ and $a$ are $C^\infty$ and compactly supported.   This problem can be solved by the method of characteristics 
and from that explicit solution one can readily ``read off" the standard estimates\footnote{When deriving the explicit solution formula, it  is convenient first to use $\partial_t\phi\neq 0$ to reduce to the case where the coefficient of $\partial_t$ in the interior equation of \eqref{nl1} is one.}
\begin{subequations}\label{nl1z}
\begin{align}
\label{nl1za}
& \|W(t)\|_{H^k(\bR^n_x\times\bR_\theta)}\leq C\left[ \|a\|_{H^k(\bR^n\times \bR)}+\|F\|_{L^2([0,t],H^k(\bR^n_x\times\bR_\theta))}\right]\text{ for }t\in [0,T], \;\;k=0,1,\\
\label{nl1zb}
& \|\partial_tW(t)\|_{L^2(\bR^n\times \bR)}\leq C\left[ \|a\|_{H^1(\bR^n\times \bR)}+\|F\|_{{L^2([0,t],H^k(\bR^n_x\times\bR_\theta))}}\right]+\sup_{s\in [0,t]} \|F(s)\|_{{L^2(\bR^n\times \bR)}}.
\end{align} 
\end{subequations}
For $F\in H^k(\tilde\Omega_T\times \bR)$ and $a\in H^k(\bR^n\times \bR)$ with $k=0$ or $1$, we approximate $F$ and $a$ by smooth compactly supported functions and take limits to obtain a solution $W$ of \eqref{nl1v} that  continues to satisfy \eqref{nl1z}.\footnote{In proving estimates for the diffraction problem one must contend with the blow up of  $P_1\phi_\mr$ near the shadow boundary.} The size of $T$ and the supports of $F$ and $a$ are limited by the size of the ball $U$ in Assumption~\ref{amr3} in which $\phi$ satisfies the eikonal equation. 

\textbf{3. Nonlinear profiles.}  For $a\in H^1(\RR^n\times \RR)$, consider the nonlinear profile problem
\begin{align}\label{nl1}
\begin{cases}
T_\phi W+(P_1\phi)W =\sqrt{\eps}f\left(\epsilon^{-\frac12}W\right) & (t,x,\theta)\in \tilde\Omega_T\times\RR,\\
W(0,x,\theta) =\frac{V_1(x,\theta)}{\partial_t\phi(0,x)} = a(x,\theta)\ & (x,\theta)\in \RR^n\times\RR,
\end{cases}
\end{align}
which we solve using step \textbf{2} and  the  iteration scheme
\begin{align*}%\label{nl1u}
\begin{cases}
T_\phi W^{n+1}+(P_1\phi)W^{n+1} =\sqrt{\eps}f\left(\epsilon^{-\frac12}W^n\right), & (t,x,\theta)\in \tilde\Omega_T\times\RR,\\
W^{n+1}(0,x,\theta) =a(x,\theta), & (x,\theta)\in \RR^n\times \RR.
\end{cases}
\end{align*}
The estimates \eqref{nl1za} and the Lipschitz property of $f$ imply that   the iterates $W^n$ are bounded in $C([0,T],H^1(\bR^n\times \bR))$.  The equation and \eqref{nl1zb} show that the $W^n$ are bounded in $H^1(\tilde\Omega_T\times\bR)$, and hence a subsequence converges weakly to some $W\in H^1(\tilde \Omega_T\times \bR)$.   On the other hand the estimate \eqref{nl1za} with $k=0$ and $f$ Lipschitz  imply that the $W^n$ converge in $L^2(\tilde\Omega_T\times \bR)$  to a limit that must equal $W$.  By interpolation (Proposition \ref{interpolate})  we conclude 
\begin{align*}%\label{npa}
W_n\to W \ \text{ in } \ H^s(\tilde\Omega_T\times \bR) \ \text{ for } \ s<1,  \ \text{ where } \ W\in H^1(\tilde\Omega_T\times \bR).
\end{align*}
We now have a solution  $W\in H^1(\tilde\Omega_T\times \bR)$  to \eqref{nl1} with compact support in $(t,x,\theta)$. 
\begin{rem}
The function $W$ is not yet suitable for use in constructing an approximate solution.    If we take 
$U(t,x,\theta)=\int^\theta_{-\infty}W(t,x,s) \,\mathrm ds$, then $U$ generally fails to lie in $L^2$ and is not regular enough for the computation \eqref{b6b} to make sense.  To remedy this we  regularize $W$ and take a moment-zero approximation before taking the primitive in $\theta$ to obain $U^{\rho,\sigma}_\omega$ as in step \textbf{7}.
\end{rem}

\textbf{4. Regularize $W$ in $(x,\theta)$.} Set $\rho=(\rho_x,\rho_\theta)$ and let $\delta_\rho(x,\theta)$ be a smooth approximate identity.  Define 
$W^\rho=R^\rho W \coloneqq \delta_\rho *W$.       
Observe that regularizing with $R^\rho$ preserves the initial condition at $t=0$.

Since $W\in H^1(\tilde\Omega_T\times \bR)\subset H^1([0,T],L^2(\bR^n_x\times \bR_\theta))$ we have 
\begin{subequations}%\label{regd}
\begin{align}
\label{regda}
& W^{\rho}\in H^1([0,T],H^\infty(\bR^n_x\times \bR_\theta))\text{ for each }\rho,\\
\label{regdb}
& \|W-W^\rho\|_{H^1(\tilde \Omega_T\times\bR_\theta)}=o_\rho(1).
\end{align}
\end{subequations}
Applying $R^{\rho}$ to the problem \eqref{nl1} gives 
\begin{align*}%\label{rege}
\begin{cases}
T_\phi W^{\rho}+(P_1\phi)W^{\rho} =\sqrt{\eps}f\left(\epsilon^{-\frac12}W\right)^{\rho}+[T_\phi,R^{\rho}]W+[P_1\phi,R^{\rho}]W, & (t,x,\theta)\in \tilde\Omega_T\times \RR, \\
W^{\rho}(0,x,\theta) =a^{\rho}(x,\theta)\in H^\infty(\bR^n\times\bR), & (x,\theta)\in \RR^n\times \RR.
\end{cases}
\end{align*}
With Friedrichs's Lemma (Lemma \ref{fl}) and Lemma \ref{rl}, we see that 
\begin{subequations}\label{regg}
\begin{align}
\label{regga}
T_\phi W^{\rho}+(P_1\phi)W^{\rho} & =\sqrt{\eps}f\left(\epsilon^{-\frac12} W\right)+o_\rho(1) && \text{ in }H^1(\tilde\Omega_T\times \bR), \\
\label{reggb}
W^{\rho}(0,x,\theta) & =a(x,\theta)+o_\rho(1) && \text{ in } H^1(\bR^n\times\bR).
\end{align}
\end{subequations}

\textbf{5. Regularize in $t$.}
Now regularize in $t$ using an approximate identity $\delta_\sigma(t)\in C^\infty_c(\bR)$ \emph{supported in} $t\leq 0$.\footnote{This implies that for $t\geq 0$, $(R^\sigma f)(t) \coloneqq (f*\delta_\sigma)(t)$ depends only on $f_{|t\geq 0}$.}  This disturbs the initial condition, but in a way that we can control.   With $W^{\rho,\sigma}\coloneqq R^\sigma W^\rho$ we obtain 
\begin{subequations}\label{regh}
\begin{align}
\label{regha}
T_{\phi}W^{\rho,\sigma}+(P_1\phi)W^{\rho,\sigma} & =  \sqrt{\eps}f\left(\epsilon^{-\frac12}W\right)+o_{\rho}(1)+o_\sigma(1) && \text{ in }H^1(\tilde\Omega_T\times \bR)\\
\label{reghb}
(W^{\rho,\sigma})_{|t=0} & =a+o_{\rho}(1)+o_\sigma(1) && \text{ in }H^1({\bR^n}\times\bR)\\
\label{reghc}
\|W^\rho-W^{\rho,\sigma}\|_{H^1(\tilde\Omega_T\times \bR)} & =o_\sigma(1). &&
\end{align}
\end{subequations}
The function $W^{\rho,\sigma}$ is $C^\infty$ with compact support in $(t,x,\theta)$.  Here \eqref{regha} follows by applying  $R^\sigma$ to \eqref{regga} and using Friedrichs's lemma.
The equality \eqref{reghb} follows by writing
$$W^{\rho,\sigma}-a=(W^{\rho,\sigma}-W^\rho)+(W^\rho-a),$$
using \eqref{reggb}, and applying Remark \ref{tl3}.

\textbf{6. Moment zero approximation of $W^{\rho,\sigma}$.} Using $W^{\rho,\sigma}_\omega$ to denote the moment zero approximation of $W^{\rho,\sigma}$ as in Definition \ref{mz2},  we apply the Fourier multiplier $\chi_\omega(D_\theta)$ to the problem \eqref{regh} and apply Proposition \ref{a10z} to obtain
\begin{align}\label{nl2a}
T_\phi W^{\rho,\sigma}_\omega+(P_1\phi)W^{\rho,\sigma}_\omega & =\sqrt{\eps}f\left(\epsilon^{-\frac12}W\right)+o_{\rho}(1)+o_\sigma(1)+o_\omega(1) && \text{ in }H^1(\tilde\Omega\times\bR), \\
(W^{\rho,\sigma}_\omega)_{|t=0} & =a+o_\rho(1)+o_\sigma(1)+o_\omega(1) && \text{ in } H^1(\bR^n\times\bR).
\end{align}
The function  $W^{\rho,\sigma}_p$ no longer has compact $\theta$ support, but is $C^\infty$ and decays rapidly in $\theta$.

\textbf{7. Approximate solution $u^\eps_\ma$.}  For $0<d<1/2$ define  $$u^\eps_\ma(t,x)\coloneqq \sqrt{\eps}U^{\rho,\sigma}_\omega(t,x,\phi/\eps),$$ 
where $\omega=\eps^d$ and $U^{\rho,\sigma}_\omega$ is the moment zero primitive in $\theta$ of 
$W^{\rho,\sigma}_\omega$.   In the notation of Appendix~\ref{mz},  $U^{\rho,\sigma}_\omega=\left(W^{\rho,\sigma}_\omega\right)^*$.    By a slight variant of Proposition \ref{a10z}(iii) we have
\begin{align*}
U^{\rho,\sigma}_\omega\in H^\infty(\tilde\Omega_T\times\bR) \ \text{ for each \emph{fixed} } \ \omega.
\end{align*}

We claim there exists $M>0$ such $\|u^\eps_\ma\|_{H^1(\tilde\Omega_T)}\leq M$ for $\eps$ small.   Indeed, the largest contributions to this norm have the form 
\begin{align*}
\begin{split}
\left\|\epsilon^{-\frac12}W^{\rho,\sigma}_{\eps^d}(t,x,\phi/\eps)(\partial_t\phi, \nabla\phi)\right\|_{L^2(\tilde\Omega_T)} & \lesssim \|W^{\rho,\sigma}_{\eps^d}\|_{H^1(\tilde\Omega_T\times \bR_\theta)}\\
& \lesssim \|W^{\rho,\sigma}\|_{H^1(\tilde\Omega_T\times \bR_\theta)}\lesssim \|W\|_{H^1(\tilde\Omega_T\times \bR_\theta)}.
\end{split}
\end{align*}
Here we used the estimate \eqref{a11} and Proposition \ref{a10z}(i) for the first two inequalities.

\textbf{8. Error analysis.}
Let us write $f(t,x,u^\eps,\partial_t u^\epsilon, \nabla u^\eps)$ as $f(u^\eps)$ for now.   
We have
\begin{align*}
Pu^\eps-Pu^\eps_\ma=[f(u^\eps)-f(u^\eps_\ma)]+[f(u^\eps_\ma)-Pu^\eps_\ma] \eqqcolon D+E.
\end{align*}
When estimating $\|u^\eps-u^\eps_\ma\|_{H^1(\tilde\Omega_T)}$ using \eqref{a4}, we see that the Lipschitz property of $f$ allows us to absorb the contribution of $\|D\|_{L^2(\tilde\Omega_T)}$ by an application of Gronwall's inequality.\footnote{Gronwall's inequality (integral version) says that if $y$ and $\phi$ are continuous, nonnegative functions on $[0,T]$ which satisfy for some constants $\alpha,C\geq 0$,
$y(t)\leq C\left[\alpha+\int^t_{0}\left(y(s)+\phi(s)\right) \,\mathrm ds\right] \text{ for }t\in [0,T],$
then
$$y(t)\leq C\left[\alpha e^{Ct}+\int^t_{0}e^{C(t-s)}\phi(s) \,\mathrm ds\right].$$}

Thus, it remains to estimate
\begin{align*}
\|Pu^\eps_\ma-f(u^\eps_\ma)\|_{L^2(\tilde\Omega_T)},\ \   \|(u^\eps-u^\eps_\ma)_{|t=0}\|_{H^1(\bR^n)}, \  \text{ and } \  \|{\partial_t}(u^\eps-u^\eps_\ma)_{|t=0}\|_{L^2(\tilde\Omega_T)}.
\end{align*}

\textbf{9. Estimate $\|Pu^\eps_\ma-f(u^\eps_\ma)\|_{L^2(\tilde\Omega_T)}$.}
The profile $U^{\rho,\sigma}_\omega$ is regular enough so that the computation \eqref{b6b} is valid.   Thus we have
\begin{align}\label{eaa}
    \begin{split}
  Pu^\eps_\ma=\frac{1}{\sqrt{\eps}}\left[T_\phi W^{\rho,\sigma}_\omega+(P_1\phi)W^{\rho,\sigma}_\omega\right]_{|\theta=\phi/\eps}+\sqrt{\eps}(PU^{\rho,\sigma}_\omega)_{|\theta=\phi/\eps}.     
    \end{split}
\end{align}
Since $PU^{\rho,\sigma}_\omega=\left(PW^{\rho,\sigma}_\omega\right)^*$ and $PW^{\rho,\sigma}_\omega\in H^1(\tilde\Omega_T\times\bR_\theta)$, using Proposition \ref{a10z} we obtain
\begin{align}\label{eab}
\begin{split}
\sqrt{\eps}\|(PU^{\rho,\sigma}_\omega)_{|\theta=\phi/\eps}\|_{L^2(\tilde\Omega_T)}
& =\sqrt{\eps} \|\left(PW^{\rho,\sigma}_\omega\right)^*_{|\theta=\phi/\eps}\|_{L^2(\tilde\Omega_T)}
\leq\sqrt{\eps} \|\left(PW^{\rho,\sigma}_\omega\right)^*\|_{H^1(\tilde\Omega_T\times\bR_\theta)}\\
& \leq \sqrt{\eps}\frac{\|PW^{\rho,\sigma}_\omega\|_{H^1(\tilde\Omega_T\times\bR_\theta)}}{\omega} \leq \sqrt{\eps}\frac{\|PW^{\rho,\sigma}\|_{H^1(\tilde\Omega_T\times\bR_\theta)}}{\eps^d}=O(\eps^{\frac{1}{2}-d})
\end{split}
\end{align}
for fixed $\rho$.
Here we used Proposition \ref{crude} for the first inequality and Proposition \ref{a10z} for the other two inequalities.

Next consider 
\begin{align}\label{eac}
\begin{split}
f(u^\eps_\ma) & =f(t,x,\sqrt{\eps}U^{\rho,\sigma}_\omega,\sqrt{\eps}\nabla_{t,x} U^{\rho,\sigma}_\omega+\epsilon^{-\frac12}W^{\rho,\sigma}_\omega\nabla_{t,x} \phi)_{|\theta=\phi/\eps}\\
& =f(t,x,0,\epsilon^{-\frac12}W^{\rho,\sigma}_\omega\nabla_{t,x} \phi)_{|\theta=\phi/\eps}+r_\eps.
\end{split}
\end{align}
Since $f$ is Lipschitz, we have 
\begin{align}\label{ead}
\|r_\eps\|_{L^2(\tilde\Omega_T)}\lesssim \|\sqrt{\eps}U^{\rho,\sigma}_\omega\|_{L^2(\tilde\Omega_T)}+\|\sqrt{\eps}\nabla_{t,x} U^{\rho,\sigma}_\omega\|_{L^2(\tilde\Omega_T)}\to 0 \text{ as }\eps\to 0
\end{align}
by an estimate just like \eqref{eab}.

From \eqref{eaa}--\eqref{ead} we see that
\begin{align*}
    \begin{split}
        & \|Pu^\eps_\ma-f(u^\eps_\ma)\|_{L^2(\tilde\Omega_T)} \\
        & \quad =\left\|\epsilon^{-\frac12}\left(T_\phi W^{\rho,\sigma}_\omega+(P_1\phi)W^{\rho,\sigma}_\omega-\sqrt{\eps}f\left(\epsilon^{-\frac12}W^{\rho,\sigma}_\omega\right)\right)_{|\theta=\phi/\eps}\right\|_{L^2(\tilde\Omega_T)}+o_\eps(1)\\
        & \quad \lesssim \left\|\epsilon^{-\frac12}\left(T_\phi W^{\rho,\sigma}_\omega+(P_1\phi)W^{\rho,\sigma}_\omega-\sqrt{\eps}f\left(\epsilon^{-\frac12}W\right)\right)\right\|_{L^2(\tilde\Omega_T)}\\
        & \qquad \qquad \qquad \qquad \qquad \qquad \qquad \qquad \quad \ \ +\left\|f\left(\epsilon^{-\frac12}W\right)-f\left(\epsilon^{-\frac12}W^{\rho,\sigma}_\omega\right)\right\|_{L^2(\tilde\Omega_T)}+o_\eps(1)\\
        &\quad \lesssim\left\|T_\phi W^{\rho,\sigma}_\omega+(P_1\phi)W^{\rho,\sigma}_\omega-\sqrt{\eps}f\left(\epsilon^{-\frac12}W\right)\right\|_{H^1(\tilde\Omega_T\times\bR)}  +\epsilon^{-\frac12}\left\|W-W^{\rho,\sigma}_\omega\right\|_{L^2(\tilde\Omega_T)}+o_\eps(1)\\
        &\quad \lesssim(o_\rho(1)+o_\sigma(1)+o_\omega(1))+\left\|W-W^{\rho,\sigma}_\omega\right\|_{H^1(\tilde\Omega_T\times\bR)}+o_\eps(1)\\
        & \quad =o_\rho(1)+o_\sigma(1) +o_\eps(1).
    \end{split}
\end{align*}
Here we used \eqref{a11} and the Lipschitz property of $f$ to obtain the fourth line.  We used \eqref{nl2a} and \eqref{a11}, along with Friedrichs's lemma and Proposition \ref{a10z}  to get the last line.

\textbf{10. Estimate $\|(u^\eps-u^\eps_\ma)_{|t=0}\|_{H^1(\bR^n)}.$}
Recall 
\begin{align*}
\begin{split}
V_0(x,\theta) & =U(0,x,\theta)\in H^2(\bR^n\times\bR),\\
W(0,x,\theta) & =a(x,\theta) \coloneqq \frac{V_1(x,\theta)}{\partial_t\phi(0,x)}=\partial_\theta V_0(x,\theta)\in H^1(\bR^n\times \bR).
\end{split}
\end{align*}
With $\omega=\eps^d$ consider 
\begin{align*}%\label{eah}
\begin{split}
(u^\eps-u^\eps_\ma)_{|t=0}& =g^\eps-\sqrt{\eps}U^{\rho,\sigma}_\omega(0,x,\phi_0/\eps)\\
& = \left(g^\eps-\sqrt{\eps}V_0(x,\phi_0/\eps)\right)+\left(\sqrt{\eps}V_0(x,\phi_0/\eps)-\sqrt{\eps}U^{\rho,\sigma}_\omega(0,x,\phi_0/\eps)\right)\\
& \eqqcolon A+B.
\end{split}
\end{align*}
We have $\|A\|_{H^1(\bR^n)}=o_\eps(1)$ by definition of $g^\eps$ in \eqref{a1a}.    The contribution to $\|B\|_{H^1({\bR^n})}$ whose estimate is least obvious is, with $\omega=\eps^d$:
\begin{align}\label{eahh}
\begin{split}
&\epsilon^{-\frac12} \|W(0,x,\phi_0/\eps)\nabla \phi_0-W^{\rho,\sigma}_{\omega}\nabla \phi_0\|_{L^2({\bR^n})}\\
& \qquad \qquad \leq \epsilon^{-\frac12}\|W\nabla \phi_0-W^{\rho}\nabla \phi_0\|_{L^2({\bR^n})}+\epsilon^{-\frac12}\|W^{\rho}\nabla \phi_0-W^{\rho,\sigma}\nabla \phi_0\|_{L^2({\bR^n})}\\
&\qquad \qquad \qquad \qquad +\epsilon^{-\frac12}\|W^{\rho,\sigma}\nabla \phi_0-W^{\rho,\sigma}_{\omega}\nabla \phi_0\|_{L^2({\bR^n})}\eqqcolon C+D+E.
\end{split}
\end{align}
Using Remark \ref{a11kz}, the fact that $W_{|t=0}=\partial_\theta V_0=a\in H^1({\bR^n}\times \bR_\theta)$, and Friedrich's lemma, we obtain
$$C\lesssim \|W-W^{\rho}\|_{H^1({\bR^n}\times \bR_\theta)}=o_\rho(1).$$
The tricky term in \eqref{eahh} is $D$, since $W^{\rho,\sigma}(0,x,\theta)$ is obtained by first regularizing $W^{\rho}$ in $t$ and then restricting to $t=0$.   Remark \ref{tl3} implies that:
$$\|W^{\rho}-W^{\rho,\sigma}\|_{H^1({\bR^n}\times \bR_\theta)}=o_{\sigma}(1) \ \text{ at } \ t=0.$$
Here we used that $W^{\rho}(t,x,\theta)\in H^1([0,T],H^\infty({\bR^n}\times \bR_\theta))$.  Thus, $D=o_{\sigma}(1)$ by Remark~\ref{a11kz}.   By Remark~\ref{a11kz} again and Proposition \ref{a10z}(ii) we have $E=o_\eps(1)$.

To complete the estimate of $\|(u^\eps-u^\eps_\ma)_{|t=0}\|_{H^1(\overline{\bR^n_+})}$ we must estimate the contribution to $\|B\|_{H^1({\bR^n})}$ given by 
\begin{align*}%\label{eahg}
\|\partial_{x}^\alpha(\sqrt{\eps} V_0(x,\theta)-\sqrt{\eps}U^{\rho,\sigma}_{\omega}(0,x,\theta))_{|\theta=\phi_0/\eps}\|_{L^2({\bR^n})} \ \text{ for } \ |\alpha|\leq 1.
\end{align*}
For each $\alpha$ the two terms inside $\|\cdot\|$ can be estimated separately using Proposition \ref{crude} and shown to be $o_\eps(1)$ in $L^2({\bR^n})$.  For example, we have $\nabla U^{\rho,\sigma}_{\omega}(0,x,\theta)=(\nabla W^{\rho,\sigma}_\omega)^*(0,x,\theta)$, and 
\begin{align*}
\begin{split}
\sqrt{\eps}\left\|(\nabla W^{\rho,\sigma}_\omega)^*(0,x,\phi_0/\eps)\right\|_{L^2(\bR^n)} & \lesssim \sqrt{\eps}\left\|(\nabla W^{\rho,\sigma}_\omega)^*(0,x,\theta)\right\|_{H^1(\bR^n\times \bR)}\\
&\lesssim \eps^{\frac{1}{2}-d}\left\|\nabla W^{\rho,\sigma}(0,x,\theta)\right\|_{H^1(\bR^n\times\bR)}=o_\eps(1).
\end{split}
\end{align*}
Here we used Proposition \ref{crude} for the first inequality and Proposition \ref{a10z}(iii) for the second.
Summarizing we have shown
\begin{align*}
\|(u^\eps-u^\eps_\ma)_{|t=0}\|_{H^1({\bR^n})}=o_{\rho}(1)+o_\sigma(1)+o_\eps(1).
\end{align*}

\textbf{11. Estimate $\|{\partial_t}(u^\eps-u^\eps_\ma)_{|t=0}\|_{L^2(\bR^n)}$.} At $t=0$ with 
\begin{align}\label{eahj}
V_1=\partial_\theta V_0\partial_t\phi=W\partial_t\phi=a\partial_t\phi,
\end{align}
 we have 
\begin{align*}%\label{eahk}
\partial_t(u^\eps-u^\eps_\ma)=\left(h^\eps-\epsilon^{-\frac12}V_1(0,x,\phi_0/\eps)\right)+\left(\epsilon^{-\frac12} V_1-\partial_t(\sqrt{\eps}U^{\rho,\sigma}_{\omega})\right) \eqqcolon A+B.  
\end{align*}
We have $\|A\|_{L^2({\bR^n})}=o_\eps(1)$ by assumption.   Using \eqref{eahj} and  the argument of step \textbf{10}, we find
$$\|B\|_{L^2({\bR^n})}=o_{\rho}(1)+o_\sigma(1)+o_\eps(1),$$
therefore
\begin{align*}%\label{eahl}
\|{\partial_t}(u^\eps-u^\eps_\ma)_{|t=0}\|_{L^2({\bR^n})}=o_{\rho}(1)+o_\sigma(1)+o_\eps(1).
\end{align*}

\textbf{12. Conclusion.} The estimates of steps \textbf{8}--\textbf{11}  imply  that for $0<d<1/2$:
\begin{align}\label{eaj}
\|u^\eps-u^\eps_\ma\|_{H^1(\tilde\Omega_T)}=\|u^\eps-\sqrt{\eps}U^\rho_{\eps^d}(t,x,\phi/\eps)\|_{H^1(\tilde\Omega_T)} = o_\rho(1)+o_\sigma(1)+o_\eps(1).  
\end{align}

Let $\delta_\ell\searrow 0$.   In view of Definition \ref{b}, to show that \eqref{b1zc} holds, we must explain how to choose the functions $W^\ell$ that appear in \eqref{b3}.  For each $\ell$ using \eqref{regdb}, \eqref{reghc}, and 
\eqref{eaj}, we can choose first $\rho(\ell)$, then $\sigma(\ell)$, and  then $\eps(\ell)$ such that 
\begin{align*}
\|W-W^{\rho(\ell),\sigma(\ell)}\|_{H^1(\tilde\Omega_T\times \bR)}\leq \delta_\ell,
\end{align*}
and such that the function $U^{\rho(\ell),\sigma(\ell)}_{\eps^d}$ constructed as the moment-zero $\theta$-primitive of  $W^{\rho(\ell),\sigma(\ell)}_{\eps^d}$ satisfies 
\begin{align*}%\label{eak}
\|u^\eps-\sqrt{\eps}U^{\rho(\ell)}_{\eps^d}(t,x,\phi/\eps)\|_{H^1(\Omega_T)}<\delta_\ell \text{ for }0<\eps<\eps(\ell).
\end{align*}
Thus, for each $\ell$ the function $W^{\ell} \coloneqq W^{\rho(\ell),\sigma(\ell)}$ satisfies the condition of Definition \ref{b}.
\end{proof}

\section{Diffraction of nonlinear  pulses.}\label{diff}

In this section we present the main result of this paper, Theorem \ref{mta}.    The theorem describes the diffraction of low regularity, large amplitude pulses propagating along rays that graze boundaries to high (possibly infinite) order.

Consider \eqref{in1} under the the assumptions in Theorem \ref{mta}.
The proof of existence of $u^\eps$ and the later error estimate use a classical estimate due to Sakamoto \cite{sakamoto3} for the following linear initial boundary value problem on $\Omega_T$
\begin{align*}
\begin{cases}
P(t,x,\partial_t,\nabla) u=f & (t,x)\in\Omega_T,\\
u(t,0,x')=b(t,x') &  (t,x')\in \slashed\Omega_T,\\
u(-T,x)=g(x), \ \partial_t u(-T,x)=h(x) & x\in \RR_+\times \RR^{n-1}.
\end{cases}
\end{align*}
Under the natural compatibility conditions on $(f,b,g,h)$, there exists $C>0$ such that for $t\in [-T,T]$ we have:\footnote{Here we have written $u(t)$ for the function of $x$ given by  $u(t,x)$.}
\begin{align}\label{sakamoto}
\begin{split}
&\|u(t)\|^2_{H^1(\overline{\bR^n_+})}+\|\partial_tu(t)\|^2_{L^2(\overline{\bR^n_+})}\\
&\qquad \qquad\leq C\left[\int^t_{-T}\|f(s)\|^2_{L^2(\overline{\bR^n_+})}\,\mathrm ds+\| b\|^2_{H^1(\slashed\Omega_{[-T,t]})}+\|g\|^2_{H^1(\overline{\bR^n_+})}+\|h\|^2_{L^2(\overline{\bR^n_+})}\right].
\end{split}
\end{align}

The  assumption  that $f_\zeta$ is  bounded  allows one to prove the existence of a unique solution $u^\eps\in H^1(\Omega_T)$ of \eqref{in1} by a standard  Picard iteration argument using the estimate \eqref{sakamoto}.  One obtains both existence and convergence of the iterates on $\Omega_T$ for small $\eps>0$.  The definition of $\sim_{H^1}$ plays no role in that proof.  The only restriction on $T$  comes from the size of the domain of existence of the phases $\phi_\mi$ and $\phi_\mr$.

\subsection{Profile equations}
We first remark some simplified notations.
\begin{nota}\label{notza}
\textup{Here are some abuses of notation that we often commit below.
\begin{align*}
\begin{split}
f(u^\eps_\ma)& \coloneqq f(t,x,u^\eps_\ma, \partial_t u_\ma^\epsilon, \nabla u^\eps_\ma),\\
{f}\left(\epsilon^{-\frac12}W_\mr,\epsilon^{-\frac12}W_\mi\right)& \coloneqq {f}\left(t,x,0,\epsilon^{-\frac12}W_\mr\nabla_{t,x}\phi_\mr+\epsilon^{-\frac12}W_\mi\nabla_{t,x}\phi_\mi\right),\\
f_{\mathrm{nc}}\left(\epsilon^{-\frac12}W_\mr,\epsilon^{-\frac12}W_\mi\right)& \coloneqq f\left(\epsilon^{-\frac12}W_\mr,\epsilon^{-\frac12}W_\mi\right)-f\left(\epsilon^{-\frac12}W_\mr,0\right)-f\left(0,\epsilon^{-\frac12}W_\mi\right).
\end{split}
\end{align*}
We think of $f_{\mathrm{nc}}$ as containing ``noncharacteristic oscillations" produced by nonlinear interaction of the incoming and reflected waves.}
\end{nota}

 Consider the profile equations
\begin{align}
    & \begin{cases}\label{nl3}
        T_{\phi_\mr}W_\mr+(P_1\phi_\mr)W_\mr=\sqrt{\eps}f(\epsilon^{-\frac12}W_\mr,0) & (t,x,\theta)\in\mathring{J_\mr}\times \mathbb{R},\\
        W_\mr(t,0,x',\theta)=-W_\mi(t,0,x',\theta) & (t,0,x',\theta)\in (J_\mr\cap\{x_1=0\})\times \mathbb{R}, \\
        W_\mr(t,x,\theta)=0 & (t,x,\theta)\in(\Omega_T\setminus J_\mr)\times \mathbb{R};
    \end{cases}\\
    & \begin{cases}\label{nl4}
        T_{\phi_\mi}W_\mi+(P_1\phi_\mi)W_\mi=\sqrt{\eps}f(0,\epsilon^{-\frac12}W_\mi) & (t,x,\theta)\in\mathring{J_\mi}\times \mathbb{R},\\
        W_\mi(-T,x,\theta)=a(x,\theta) & (x,\theta)\in (J_\mi\cap\{t=-T\})\times\RR,\\
        W_\mi(t,x,\theta)=0 & (t,x,\theta)\in(\Omega_T\setminus J_\mi)\times \mathbb{R}.
    \end{cases}
\end{align}
The sets $J_\mi$, $J_\mr$ were described in \eqref{e2a}, \eqref{e4} respectively.  The equations are coupled by the boundary condition.   Unlike the wavetrain case, the incoming profile $W_\mi$ can be obtained by solving \eqref{nl4} \emph{first}.   The equations \eqref{nl3}, \eqref{nl4} are derived by a computation similar to the one \eqref{b6b} that led to the profile equations \eqref{pe} for the problem in free space.

The difference between the pulse profile equations \eqref{nl3} and the wavetrain profile equations of \cite{wangwil} reflects the fact that pulse profiles have no well-defined \emph{mean} $\underline{u}(t,x)$.

\subsection{Diffraction theorem}\label{pmt}

\begin{proof}[Proof of Theorem \ref{mta}]
The proof consists of fifteen numbered steps given in the next two subsections.
\end{proof}

\subsubsection{Construction of profiles and the approximate solution $u^\eps_\ma$.}
\emph{}

\textbf{1. Linear $L^2$ estimates.} Consider the linear profile equations
\begin{align}
    & \begin{cases}\label{l3}
        T_{\phi_\mr}W_\mr+(P_1\phi_\mr)W_\mr=F_\mr& \text{in }\mathring{J_\mr}\times \mathbb{R},\\
        W_\mr(t,0,x',\theta)=-W_\mi(t,0,x',\theta) & \text{on } (J_\mr\cap\{x_1=0\})\times \mathbb{R}, \\
        W_\mr=0 & \text{on }(\Omega_T\setminus J_\mr)\times \mathbb{R};
    \end{cases}\\
    & \begin{cases}\label{l4}
        T_{\phi_\mi}W_\mi+(P_1\phi_\mi)W_\mi=F_\mi & \text{in }\mathring{J_\mi}\times \mathbb{R},\\
        W_\mi(-T,x,\theta)=a(x,\theta) & \text{on } (J_\mi\cap\{t=-T\})\times \mathbb R,\\
        W_\mi=0 & \text{on }(\Omega_T\setminus J_\mi)\times \mathbb{R}.
    \end{cases}
\end{align}
Here we suppose that $F_\mr$, $F_\mi$ have $(t,x)$-support in ${J_\mr}$, ${J_\mi}$, and that $a$ has  $x$-support in ${J_\mi}\cap \{t=-T\}$.

For $t_0\in [-T,T]$ we should expect $W_\mr$ on $J_\mr\cap\{t=t_0\}$ to be determined by the data $F_\mr$ and $W_\mi(t,0,x', \theta)$ of problem \eqref{l3} in $J_{\mr,t_0}\coloneqq J_\mr\cap \{t\leq t_0\}$.\footnote{The arguments below will make it clear that the trace on $t=t_0$ as well as traces on $x_1=0$ make sense.}   The boundary of $J_{\mr,t_0}$ consists of two flat pieces, one in $\{t=t_0\}$ and one in $\{x_1=0\}$, and a curved piece foliated by integral curves of $T_{\phi_\mi}$. 

For any $t_0\in[-T,T]$ we should expect $W_\mi$ on $J_\mi\cap\{t=t_0\}$ to be determined by the data $F_\mi$ and $g$ of problem \eqref{l4} in the set $J_{\mi,t_0}\subset J_\mi$, which we define as the backward flowout under $T_{\phi_\mi}$ in $\Omega_T$ of $J_\mi\cap\{t=t_0\}$.  The boundary of $J_{\mi,t_0}$ consists of two flat pieces, one in $\{t=t_0\}$ and one in $\{t=-T\}$, and a curved piece foliated by integral curves of $T_{\phi_\mi}$.    
 
We also define $V=J_\mi\cap \{x_1=0\}$ as in \S \ref{asmr}, and then define 
$J_{\mi,V,t_0}\subset J_\mi$ to be the backward flowout under $T_{\phi_\mi}$ in $\Omega_T$ of $V\cap\{t\leq t_0\}$.  The boundary of $J_{\mi,V,t_0}$ consists of two flat pieces, one in $\{x_1=0\}$ and one in $\{t=-T\}$, and a curved piece foliated by integral curves of $T_{\phi_\mi}$. 

The above  expectations are confirmed by the following energy estimates, whose proofs are  given in \cite{wangwil}.    For any $t_0\in [-T,T]$ we have:
\begin{subequations}\label{d12d}
\begin{align}
\label{d12da}
&\langle W_\mr,W_\mr\rangle_{t_0}\lesssim (F_\mr,F_\mr)+|((\partial_{x_1}\phi_\mi)W_\mi,W_\mi)_0| && \text{ on }J_{\mr,t_0}\times \mathbb{R},\\
\label{d12db}
&\langle W_\mi,W_\mi\rangle_{t_0}\lesssim (F_\mi,F_\mi)+\langle a,a\rangle_{-T}&&\text{ on }J_{\mi,t_0}\times \mathbb{R},\\
\label{d12dc}
&|((\partial_{x_1} \phi_\mi)W_\mi,W_\mi)_0|\lesssim (F_\mi,F_\mi)+\langle a,a\rangle_{-T}&&\text{ on }J_{\mi,V,t_0}\times \mathbb{R}. 
\end{align}
\end{subequations}
Here $\langle W_\mr,W_\mr\rangle_{t_0}$ denotes the $L^2$ inner product on $(J_{\mr,t_0}\times \mathbb{R})\cap \{t=t_0\}$, $(F_\mr,F_\mr)$ denotes the $L^2$ inner product on $J_{\mr,t_0}\times \bR$,  and $((\partial_{x_1}\phi_\mi)W_\mi,W_\mi)_0$ denotes the $L^2$ inner product on 
$\Omega^\flat_{[-T,t_0]}\times\RR_\theta$.  The other pairings are similarly defined.

The estimates \eqref{d12d} together with the  support conditions on $W_\mr$, $W_\mi$ imply
\begin{subequations}\label{d12e}
\begin{align}
\label{d12ea}
& \|W_\mi(t_0)\|_{L^2(\overline{\bR^n_+}\times \bR_\theta)} \lesssim \|F_\mi\|_{L^2(\Omega_{[-T,t_0]}\times\bR)}+\langle a\rangle_{L^2(\overline{\bR^n_+}\times \bR_\theta)} \\
\label{d12eb}
& \|W_\mr(t_0)\|_{L^2(\overline{\bR^n_+}\times \bR_\theta)} \lesssim \|F_\mr\|_{L^2(\Omega_{[-T,t_0]}\times\bR)}+\|F_\mi\|_{L^2(\Omega_{[-T,t_0]}\times\bR)}+\langle a\rangle_{L^2(\overline{\bR^n_+}\times \bR_\theta)}.
\end{align}
\end{subequations}
Here $\langle \cdot\rangle_{L^2(\overline{\RR^n_+}\times \RR_\theta)}$ is the $L^2$ norm on $\overline{\RR^n_+}\times \RR_\theta$.

\begin{rem} 
The estimates \eqref{d12d}, \eqref{d12e} were proved in \cite[\S\S 5.1--5.2]{wangwil} by, for example, taking the $L^2$ pairing of 
the interior equation in \eqref{l3} with $W_\mr$ and then integrating by parts using the Gauss-Green formula.  
One first approximates 
$F_\mr$,  $F_\mi$, and $a$ in $L^2$ norms  by smooth $F^k_\mr$, $F^k_\mi$, and  $a^k$ supported away from $\mathrm{SB}$.  To obtain inequalities \eqref{d12d}, \eqref{d12e} with implicit constants $C$ independent of $k$ we use the key observation that the bad term $((P\phi_\mr) W^k_\mr,W^k_\mr)_{L^2}$ in the energy estimates \emph{cancels out}, permitting us to take a limit as $k\to \infty$ in the estimates for $W^k_\mr$ and $W^k_\mi$.\footnote{The factor $P\phi_\mr$ generally blows up  near $\mathrm{SB}_+$ in a way that becomes more complicated as the order of tangency increases.} This cancellation is observed in \cite[(5.5), (5.6)]{wangwil}, and is essential for our treatment of higher order tangency.  The implicit constants in \eqref{d12e} depend only on the bounded $C^1$ norms of $\phi_\mi$ and $\phi_\mr$.
\end{rem}

\textbf{2. Nonlinear $L^2$ solutions.} 
(a)  The Lipschitz property of $f$ allows us to use Picard iteration to solve  the nonlinear profile problem \eqref{nl3}, \eqref{nl4}  on $J_\mr$, $J_\mi$ with $W_\mr$, $W_\mi$ in $L^2(\Omega_T\times \bR)$.   The estimates \eqref{d12e} imply
$W_\bullet\in C([-T,T],L^2(\bR^n_+\times \bR))$ for $\bullet=\mr, \mi$.

(b) 
The functions $W_\mr,W_\mi$ lie in  $L^2(\Omega_T\times \bR)$, have $(t,x)$-support in $J_\mr,J_\mi$, and satisfy the profile equations in the sense of distributions only on $J_\mr, J_\mi$.

\textbf{3. Linear $H^1$ estimates.}
In this step we take $a^\mu(x,\theta)$ to be any $C^\infty$ function with compact $(x,\theta)$-support such that $a^\mu$ vanishes in a $\mu$-neighborhood of $\mathrm{SB}\cap \{t=-T\}$.\footnote{A ``$\mu$-neighborhood'' of $\mathrm{SB}\cap \{t=-T\}$ is an open set $\cO_\mu\subset \overline{\bR^n_+}$ depending on $\mu>0$ small for which there exist positive constants $c_1$, $c_2$ such that 
$$c_1\mu\leq \mathrm{dist}(\mathrm{SB}\cap \{t=-T\}, \cO_\mu^c)\leq c_2\mu.$$}. 
For each $\mu$ we solve
\[\begin{split}
    & \begin{cases}
       T_{\phi_\mr}W^\mu_\mr+(P_1\phi_\mr)W^\mu_\mr=F^\mu_\mr  & (t,x,\theta)\in\Omega_T\times \bR\\
       W^\mu_\mr(t,0,x',\theta) =-W^\mu_\mi(t,0,x',\theta) & (t,x',\theta)\in [-T,T]\times \RR^{n-1}\times \RR,
    \end{cases}\\
    & \begin{cases}
        T_{\phi_\mi}W^\mu_\mi+(P_1\phi_\mi)W^\mu_\mi=F^\mu_\mi  & \qquad \ \ \ (t,x,\theta)\in\Omega_T\times \bR\\
       W^\mu_\mi(-T,x,\theta)=a^\mu(x,\theta)  & \qquad \ \ \ (x,\theta)\in \RR_+\times\RR^{n-1}\times\RR.
    \end{cases}
\end{split}\]
Here the $F^\mu_\mr$, $F^\mu_\mi\in H^1(\Omega_T\times \bR)$ have compact $(t,x)$-support contained respectively in $\mathring{J_\mr}$, $\mathring{J_\mi}$ and lying outside a $\mu$-neighborhood of $\mathrm{SB}$.
Solutions $W^\mu_{\mr,\mi}$ are obtained on $\Omega_T\times \bR$ by taking zero extensions; they  have support $(t,x)$-support in $\mathring{J_{\mr,\mi}}$ and lying outside a $\mu$-neighborhood of $\mathrm{SB}$.

The solutions $W^\mu_\mi$ satisfies the following estimates for $t_0\in[-T,T]$:
\begin{subequations}\label{d12gg}
\begin{align}
\label{d12gga}
\|W^\mu_\mi(t_0)\|_{H^1(\overline{\bR^n_{+}}\times\bR_\theta)} \lesssim & \left[ \|a^\mu\|_{H^1(\overline{\bR^n_+}\times \bR)}+\|F^\mu_\mi\|_{H^1(\Omega_{[-T,t_0]}\times\bR)}+\|W^\mu_\mi\|_{H^1(\Omega_{[-T,t_0]}\times\bR)}\right]\\
\|\partial_tW^\mu_\mi(t_0)\|_{L^2(\overline{\bR^n_+}\times \bR)} \lesssim & 
\left[ \|a^\mu\|_{H^1(\overline{\bR^n_+}\times \bR)}+\|F^\mu_\mi\|_{H^1(\Omega_{[-T,t_0]}\times\bR)}+\|W^\mu_\mi\|_{H^1(\Omega_{[-T,t_0]}\times\bR)}\right] \nonumber\\
\label{d12ggb}
& \quad +\|F^\mu_\mi(t_0)\|_{{L^2(\overline{\bR^n_+}\times \bR)}}.
\end{align}
\end{subequations}
The estimate \eqref{d12gga} is proved by differentiating the equation \eqref{l4} with $\nabla_{x,\theta}$ and applying the $L^2$ estimate \eqref{d12e}.  The term $\|W^\mu_\mi\|_{H^1(\Omega_{[-T,t_0]}\times\bR)}$ on the right arises from the commutators
$[T_{\phi_\mi},\nabla_{x,\theta}]W^\mu_\mi$ and $[P_1\phi_\mi,\nabla_{x,\theta}]W^\mu_\mi$.   
Estimate \eqref{d12ggb} follows from \eqref{d12gga} and the equation \eqref{l4}.   Observe that \eqref{d12ggb} gives control of the $L^2$ norm of $(\partial_t W^\mu_\mi)_{|t=-T}$. This is needed below in the estimate of $W^\mu_\mr$.

Combining the  estimates \eqref{d12gg} gives for $t_0\in [-T,T]:$
\begin{align}\label{d12gh}
\begin{split}
& \|W^\mu_\mi(t_0)\|_{H^1(\overline{\bR^n_{+}}\times\bR_\theta)}+\|\partial_tW^\mu_\mi(t_0)\|_{L^2(\overline{\bR^n_+}\times \bR)}\\
&\qquad \lesssim \left[ \|a^\mu\|_{H^1(\overline{\bR^n_+}\times \bR)}+\|F^\mu_\mi\|_{H^1(\Omega_{[-T,t_0]}\times\bR)}+\|W^\mu_\mi\|_{H^1(\Omega_{[-T,t_0]}\times\bR)}\right]+\|F^\mu_\mi(t_0)\|_{{L^2(\overline{\bR^n_+}\times \bR)}}.
\end{split}
\end{align}

From \eqref{d12da} we obtain for $t_0\in[-T,T]$:
\begin{align}\label{d12k}
 \|W_\mr^\mu(t_0)\|_{L^2(\overline{\bR^n_{+}}\times\bR_\theta)}\lesssim \|F^\mu_\mr\|_{L^2(\Omega_{[-T,t_0]}\times \bR)}+\left|((\partial_{x_1}\phi_\mi)W^\mu_\mi,W^\mu_\mi)_0\right|.
 \end{align}
Differentiating the $W^\mu_\mr$ problem with $\nabla_{t,x',\theta}$, we obtain from \eqref{d12k}: 
\begin{align}\label{d12l}
\begin{split}
& \|\nabla_{t,x',\theta}W_\mr^\mu(t_0)\|_{L^2(\overline{\bR^n_{+}}\times\bR_\theta)} \\
& \qquad \lesssim_\mu \|F^\mu_\mr\|_{H^1(\Omega_{[-T,t_0]}\times \bR)}+\|W^\mu_\mr\|_{H^1(\Omega_{[-T,t_0]}\times \bR)}+\left|((\partial_{x_1} \phi_\mi)\nabla_{t,x',\theta}W^\mu_\mi,\nabla_{t,x',\theta}W^\mu_\mi)_0\right|.
\end{split}
 \end{align}
Here the term $\|W^\mu_\mr\|_{H^1(\Omega_{[-T,t_0]}\times \bR)}$ arises from  commutators.\footnote{The commutator  $[T_{\phi_\mr},\nabla_{t,x',\theta}]W^\mu_\mr$  has a contribution of the form $(-2\nabla_{t,x',\theta}\partial_{x_1}\phi_\mr)\partial_{x_1} W^\mu_\mr$.   Second derivatives of $\phi_\mr$ are bounded on the support of $W^\mu_\mr$ but blow up near $\mathrm{SB}_+$.  Hence, the implicit constant in \eqref{d12l} depends on $\mu$.} 

To estimate $\partial_{x_1} W^\mu_\mr$  we use \eqref{d12l},  the first equation in \eqref{l3}, and Assumption \ref{exv}.  This yields
\begin{align}\label{d12m}
\begin{split}
\|\partial_{x_1} W_\mr^\mu(t_0)\|_{L^2(\overline{\bR^n_{+}}\times\bR_\theta)}
& \lesssim_\mu \|F^\mu_\mr\|_{H^1(\Omega_{[-T,t_0]}\times \bR)}+\|W^\mu_\mr\|_{H^1(\Omega_{[-T,t_0]}\times \bR)}\\
& \qquad +\left|((\partial_{x_1} \phi_\mi)\nabla_{t,x',\theta}W^\mu_\mi,\nabla_{t,x',\theta}W^\mu_\mi)_0\right|+\|F^\mu_\mr(t_0)\|_{L^2(\overline{\bR^n_{+}}\times\bR_\theta)}.
\end{split}
 \end{align}

To control the term $\left|((\partial_{x_1}\phi_\mi)\nabla_{t,x',\theta}W_\mi,\nabla_{t,x',\theta}W_\mi)_0\right|$ we differentiate the first equation in \eqref{l4} with $\nabla_{t,x',\theta}$ and apply \eqref{d12dc} to obtain:\footnote{Here, in \eqref{d12k}, and in \eqref{d12l},   the boundary pairing $(\cdot,\cdot)_0$ is on $x_1=0$, $-T\leq t_0$. }
\begin{align*}%\label{d12n}
\begin{split}
& \left|((\partial_{x_1}\phi_\mi)\nabla_{t,x',\theta}W_\mi,\nabla_{t,x',\theta}W_\mi)_0\right| \\
& \qquad \qquad \qquad \lesssim \|F^\mu_\mi\|_{H^1(\Omega_{[-T,t_0]}\times \bR)}+\|W^\mu_\mi\|_{H^1(\Omega_{[-T,t_0]}\times \bR)}+\|(\partial_tW^\mu_\mi)_{|t=-T}\|_{L^2(\overline{\bR^n_{+}}\times\bR_\theta)}.
\end{split}
\end{align*}

Combining \eqref{d12gh} with \eqref{d12l}, \eqref{d12m}, and \eqref{d12ggb} we have for $t_0\in [-T,T]$:
\begin{align}\label{d12o}
\begin{split}
&\|W^\mu_\mr(t_0)\|_{H^1(\overline{\bR^n_{+}}\times\bR_\theta)}+\|\partial_tW^\mu_\mr(t_0)\|_{L^2(\overline{\bR^n_+}\times \bR)} \\
& \qquad \qquad \lesssim_\mu \|a^\mu\|_{H^1(\overline{\bR^n_+}\times \bR)}+\|(F^\mu_\mi,F^\mu_\mr)\|_{H^1(\Omega_{[-T,t_0]}\times\bR)} \\
& \qquad \qquad \qquad \qquad +\|(W^\mu_\mi,W^\mu_\mr)\|_{H^1(\Omega_{[-T,t_0]}\times\bR)}+ \|(F^\mu_\mi(t_0),F^\mu_\mr(t_0))\|_{{L^2(\overline{\bR^n_+}\times \bR)}}.
\end{split}
\end{align}

\textbf{4. Truncation of initial datum. }Step \textbf{5} of the  proof uses a careful truncation of the initial datum $a(x,\theta)$ near the $(n-1)$-dimensional hypersurface
$\mathrm{SB}_{-T}\coloneqq \mathrm{SB}\cap \{t=-T\}$  in $\{t=-T\}$  given by the next lemma.

\begin{lem}\label{tr1}
Let $\cS_{-T}\coloneqq \{(-T,x) \mid \phi_\mi(-T,x)=0\}$, an $(n-1)$-dimensional hypersurface in $\{t=-T\}$.   For $T>0$ small enough we have
\begin{enumerate}
    \item[(i)] the intersection $\mathrm{SB}_{-T}\cap\cS_{-T}$ is transverse; that is, $\mathrm{SB}_{-T}\pitchfork\cS_{-T}$.

    \item[(ii)] for $\mu>0$ there exist functions $a^\mu\in H^1(\overline{\bR^n_+}\times\bR_\theta)$ vanishing in a $\mu$-neighborhood of $\mathrm{SB}_{-T}$ such that as $\mu\to 0$
\begin{align}\label{tr2}
a^\mu\to a \text{ in }L^2(\overline{\bR^n_+}\times\bR_\theta) \text{ and }  (a^\mu)_{|\cS_{-T}\times\bR}\to a_{|\cS_{-T}\times \bR} \text{ in } L^2(\cS_{-T}\times\bR_\theta);
\end{align}

    \item[(iii)] the $a^\mu$ can be taken to be $C^\infty$ with compact $(x,\theta)$-support.
\end{enumerate}

\end{lem}

\begin{rem}
\textup{The condition \eqref{tr2} is strictly weaker than the condition $a^\mu\to a$ in $H^1(\overline{\bR^n_+}\times\bR_\theta)$, which is impossible to arrange for $a^\mu$ vanishing on $\mathrm{SB}_{-T}$ unless the trace of $a$ on $\mathrm{SB}_{-T}$ also vanishes.  We do \emph{not} make that assumption on the trace of  $a$, since it  would significantly weaken our theorem.}
\end{rem}

\begin{proof}[Proof of Lemma \ref{tr1}]
\textbf{(i)} Transversality of the intersection for small $T>0$ follows by continuity from transversality at $T=0$, which we now establish. 
We note that the tangent space to $\cS_0$ at $0$ is 
$$\cT_0\cS_0=\{(0,x) \mid (0,x)\perp (\partial_t\phi_\mi(0,x), \nabla\phi_\mi(0,x))\}.$$
Since $\partial_{x_1} \phi_\mi(0)=0$, this tangent space includes vectors $(t,x)$ of the form $(0,x_1,0)$ with $x_1\neq 0$.  
On the other hand the tangent space to $\mathrm{SB}_0$ at $0$ is 
$$\cT_0(\mathrm{SB}_0)=\textrm{span} \left(\{(T_{\phi_\mi})(0)\}\cup \cT_0G_{\phi_\mi}\right).$$
We claim that every vector in $\cT_0(\mathrm{SB}_0)$ has $x_1$-component $0$.   That is clear for $T_{\phi_\mi}(0)$.   Moreover,
$$\cT_0G_{\phi_\mi}=\mathcal T_0(\{x_1=0\})\cap \mathcal T_0(\{\Xi=0\}),$$
for $\Xi$ as in Assumption~\ref{A02}.  This establishes the claim, and the desired transversality.

\textbf{(ii), (iii)} Choose $T>0$ so that $\mathrm{SB}_{-T}\pitchfork\cS_{-T}$, and choose a $C^\infty$ cutoff $\chi^\mu(x)$ vanishing in a $\mu$-neighborhood of $\mathrm{SB}_{-T}$, with $\chi^\mu=1$ outside a slightly larger $\mu$-neighborhood. Observe that $\chi^\mu\to 1$ almost everywhere in $\overline{\bR^n_+}\times\bR_\theta$ as $\mu\to 0$. Transversality of the above intersection implies $(\chi^\mu)_{|\cS_{-T}}\to 1$ a.e. with respect to Lebesgue measure on $\cS_{-T}$.   Thus, if we set $a^\mu\coloneqq \chi^\mu a$, then the condition \eqref{tr2} holds.   For the second limit in \eqref{tr2}  we used
\begin{align*}
(a^\mu)_{|\cS_{-T}}=(a_{|\cS_{-T}})\cdot (\chi^\mu)_{|\cS_{-T}}.
\end{align*}
For (iii) take $C^\infty$ functions $b^\mu$ with compact $(x,\theta)$-support converging to $a$ in $H^1(\overline{\bR^n_+}\times\bR_\theta)$ as $\mu\to 0$ and set $a^\mu\coloneqq \chi^\mu \cdot b^\mu$.
\end{proof}

\begin{rem}
The proof of Lemma \ref{tr1} shows there is no reason to expect the $a^\mu$ to be uniformly bounded in $H^1(\overline{\bR^n_+}\times\bR_\theta)$ for $\mu>0$ small.  Indeed, derivatives of $\chi^\mu$ blow up as $\mu\to 0$. 
\end{rem}

\textbf{5. Nonlinear $H^1$ solutions.} 
(a) For $a\in H^1(\overline{\bR^n_+}\times\bR_\theta)$ as in Theorem \ref{mta},   choose smooth $a^\mu\to a$ as in items (ii), (iii) of Lemma \ref{tr1}.  With the estimate \eqref{d12gh} and Picard iteration, for each $\mu$ we solve the nonlinear profile problem  \eqref{nl4} on $\Omega_T\times\bR$ with initial data $a^\mu$.  Denote the solutions by $W^\mu_\mi$.  By the argument of step \textbf{2} the solutions lie in $L^2(\Omega_T\times\bR)$.  The Lipschitz property of $f$, the estimate \eqref{d12gh}, and Gronwall's inequality imply that for any interval $[-T,t_0]$ with $t_0\in [-T,T]$ on which a solution $W^\mu_\mi$ exists, we have 
\begin{align*}
\begin{split}
&\|W^\mu_\mi(t_0)\|_{H^1(\overline{\bR^n_{+}}\times\bR_\theta)}+\|\partial_tW^\mu_\mi(t_0)\|_{L^2(\overline{\bR^n_+}\times \bR)}\lesssim \|a^\mu\|_{H^1(\overline{\bR^n_+}\times \bR)}+\sup_{t\in [-T,T]}\|W^\mu_\mi(t)\|_{L^2(\overline{\bR^n_+}\times \bR)}.
\end{split}
\end{align*}
This shows that for each $\mu>0$ the solution $W^\mu_\mi\in H^1(\Omega_T\times\bR)$.  The estimate \eqref{d12ea} implies $W^\mu_\mi\to W_\mi$ in $L^2(\Omega_T\times \bR)$ as $\mu\to 0$.

(b) The argument of step \textbf{2} yields solutions $W^\mu_\mr$ of the nonlinear ``$\mu$-version" of \eqref{nl3}, where $a$ is replaced by smooth $a^\mu$  that are uniformly bounded for small $\mu>0$ in $L^2(\Omega_T\times \bR_\theta)$.    
Using the Lipschitz property of $f$ and the estimate \eqref{d12o} we apply Gronwall's inequality to obtain for $t_0\in[-T,T]$:\footnote{Here we use Gronwall's inequality in the form: if $y$ and $\phi$ are continuous, nonnegative functions on $[-T,T]$ which satisfy for some constants $\alpha,C\geq 0$, 
$y(t)\leq C\left[\alpha+\int^t_{-T}\left(y(s)+\phi(s)\right) \,\mathrm ds\right] \text{ for }t\in [-T,T],$
then
$$y(t)\leq C\left[\alpha e^{C(t+T)}+\int^t_{-T}e^{C(t-s)}\phi(s) \,\mathrm ds\right].$$}  
\begin{align}\label{nl6u}
\begin{split}
&\|W^\mu_\mr(t_0)\|_{H^1(\overline{\bR^n_{+}}\times\bR_\theta)}+\|\partial_tW^\mu_\mr(t_0)\|_{L^2(\overline{\bR^n_+}\times \bR)}\\
 &\qquad  \lesssim_\mu \|a^\mu\|_{H^1(\overline{\bR^n_+}\times \bR)}+\|W^\mu_\mi\|_{H^1(\Omega_T\times\bR)}+\sup_{t\in [-T,T]}\|(W^\mu_\mi(t),W^\mu_\mr(t))\|_{L^2(\overline{\bR^n_+}\times \bR)}.
 \end{split}
 \end{align}
This implies that for each $\mu$,  $W^\mu_\mr$ lies in $H^1(\Omega_T\times\bR)$.   Due to the blow up of $P\phi_\mr$ near $\mathrm{SB}_+$, 
the implicit constant in \eqref{nl6u} depends on $\mu$.  
Moreover, the estimate \eqref{d12e} implies  $W^\mu_\mr\to W_\mr$  in $L^2(\Omega_T\times\bR)$ as $\mu\to 0$.  Note $(W^\mu_\mi,W^\mu_\mr)$ is supported outside a $\mu$-neighborhood of SB for each $\mu$.  

(c) Summarizing,  we have
\begin{subequations}\label{nl6}
\begin{align}
\label{nl6a}
& \|W_\mr-W^{\mu}_\mr\|_{L^2(\Omega_T\times \mathbb{R})}+\|W_\mi-W^{\mu}_\mi\|_{L^2(\Omega_T\times \mathbb{R})}=o_\mu(1),\\
\label{nl6b}
& W^\mu_\mr, \ W^\mu_\mi\in H^1(\Omega_T\times \mathbb{R}) \ \text{ for each } \ \mu.
\end{align}
\end{subequations}
The estimate  \eqref{nl6a} uses the cancellation of the $(P\phi_\mr)W_\mr$ term in the $L^2$ estimates.

\textbf{6.}   We need to regularize $W^{\mu}_\mr$ and $W^\mu_\mi$  in order to carry out the computation \eqref{b6b} and perform the error analysis.  
The regularization procedure of \cite{wangwil} does not work now.\footnote{Reason: The zero extension of $W^\mu_\mr+W^\mu_\mi\eqqcolon V^\mu$ is in $H^1$ but not $H^2$, which is what we'd need here to carry out that procedure.}
One challenge is to regularize in $x_1$ without destroying the boundary condition.  It turns out that while we can't preserve the boundary condition exactly,  we can  preserve it nearly enough to carry out the error analysis.

Regularize $W^{\mu}$ first using $R^{\rho}$, a tangential regularization in $(t,x',\theta)$ only.
Writing $R^{\rho}G=G^{\rho}$,  this gives:
\begin{align*}%\label{newa3}
T_{\phi_\mr}W^{\mu,\rho}_\mr+(P_1\phi_\mr)W^{\mu,\rho}_\mr=\sqrt{\eps}f\left(\epsilon^{-\frac12}W^{\mu}_\mr,0\right)^{\rho}+[T_{\phi_\mr},R^{\rho}]W^{\mu}_\mr+[P_1\phi_\mr,R^{\rho}]W^{\mu}_\mr.
\end{align*}
With $\slashed{\Omega}_T\coloneqq [-T,T]\times \bR^{n-1}_{x'}$ we have
\begin{align*}%\label{newa2}
W_\mr^\mu\in H^1(\Omega_T\times \bR)\subset L^2(\overline{\bR_+},H^1(\slashed\Omega_T\times \bR_\theta)) \ \Rightarrow \ W^{\mu,\rho}_\mr\in L^2(\overline{\bR_+},H^\infty(\slashed\Omega_T\times \bR_\theta)).
\end{align*}
With Friedrich's lemma we see for each fixed $\mu>0$:
\begin{align}\label{newa6}
T_{\phi_\mr}W^{\mu,\rho}_\mr+(P_1\phi_\mr)W^{\mu,\rho}_\mr= \sqrt{\eps}f\left(\epsilon^{-\frac12}W^{\mu}_\mr,0\right)+o_\rho(1) \ \text{ in } \ H^1(\Omega_T\times \bR).
\end{align}
A similar statement applies to $W^{\mu,\rho}_\mi$; moreover, at this stage the boundary condition is preserved:\footnote{To regularize $W^\mu_\mi$ we can use an approximate identity with $t$-support in $t\leq 0$, so that $W^{\mu,\rho}_\mi$ depends just on $W^\mu_\mi$ in $t\geq -T$.} 
\begin{align*}
W^{\mu,\rho}_\mr+W^{\mu,\rho}_\mi=0 \ \text{ on } \ x_1=0.
\end{align*}

\textbf{7.} 
Now regularize in $x_1$ using an approximate identity $\delta_\sigma(x_1)\in C^\infty_c(\bR)$ \emph{supported in} $x_1\leq 0$.\footnote{This implies that for $x_1\geq 0$, $R^\sigma f(x_1)\coloneqq (f*\delta_\sigma)(x_1)$ depends only on $f_{|x_1\geq 0}$.}  We obtain\footnote{See Remark \ref{notz} for clarification of the notation in \eqref{nl7}.}
\begin{subequations}\label{nl7}
 \begin{align}
        \label{nl7a}
        & T_{\phi_\mr}W^{{\mu},\rho,\sigma}_\mr+(P_1\phi_\mr)W^{{\mu},\rho,\sigma}_\mr = \sqrt{\eps}f\left(\epsilon^{-\frac12}W^{\mu}_\mr,0\right)+o_{\rho}(1) +o_\sigma(1) && \text{ in }H^1(\Omega_T\times \bR),\\
        \label{nl7b}
        & (W^{{\mu},\rho,\sigma}_\mr)_{|x_1=0}=-(W^{{\mu},\rho,\sigma}_\mi)_{|x_1=0} =o_\sigma(1) && \text{ in }H^1(\Omega_T^\flat\times \bR), 
    \end{align}
\end{subequations}
and
\begin{subequations}\label{nl7z}
\begin{align}
\label{nl7za}
T_{\phi_\mi}W^{{\mu},\rho,\sigma}_\mi+(P_1\phi_\mi)W^{{\mu},\rho,\sigma}_\mi
& =  \sqrt{\eps}f\left(0,\epsilon^{-\frac12}W^{\mu}_\mi\right)+o_{\rho}(1)+o_\sigma(1) && \text{ in }H^1(\Omega_T\times \bR),\\
\label{nl7zb1}
(W^\mu_\mi)_{|t=-T} & =a+o_\mu(1) && \text{ in }L^2(\overline{\bR^n_+}\times\bR_\theta),\\
\label{nl7zb2}
(W^\mu_\mi)_{|\cS_{-T}} & =a_{|\cS_{-T}}+o_\mu(1) && \text{ in }L^2(\cS_{-T}\times\bR_\theta),\\
\label{nl7zc}
(W^{{\mu},\rho,\sigma}_\mi)_{|t=-T} & =a^\mu+o_{\rho}(1)+o_\sigma(1) && \text{ in }H^1(\overline{\bR^n_+}\times\bR_\theta).
\end{align}
\end{subequations}
The functions $W^{\mu,\rho,\sigma}_\mr$, $W^{\mu,\rho,\sigma}_\mi$ are $C^\infty$ with compact support in $(t,x,\theta)$ and the $(t,x)$-support lies strictly outside a $\mu$-neighborhood of SB.   
Here \eqref{nl7a} and \eqref{nl7za}  follow from Friedrichs's lemma, and \eqref{nl7b} follows from Lemma \ref{tl}.
For \eqref{nl7b} we used \eqref{newa6} and the fact that $\partial_{x_1}\phi_\mr\neq 0$ on the support of $W^{\mu,\rho}_\mr$ to get the extra regularity in $x_1$ needed in Lemma \ref{tl} and Remark \ref{tl3}. 
Our choice of the $a^\mu$ in  Lemma \ref{tr1} implies \eqref{nl7zb1}, \eqref{nl7zb2}, while \eqref{nl7zc} depends on Remark \ref{tl3} and is further explained in step \textbf{12} below.

\begin{rem}\label{notz}
(i) The right side of \eqref{nl7a} tells us, for example,  that for fixed $\mu$ and $\rho$, the quantity $o_{\sigma}(1)\to 0$ in $H^1(\Omega_T\times \bR)$ 
as $\sigma\to 0$, uniformly with respect to small $\eps>0$.   Similarly for fixed $\mu$, the quantity $o_\rho(1)\to 0$ in $H^1(\Omega_T\times \bR)$  
as $\rho\to 0$, uniformly with respect to small $\sigma>0$, $\eps>0$.   For example, one of the contributions to the $o_\sigma(1)$ term in \eqref{nl7a} is $[T_{\phi_\mr},R^\sigma]W^{\mu,\rho}.$ 
Each term on the right in \eqref{nl7a} depends on $\mu$, $\rho$, $\sigma$, and $\eps$,  and must be interpreted with respect to the ordering of parameters given by $\mu,\rho,\sigma, \eps$.  

(ii) For each $(\mu,\rho,\sigma)$ the functions $W^{\mu,\rho,\sigma}_{\mr,\mi}$ are $C^\infty$ with compact $(t,x,\theta)$-support,
have support outside a $\mu$-neighborhood of the curved part of the boundary of $J_{\mr,\mi}$ and of $\mathrm{SB}_-$ for small $(\rho,\sigma)$, and  satisfy the modified profile equations \eqref{nl7}, \eqref{nl7z} in the classical sense on $\Omega_T\times \bR$.
\end{rem}

\textbf{8.} For $\bullet=\mr, \mi$, let $U_\bullet^{{\mu},\rho,\sigma}(t,x,\theta)=\int^\theta_{-\infty}W_\bullet^{{\mu},\rho,\sigma}(t,x,s)\,\mathrm ds$.  Note that $U_\bullet^{{\mu},\rho,\sigma}$ is not necessarily even in $L^2(\Omega_T\times \bR)$,  since it does not necessarily decay as  $\theta\to +\infty$.
So we first take a moment-zero approximation $W^{{\mu},\rho,\sigma}_{\bullet,\omega}$ ($\omega=\eps^d$, $0<d<1/2$), and then take $U^{{\mu},\rho,\sigma}_{\bullet, \omega}$ to be the unique moment-zero primitive in $\theta$ of $W^{{\mu},\rho,\sigma}_{\bullet,\omega}$.  For fixed $\omega$,
$U^{{\mu},\rho,\sigma}_{\bullet,\omega}$ is $C^\infty$ with compact support in $(t,x)$, and is rapidly decaying in $\theta$. 

\textbf{9.} Define our approximate solution to the original continuation/boundary problem by
\begin{align*}
u^\eps_\ma(t,x)=\sqrt{\eps}U^{{\mu},\rho,\sigma}_{\mr,\eps^d}(t,x,\theta)_{|\theta=\frac{\phi_r}{\eps}}+\sqrt{\eps}U^{{\mu},\rho,\sigma}_{\mi,\eps^d}(t,x,\theta)_{|\theta=\frac{\phi_i}{\eps}}
\end{align*}
We proceed to use the  Sakamoto estimate \eqref{sakamoto} to show that $\|u^\eps-u^\eps_\ma\|_{H^1(\Omega_T)}$ is small for appropriately chosen $\mu,\rho,\sigma,\eps$.

\subsubsection{Error analysis}
\emph{}

\textbf{10. }Recall that $u^\eps$ and $u^\eps_\ma$ are both defined on $\Omega_T$.   Writing 
$$Pu^\eps-Pu^\eps_\ma=(f(u^\eps)-f(u^\eps_\ma))+(f(u^\eps_\ma)-P(u^\eps_\ma))$$ and applying \eqref{sakamoto} we obtain for $-T<t \leq T$:
\begin{align*}%\label{ea1}
\begin{split}
&\|(u^\eps-u^\eps_\ma)(t)\|^2_{H^1(\overline{\bR^n_+})}+\|\partial_t(u^\eps-u^\eps_\ma)(t)\|^2_{L^2(\overline{\bR^n_+})}\\
&\qquad \lesssim \left[\int^t_{-T}\|(f(u^\eps)-f(u^\eps_\ma))(s)\|^2_{L^2(\overline{\bR^n_+})}\,\mathrm ds+\int^t_{-T}\|(f(u^\eps_\ma)-Pu^\eps_\ma))(s)\|^2_{L^2(\overline{\bR^n_+})}\,\mathrm ds\right]\\
&\qquad\qquad + \left[\langle u^\eps-u^\eps_\ma\rangle^2_{H^1(\Omega^\flat_{[-T,t]})}+\|(u^\eps-u^\eps_\ma)_{|t=-T}\|^2_{H^1(\overline{\bR^n_+})}+\|\partial_t(u^\eps-u^\eps_\ma)_{|t=-T}\|^2_{L^2(\overline{\bR^n_+})}\right]\\
&\qquad\qquad\qquad \eqqcolon [A+B]+[C+D+E].
\end{split}
\end{align*}
Using the Lipschitz property of $f$, we can apply Gronwall's inequality to absorb $A$ and obtain
\begin{align*}
\begin{split}
&\|(u^\eps-u^\eps_\ma)(t)\|^2_{H^1(\overline{\bR^n_+})}+\|\partial_t(u^\eps-u^\eps_\ma)(t)\|^2_{L^2(\overline{\bR^n_+})}\\
&\qquad \qquad \lesssim\int^t_{-T}\|(f(u^\eps_\ma)-Pu^\eps_\ma))(s)\|^2_{L^2(\overline{\bR^n_+})}\,\mathrm ds +\langle u^\eps-u^\eps_\ma\rangle^2_{H^1(\Omega^\flat_{T})}\\
&\qquad\qquad \qquad \qquad  + \|(u^\eps-u^\eps_\ma)_{|t=-T}\|^2_{H^1(\overline{\bR^n_+})}+\|\partial_t(u^\eps-u^\eps_\ma)_{|t=-T}\|^2_{L^2(\overline{\bR^n_+})}.
\end{split}
\end{align*}
Thus, it remains to estimate 
\begin{align*}
\|Pu^\eps_\ma-f(u^\eps_\ma)\|_{L^2(\Omega_T)}, \ \|(u^\eps-u^\eps_\ma)_{|t=-T}\|_{H^1(\overline{\bR^n_+})}, \ \|\partial_t(u^\eps-u^\eps_\ma)_{|t=-T}\|_{L^2(\overline{\bR^n_+})}, \text{ and }\langle u^\eps-u^\eps_\ma\rangle_{H^1(\Omega^\flat_{T})}.
\end{align*}

\textbf{11. Estimate $\|Pu^\eps_\ma-f(u^\eps_\ma)\|_{L^2(\Omega_T)}$.}   Let $\omega=\eps^d$, where $0<d<1/2$.  The profiles $U^{\mu,\rho,\sigma}_{\bullet,\omega}$, $\bullet=\mi, \mr$,  are regular enough so that the computation \eqref{b6b} is valid.   Thus we have
\begin{align}\label{eaa2}
\begin{split}
 Pu^\eps_\ma = & \frac{1}{\sqrt{\eps}}\left[T_{\phi_\mi} W^{\mu,\rho,\sigma}_{\mi,\omega}+(P_1\phi_\mi)W^{\mu,\rho,\sigma}_{\mi,\omega}\right]_{|\theta=\frac{\phi_\mi}{\eps}}+\sqrt{\eps}(PU^{\mu,\rho,\sigma}_{\mi,\omega})_{|\theta=\frac{\phi_\mi}{\eps}}\\
 & \qquad +\frac{1}{\sqrt{\eps}}\left[T_{\phi_\mr} W^{\mu,\rho,\sigma}_{\mr,\omega}+(P_1\phi_\mr)W^{\mu,\rho,\sigma}_{\mr,\omega}\right]_{|\theta=\frac{\phi_r}{\eps}}+\sqrt{\eps}(PU^{\mu,\rho,\sigma}_{\mr,\omega})_{|\theta=\frac{\phi_\mr}{\eps}}.
    \end{split}
\end{align}
Since $PU^{\mu,\rho,\sigma}_{\mi,\omega}=\left(PW^{\mu,\rho,\sigma}_{\mi,\omega}\right)^*$ and $PW^{\mu,\rho,\sigma}_{\mi,\omega}\in H^1(\Omega_T\times\bR_\theta)$, using Proposition \ref{a10z} we obtain
\begin{align}\label{eab2}
\begin{split}
& \sqrt{\eps}\left\|(PU^{\mu,\rho,\sigma}_{\mi,\omega})_{|\theta=\frac{\phi_\mi}{\eps}}\right\|_{L^2(\Omega_T)} =\sqrt{\eps}\left\|\left(PW^{\mu,\rho,\sigma}_{\mi,\omega}\right)^*_{|\theta=\frac{\phi_i}{\eps}}\right\|_{L^2(\Omega_T)} \leq \sqrt{\eps}\left\|\left(PW^{\mu,\rho,\sigma}_{\mi,\omega}\right)^*\right\|_{H^1(\Omega_T\times\bR_\theta)} \\
& \qquad \qquad \leq\sqrt{\eps}\frac{\|PW^{\mu,\rho,\sigma}_{\mi,\omega}\|_{H^1(\Omega_T\times\bR_\theta)}}{\omega} \leq \sqrt{\eps}\frac{\|PW^{\mu,\rho,\sigma}_\mi\|_{H^1(\Omega_T\times\bR_\theta)}}{\eps^d}=O(\eps^{\frac{1}{2}-d})
\end{split}
\end{align}
for fixed $\rho$.
Here we used Proposition \ref{crude} for the first inequality and Proposition \ref{a10z} for the other two inequalities.   The last term in \eqref{eaa} is estimated the same way.

Next consider $f(u^\eps_\ma).$ Define $A$ and $r_\eps$ by 
\begin{align}\label{eac2}
\begin{split}
f(u^\eps_\ma)= & f\left(t,x,\sqrt{\eps}\sum_{\bullet=\mi,\mr}U^{\mu,\rho,\sigma}_{\bullet,\omega},\sum_{\bullet = \mi, \mr} \left(\sqrt{\eps}\nabla_{t,x} U^{\mu,\rho,\sigma}_{\bullet,\omega}+\epsilon^{-\frac12}W^{\mu,\rho,\sigma}_{\bullet,\omega}\nabla_{t,x} \phi_\bullet \right)\right)_{|\theta_i=\frac{\phi_i}{\eps},\theta_r=\frac{\phi_r}{\eps}}\\
= & f\left(t,x,0,\epsilon^{-\frac12}W^{\mu,\rho,\sigma}_{\mi,\omega}\nabla_{t,x} \phi_\mi+\epsilon^{-\frac12}W^{\mu,\rho,\sigma}_{\mr,\omega}\nabla_{t,x} \phi_\mr\right)+r_\eps \eqqcolon A+r_\eps.
\end{split}
\end{align}
Since $f$ is Lipschitz, we have as $\epsilon\to 0$
\begin{align}\label{ead2}
\|r_\eps\|_{L^2(\Omega_T)}\lesssim \sqrt{\eps}\|U^{\mu,\rho,\sigma}_{\mi,\omega}+U^{\mu,\rho,\sigma}_{\mr,\omega}\|_{L^2(\Omega_T)}+\sqrt{\eps}\|\nabla_{t,x} U^{\mu,\rho,\sigma}_{\mi,\omega}+\nabla_{t,x} U^{r,\rho,\sigma}_{\mr,\omega}\|_{L^2(\Omega_T)}\to 0 
\end{align}
by an estimate just like \eqref{eab2}.   Using Notation \ref{notza} we can write
\begin{align}\label{ead3}
A=f\left(\epsilon^{-\frac12}W^{\mu,\rho,\sigma}_{\mr,\omega},0\right)+f\left(0,\epsilon^{-\frac12}W^{\mu,\rho,\sigma}_{\mi,\omega}\right)+f_{\mathrm{nc}}\left(\epsilon^{-\frac12}W^{\mu,\rho,\sigma}_{\mr,\omega},\epsilon^{-\frac12}W^{\mu,\rho,\sigma}_{\mi,\omega}\right),
\end{align}
where by Lemma \ref{nc}:
\begin{align}\label{ead4}
\left\|f_{\mathrm{nc}}\left(\epsilon^{-\frac12}W^{\mu,\rho,\sigma}_{\mr,\omega},\epsilon^{-\frac12}W^{\mu,\rho,\sigma}_{\mi,\omega}\right)\right\|_{L^2(\Omega_T)}=o_\eps(1).
\end{align}

From \eqref{eaa2}--\eqref{ead4} we see that $\|Pu^\eps_\ma-f(u^\eps_\ma)\|_{L^2(\Omega_T)}\leq B_\mi+B_\mr$, where\footnote{The term $B_\mr$ is the obvious analogue of $B_\mi$.}
\begin{align*}%\label{eae2}
    \begin{split}
        B_\mi & = \left\|\epsilon^{-\frac12}\left(T_{\phi_\mi} W^{\mu,\rho,\sigma}_{\mi,\omega}+(P_1\phi_\mi)W^{\mu,\rho,\sigma}_{\mi,\omega}-\sqrt{\eps}f\left(0,\epsilon^{-\frac12}W^{\mu,\rho,\sigma}_{\mi,\omega}\right)\right)_{|\theta=\phi_\mi/\eps}\right\|_{L^2(\Omega_T)}+o_\eps(1)\\
        & \lesssim\left\|\epsilon^{-\frac12}\left(T_{\phi_\mi} W^{\mu,\rho,\sigma}_{\mi,\omega}+(P_1\phi_\mi)W^{\mu,\rho,\sigma}_{\mi,\omega}-\sqrt{\eps}f\left(0,\epsilon^{-\frac12}W^\mu_\mi\right)\right)\right\|_{L^2(\Omega_T)}\\
        &\qquad\qquad \qquad \qquad \qquad \qquad+ \left\|f\left(0,\epsilon^{-\frac12}W^\mu_\mi\right)-f\left(0,\epsilon^{-\frac12}W^{\mu,\rho,\sigma}_{\mi,\omega}\right)\right\|_{L^2(\Omega_T)}+o_\eps(1)\\
        &\lesssim\left\|T_{\phi_\mi} W^{\mu,\rho,\sigma}_{\mi,\omega}+(P_1\phi_\mi)W^{\mu,\rho,\sigma}_{\mi,\omega}-\sqrt{\eps}f\left(0,\epsilon^{-\frac12}W^\mu_\mi\right)\right\|_{H^1(\Omega_T\times\bR)}\\
        & \qquad \qquad \qquad \qquad \qquad \qquad \qquad \qquad \qquad \quad \,\, +\epsilon^{-\frac12}\left\|W^\mu_\mi-W^{\mu,\rho,\sigma}_{\mi,\omega}\right\|_{L^2(\Omega_T)}+o_\eps(1)\\
        & \lesssim (o_\rho(1)+o_\sigma(1)+o_\omega(1))+\left\|W^\mu_\mi-W^{\mu,\rho,\sigma}_{\mi,\omega}\right\|_{H^1(\Omega_T\times\bR)}+o_\eps(1)\\
        & =o_\rho(1)+o_\sigma(1)+o_\eps(1).
    \end{split}
\end{align*}
Here we used \eqref{a11} and the Lipschitz property of $f$ to obtain the fourth line.  We used \eqref{nl7z} and \eqref{a11}, along with Friedrichs's lemma and Proposition \ref{a10z}  to get the last line.   The estimate of $B_\mr$ is done the same way.  Summarizing, we have shown
\begin{align*}%\label{eae3}
\|Pu^\eps_\ma-f(u^\eps_\ma)\|_{L^2(\Omega_T)}=o_\rho(1)+o_\sigma(1)+o_\eps(1).
\end{align*}

\textbf{12. Estimate $\|(u^\eps-u^\eps_\ma)_{|t=-T}\|_{H^1(\overline{\bR^n_+})}$.}   At $t=-T$ with $V_0=V_0(x,\phi_0/\eps)$ we have 
\begin{align*}%\label{ea4}
u^\eps-u^\eps_\ma=(g^\eps-\sqrt{\eps}V_0)+(\sqrt{\eps} V_0-\sqrt{\eps}U^{\mu,\rho,\sigma}_{\mi,\eps^d})\eqqcolon A+B.  
\end{align*}
Recall that $\|A\|_{H^1(\overline{\bR^n_+})}=o_\eps(1)$ by assumption.  The contribution to $\|B\|_{H^1(\overline{\bR^n_+})}$ whose estimate is least obvious is, with $\omega=\eps^d$:\footnote{Here $W^{\mu,\rho_{x',\theta},\sigma}_\mi$ is the regularization of $W^\mu_\mi$  in $(x,
\theta)$; we set $f^{\rho_{x',\theta}} \coloneqq (R^{\rho_{x'}}\circ R^{\rho_\theta})f$.   Observe that $W^{\mu,\rho_{x',\theta},\sigma,\rho_t}_\mi=W^{\mu,\rho,\sigma}_\mi$.}
\begin{align*}%\label{ea5}
\begin{split}
&\epsilon^{-\frac12}\|W_\mi(-T,x,\phi_0/\eps)\nabla \phi_0-W^{\mu,\rho,\sigma}_{\mi,\omega}\nabla\phi_0\|_{L^2(\overline{\bR^n_+})}\\
& \qquad   \leq\epsilon^{-\frac12}\|W_\mi\nabla \phi_0-W^{\mu,\rho_{x',\theta},\sigma}_{\mi}\nabla \phi_0\|_{L^2(\overline{\bR^n_+})}\\
& \qquad \qquad  +\epsilon^{-\frac12}\|W^{\mu,\rho_{x',\theta},\sigma}_\mi\nabla \phi_0-W^{\mu,\rho_{x',\theta},\sigma,\rho_t}_{\mi}\nabla \phi_0\|_{L^2(\overline{\bR^n_+})}\\
&\qquad \qquad \qquad  +\epsilon^{-\frac12}\|W^{\mu,\rho_{x',\theta},\sigma,\rho_t}_\mi\nabla \phi_0-W^{\mu,\rho_{x',\theta},\sigma,\rho_t}_{\mi,\omega}\nabla \phi_0\|_{L^2(\overline{\bR^n_+})}\\
& \qquad \qquad \qquad \qquad \eqqcolon C+D+E.
\end{split}
\end{align*}
We have
\begin{align*}%\label{ea5y}
\begin{split}
C & \lesssim \epsilon^{-\frac12} \|W_\mi-W_\mi^\mu\|_{L^2(\overline{\bR^n_+})}+\epsilon^{-\frac12} \|W_\mi^\mu-W_\mi^{\mu,\rho_{x',\theta},\sigma}\|_{L^2(\overline{\bR^n_+})}\\
&= \left(\left\|\frac{a-a^\mu}{|\nabla\phi_0|^{1/2}}\right\|_{L^2(\cS_{-T}\times \bR_\theta)}+o_\eps(1)\right)+\epsilon^{-\frac12}\|W_\mi^\mu-W_\mi^{\mu,\rho_{x',\theta},\sigma}\|_{L^2(\overline{\bR^n_+})} \eqqcolon C_1+C_2,
\end{split}
\end{align*}
where we used  Proposition \ref{a20} and Remark \ref{a25} for the first equality.  The choice of $a^\mu$ then implies that $C_1=o_\mu(1)+o_\eps(1)$.  By Remark \ref{a11kz} and Friedrichs's lemma, we have $C_2=o_{\rho_{x',\theta}}(1)+o_\sigma(1)$.
Thus,
\begin{align*}%\label{ea5x}
C= o_\mu(1)+o_\rho(1)+o_\sigma(1)+o_\eps(1).
\end{align*}

Next consider $D$, where  $W^{\mu,\rho_{x',\theta},\sigma,\rho_t}_{\mi}(-T,x,\theta)$ is obtained by first regularizing $W^{\mu,\rho_{x',\theta},\sigma}_{\mi}$ in $t$ and then restricting to $t=-T$.   Remark \ref{tl3} implies that
$$\|W_\mi^{\mu,\rho_{x',\theta},\sigma}-W^{\mu,\rho_{x',\theta},\sigma,\rho_t}_{\mi}\|_{H^1(\overline{\bR^n_+}\times \bR_\theta)}=o_{\rho_t}(1) \ \text{ at } \ t=-T,$$
since $W_\mi^{\mu,\rho_{x',\theta},\sigma}(t,x,\theta)\in H^1([-T,T],H^\infty(\overline{\bR^n_+}\times \bR_\theta))$.   Thus, $D=o_{\rho_t}(1)$ by Remark  \ref{a11kz}.   By Remark \ref{a11kz} again and Proposition \ref{a10z}(ii) we have $E=o_\eps(1)$.

To complete the estimate of $\|(u^\eps-u^\eps_\ma)_{|t=-T}\|_{H^1(\overline{\bR^n_+})}$ we must estimate the contribution to 
$\|B\|_{H^1(\overline{\bR^n_+})}$ given by 
\begin{align*}%\label{ea6}
\|\partial_{x}^\alpha(\sqrt{\eps} V_0-\sqrt{\eps}U^{\mu,\rho,\sigma}_{\mi,\eps^d})(-T,x,\theta)_{|\theta=\phi_0/\eps}\|_{L^2(\overline{\bR^n_+})} \ \text{ for } \ |\alpha|\leq 1.
\end{align*}
For each $\alpha$ the two terms inside $\|\cdot\|$ can be estimated separately using Proposition \ref{crude} and shown to be $o_\eps(1)$ in $L^2(\overline{\bR^n_+})$.  Summarizing we have shown
\begin{align*}
\|(u^\eps-u^\eps_\ma)_{|t=-T}\|_{H^1(\overline{\bR^n_+})}=o_\mu(1)+o_{\rho}(1)+o_\sigma(1)+o_\eps(1).
\end{align*}

\textbf{13. Estimate $\|\partial_t(u^\eps-u^\eps_\ma)_{|t=-T}\|_{L^2(\overline{\bR^n_+})}$.} At $t=-T$ with 
\begin{align}\label{ea7}
V_1=(\partial_\theta V_0)\partial_t\phi_\mi=W_\mi\partial_t\phi_\mi=a\partial_t\phi_\mi,
\end{align}
 we have 
\begin{align*}%\label{ea8}
\partial_t(u^\eps-u^\eps_\ma)=\left(h^\eps-\epsilon^{-\frac12}V_1(x,\phi_0/\eps)\right)+\left(\epsilon^{-\frac12} V_1-\partial_t(\sqrt{\eps}U^{\mu,\rho,\sigma}_{\mi,\eps^d})\right) \eqqcolon A+B.  
\end{align*}
We have $\|A\|_{L^2(\overline{\bR^n_+})}=o_\eps(1)$ by assumption.   Using \eqref{ea7} and  the argument of step \textbf{12}, we find
$$\|B\|_{L^2(\overline{\bR^n_+})}=o_\mu(1)+o_{\rho}(1)+o_\sigma(1)+o_\eps(1),$$
therefore
\begin{align*}
\|\partial_t(u^\eps-u^\eps_\ma)_{|t=-T}\|_{L^2(\overline{\bR^n_+})}=o_\mu(1)+o_{\rho}(1)+o_\sigma(1)+o_\eps(1).
\end{align*}

\textbf{14. Estimate $\langle u^\eps-u^\eps_\ma\rangle_{H^1(\Omega^\flat_{T})}$.} On $\Omega^\flat_T$ we have 
$u^\eps=0$.  The main contribution to $\langle u^\eps_\ma\rangle_{H^1(\Omega^\flat_{T})}$ is 
\begin{align*}
\left\langle \epsilon^{-\frac12} \left(W^{\mu,\rho,\sigma}_{\mr,\eps^d}+W^{\mu,\rho,\sigma}_{\mi,\eps^d}\right)\nabla_{t,x'}\phi_\mi\right\rangle_{L^2(\Omega^\flat_T)} \eqqcolon A.
\end{align*}
Here we used $\phi_\mr=\phi_\mi$ on $x_1=0$.  By Remark \ref{a11kz} and \eqref{nl7b}:
\begin{align*}
A\lesssim \left\|W^{\mu,\rho,\sigma}_{\mr,\eps^d}+W^{\mu,\rho,\sigma}_{\mi,\eps^d}\right\|_{H^1(\Omega^\flat_T\times \bR_\theta)}\
\leq \left\|W^{\mu,\rho,\sigma}_{\mr}+W^{\mu,\rho,\sigma}_{\mi}\right\|_{H^1(\Omega^\flat_T\times \bR_\theta)}=o_\sigma(1).
\end{align*}
Proposition \ref{crude} implies that the other contributions to $\langle u^\eps_\ma\rangle_{H^1(\Omega^\flat_{T})}$ are $o_\eps(1)$, so 
\begin{align*}
\langle u^\eps-u^\eps_\ma\rangle_{H^1(\Omega^\flat_{T})}=o_\sigma(1)+o_\eps(1).
\end{align*}

\textbf{15. Choice of $W^\ell_r$, $W^\ell_i$.} The estimates of steps \textbf{9-13} show that 
\begin{align*}
\|u^\eps-u^\eps_\ma\|_{L^2(\Omega_T)}=o_\mu(1)+o_{\rho}(1)+o_\sigma(1)+o_\eps(1).
\end{align*}
 Take $\delta_\ell\searrow 0$ and set $W^\ell_\bullet\coloneqq W^{{\mu(\ell)}, \rho(\ell),\sigma(\ell)}_\bullet$, $\bullet=\mr, \mi$.   
For each  $\ell$ use the above error estimates to choose $\mu(\ell)$, then $\rho(\ell)$, then $\sigma(\ell)$, and finally $\eps_\ell$ so that \eqref{nl10} and \eqref{nl11} hold.

\begin{rem}
\textup{ We saw in \eqref{ead4} that the nonlinear interaction term  $f_{\mathrm{nc}}\left(\epsilon^{-\frac12}W^{{\mu},\rho,\sigma}_{\mr,\omega},\epsilon^{-\frac12}W^{{\mu},\rho,\sigma}_{\mi,\omega}\right)$ is negligible since it is  $o_\eps(1)$ in $L^2(\Omega_T)$.  In the wavetrain case the interaction term was of size $\sim 1$ and had to be solved away by including an additional corrector term in the formula for $u^\eps_\ma$.   See (7.13) of \cite{wangwil}.}

\end{rem}

\section{Verification of Assumption \ref{exv} on the reflected phase}\label{rphase}

In this section we verify that Assumption \ref{exv} is satisfied when the operator $P$ is the wave operator acting in the exterior of a convex obstacle.

\subsection{A model case}
Consider   the wave operator $\Box \coloneqq -\partial_t^2 + \Delta$ acting  in the exterior region $\{x\in \RR^2 \mid x_1>-x_2^2\}$. We make a spatial change of variable 
\begin{align}\notag
\begin{split}
y_1 = x_1+x_2^2, \ \  y_2=G(x)
\end{split}
\end{align}
and $G$ is to be chosen below.
Then 
\[ \eta_1\mathrm d y_1+\eta_2 \mathrm d y_2 = \eta_1(\mathrm d x_1+2x_2\mathrm d x_2)+\eta_2(G_{x_1}\mathrm d x_1 + G_{x_2}\mathrm d x_2) = (\eta_1 + G_{x_1}\eta_2)\mathrm dx_1 + (2x_2\eta_1+G_{x_2} \eta_2)\mathrm d x_2. \]
Hence using coordinates $(t,y,\tau,\eta)$, the principal symbol of $\Box$ reads
\[ -\tau^2+(\eta_1+G_{x_1}\eta_2)^2 + (2x_2\eta_1+G_{x_2} \eta_2)^2. \]
Assuming that $G$ solves 
\[ (\partial_{x_1}+2x_2\partial_{x_2})G(x)=0 \ \text{ for } x_1>-x_2^2, \ \text{ and } \ G(-x_2^2,x_2)=x_2, \]
then the mixed terms in the principal symbol vanish, and we have 
\[ \sigma(\Box) = -\tau^2 + (1+4x_2^2)\eta_1^2 + (G_{x_1}^2+G_{x_2}^2)\eta_2^2. \]
Now $\frac{\sigma(\Box)}{1+4x_2^2}$ is of the form \eqref{d1}.
The method of characteristics says that $G$ can be solved by putting 
\[ G(s-z^2, z e^{2s}) = z, \ \ z\in \mathbb R. \]
According to \eqref{e4yb}, to show $\partial_{y_1}\phi_\mr\neq 0$ for all $(t,y)\in \mathring{J}_\mr$, it suffices to show $\eta_1(Z_r^{-1}(t,y))\neq 0$. Equivalently, one only needs to show 
\[ \eta_1(s,t,y_2)\neq 0, \ \text{for all } \ (s,t, y_2)\in [0, s_0)\times \mathring{V}_\mr. \]
By definition, we know $\eta_1(0,t, y_2)=-\partial_{y_1}\phi_\mi(0,t, y_2)>0$. Notice also that $\eta_1$ solves the equation
\[ \frac{\mathrm d \eta_1}{\mathrm d s} = -\partial_{y_1}(1+4x_2^2)\eta_1^2 -\partial_{y_1}(G_{x_1}^2+G_{x_2}^2)\eta_2^2 = -\partial_{y_1}(1+4x_2^2)\eta_1^2-\partial_{y_1}( (1+4x_2^2)G_{x_2}^2 ) \eta_2^2. \]
Notice that 
\[ \partial_{y_1} = -\frac{ G_{x_1}\partial_{x_2} - G_{x_2}\partial_{x_1} }{ (1+4x_2^2)G_{x_2} } = \frac{ \partial_{x_1} + 2x_2\partial_{x_2} }{1+4x_2^2}. \]
This implies that 
\[\begin{split} 
\partial_{y_1}(1+4x_2^2) & = \frac{16x_2^2}{1+4x_2^2}, \\
\partial_{y_1}( (1+4x_2^2)G_{x_2}^2 ) & = \frac{(\partial_{x_1}+2x_2\partial_{x_2})( (1+4x_2^2)G_{x_2}^2 )}{1+4x_2^2} \\
& = \frac{ 16x_2^2G_{x_2}^2 +2(1+4x_2^2)G_{x_2}( G_{x_1x_2} + 2x_2 G_{x_2x_2} ) }{1+4x_2^2} = -\frac{4 G_{x_2}^2}{1+4x_2^2}. 
\end{split}\]
Therefore
\[ \frac{\mathrm d \eta_1}{\mathrm ds} + \frac{16 x_2^2}{1+4x_2^2} \eta_1^2 = \frac{4G_{x_2}^2}{1+4x_2^2}\eta_2^2\geq 0 \ \text{ for all } \ (s,t, y_2)\in [0, s_0)\times \mathring{V}_\mr. \]
Denote $\varpi(s)\coloneqq  \frac{16 x_2^2}{1+4x_2^2}$, then $0\leq \varpi\leq 4$ for all $s$. 
We claim that  $\eta_1\neq 0$ for all $s\in [0, s_0)$. In fact, if this is not true for some fixed $(t,y_2)$, then there exists a minimal $s_*\in (0,s_0)$ such that $\eta_1(s_*)=0$ while $\eta_1(s)> 0$ for $s\in (0, s_*)$. Thus for any $0<\eps<s_*$, we have  
\[ \frac{1}{\eta_1^2}\frac{\mathrm d\eta_1}{\mathrm d s} + \varpi(s)\geq 0, \ s\in [0, s_*-\epsilon]. \]
Integrate this inequality in $s\in [0, s_*-\epsilon]$ and we obtain 
\[\begin{split} 
& -\frac{1}{\eta_1(s_*-\epsilon)}+\frac{1}{\eta_1(0)} + \int_0^{s_*-\epsilon}\varpi(s)\mathrm ds \geq 0 \ \Rightarrow \ \frac{1}{\eta_1(s_*-\epsilon)}\leq \frac{1}{\eta_1(0)} + 4s_*<\infty.
\end{split}\]
Taking $\epsilon\to 0+$ yields contradiction as $\lim_{\epsilon\to 0+}\frac{1}{\eta_1(s_*-\epsilon)}=+\infty$. This shows that $s_*$ does not exist. Hence $\eta_1\neq 0$ on $[0, s_0)$.

Since the Hamiltonian flows of $\sigma(\Box)$ and $\frac{\sigma(\Box)}{1+4x_2^2}$ are the same up to reparametrization, we conclude that $\partial_{y_1}\phi_\mr\neq 0$ on $\mathring J_\mr$.

\subsection{General convex obstacles}
More  generally, consider the wave operator $\Box \coloneqq -\partial_t^2+\Delta$ acting in the exterior region $\RR^n\setminus \overline{\mathcal O}$ for $\mathcal O\coloneqq\{x\in \RR^n \mid x_1<-F(x'), \  x'\in \mathbb R^{n-1}\}$, where  $F$ is convex. We introduce the spatial change of variable 
\begin{align}\notag
\begin{split}
&y_1 = x_1+F(x'),\  \ y' = G(x),
\end{split}
\end{align}
where $G\in C^{\infty}(\mathbb R^n\setminus \overline{\mathcal O}; \mathbb R^{n-1})$ and satisfies the equation 
\begin{align}\label{gc1}
 G_{x_1}^\top + \snabla F ( \snabla G)^\top =0, \ G(-F(x'),x')=x'.
 \end{align}
Here $\snabla$ is the gradient in $x'$ and is regarded as a row vector. All other vectors in this section, when written without superscript $\top$,  are regarded as column vectors.   The system \eqref{gc1} is a decoupled system of transport equations for the components of $G$, so can be solved by integrating along characteristics.

In the new coordinates the principal symbol of $\Box$ is\footnote{Here we write $\langle a,b\rangle$ for the dot product of two column vectors.   We will write $a\cdot b$ for the dot product of two row vectors.}
\[\begin{split} 
\sigma(\Box) & = -\tau^2+(1+|\snabla F|^2)\eta_1^2 + |\langle G_{x_1}, \eta' \rangle|^2+ |(\snabla G)^\top\eta'|^2 \\
& = (1+|\snabla F|^2)\left( \eta_1^2 + \frac{|\langle G_{x_1}, \eta' \rangle|^2 + |(\snabla G)^\top \eta'|^2-\tau^2}{1+|\snabla F|^2} \right).
\end{split}\]
Thus 
\[ p(t,y,\tau,\eta) \coloneqq \eta_1^2 + \frac{|\langle G_{x_1}, \eta' \rangle|^2 +|(\snabla G)^\top\eta'|^2-\tau'^2}{1+|\snabla F|^2},  \]
which has the same form of \eqref{d1}. Since the flow of $p$ and the flow of $\sigma(\Box)$ are the same up to reparametrization, it suffices to show $\eta_1$ for $\sigma(\Box)$ does not vanish. Notice
\[ \frac{\mathrm d \eta_1}{\mathrm d s} = -\partial_{y_1}(1+|\snabla F|^2)\eta_1^2 -\partial_{y_1}\left( |\langle G_{x_1}, \eta' \rangle|^2+ |(\snabla G)^\top \eta'|^2 \right).\]
Notice that by the change of coordinates,
\[\begin{split} 
\partial_{x_1} & = \partial_{y_1}+  ( G_{x_1})^\top\cdot \snabla_{y'} , \\
\snabla_{x'} & = (\snabla_{x'} F)\partial_{y_1} + ( (\snabla_{x'} G)^\top \snabla_{y'}^\top )^\top.
\end{split}\]
Taking inner product of the second identity and $\slashed\nabla F$:
\[ \snabla F\cdot \snabla_{x'} = (\snabla F\cdot  \snabla F)\partial_{y_1} + \snabla F (\snabla G)^\top \cdot\snabla_{y'}. \]
Summing this with the formula for  $\partial_{x_1}$, we find 
\[ \partial_{x_1}+\snabla F\cdot \slashed\nabla_{x'} = (1+|\snabla F|^2)\partial_{y_1} + \left((G_{x_1})^\top +\snabla F(\snabla G)^\top \right)\cdot \snabla_{y'} = (1+|\snabla F|^2)\partial_{y_1} \]
using the defining equation for $G$. Therefore 
\[ \partial_{y_1} = \frac{ \partial_{x_1} + \snabla F\cdot \snabla }{1+|\snabla F|^2}. \]

Now we can write 
\[ \frac{\mathrm d \eta_1}{\mathrm d s} = -\frac{ (\partial_{x_1}+\snabla F\cdot \snabla)(1+|\snabla F|^2)\eta_1^2  }{1+|\snabla F|^2} + \frac{ \langle \eta', \mathcal A \eta' \rangle }{1+|\snabla F|^2} \]
where 
\[ \mathcal A \coloneqq -(\partial_{x_1}+\snabla F \cdot \snabla)\left( \snabla G (I+(\snabla F)^\top \snabla F) (\snabla G)^\top \right). \]

We first compute
\[\begin{split} 
(\partial_{x_1}+\snabla F \cdot \snabla)(1+|\snabla F|^2) & = \snabla F\cdot \snabla (|\snabla F|^2) \\
& = \sum_{k,\ell} F_k(|F_\ell|^2)_k = 2\sum_{k,\ell}F_k F_\ell F_{k\ell} = 2\snabla F(\snabla^2 F)(\snabla F)^\top. 
\end{split}\]
where the subscripts $2\leq k, \ell\leq n$ are partial derivatives with respect to $x_k, x_\ell$. 

Our next aim is to show $\mathcal A \geq 0$ remembering $G_{x_1} + (\snabla G)(\snabla F)^\top =0$. We compute 
\[\begin{split}
    & (\partial_{x_1}+\snabla F\cdot \snabla)\snabla G = (\snabla F\cdot \snabla )\snabla G - \snabla (\snabla F\cdot \snabla G) = -(\snabla G)(\snabla^2 F), \\
    & (\partial_{x_1}+\snabla F\cdot \snabla)(\snabla G)^\top = (\snabla F\cdot \snabla )(\snabla G)^\top - ( \snabla(\snabla F\cdot \snabla G) )^\top = -(\snabla^2F)(\snabla G)^\top.
\end{split}\]
Therefore
\[ \mathcal A = \snabla G\left[ 2\snabla^2 F + \underbrace{\snabla^2F (\snabla F)^\top \snabla F+(\snabla F)^\top \snabla F \snabla^2 F - \snabla F\cdot \snabla ( (\snabla F)^\top \snabla F )}_{=0} \right] (\snabla G)^\top. \]
The underbraced part can be justified by looking at the $(i,j)$-th element of the matrix:
\[ \sum_k F_{ik}F_j F_k + \sum_k F_{jk}F_i F_k - \sum_k F_k (F_i F_j)_k =0. \]
We conclude that 
\[ \mathcal A = 2 (\snabla G)(\snabla^2 F)(\snabla G)^\top \]
is symmetric and $\mathcal A\geq 0$.

As a result 
\[ \frac{\mathrm d \eta_1}{\mathrm d s} + \frac{ 2\snabla F (\snabla^2 F)(\snabla F)^\top }{1+|\snabla F|^2} \eta_1^2\geq 0, \ \text{ for all } s\in [0, s_0) \]
with $\eta_1(0,t,y')>0$ by definition. For fixed $(t,y')$, let us denote 
\[ \varpi(s)\coloneqq \frac{ 2\snabla F (\snabla^2 F)(\snabla F)^\top }{1+|\snabla F|^2}, \]
and let $C>0$ be a constant that bounds the eigenvalues of $\snabla^2 F$. Then 
\[ 0\leq \varpi(s)\leq 2C \ \text{ for all } \ s \]
and 
\[ \frac{\mathrm d \eta_1}{\mathrm d s} + \varpi(s) \eta_1^2\geq 0 \  \text{ for all } \ s\in [0, s_0). \]
We can now run the same argument as in the model case to show $\eta_1\neq 0$ for all $s\in [0, s_0)$. In fact, if this is not true, then there exists a minimal $s_*\in (0,s_0)$ such that $\eta_1(s_*)=0$ while $\eta_1(s)> 0$ for $s\in (0, s_*)$. Thus for any $\epsilon>0$, we have
\[ \frac{1}{\eta_1^2}\frac{\mathrm d \eta_1}{\mathrm d s} + \varpi(s)\geq 0, \ \ s\in (0, s_*-\epsilon). \]
Integrate this inequality for $s\in (0, s_*-\epsilon)$ and we find 
\[\begin{split}
    & -\frac{1}{\eta_1(s_*-\epsilon)} + \frac{1}{\eta_1(0)} + \int_0^{s_*-\epsilon} \varpi(s)\mathrm ds \geq 0 \\
    & \qquad \qquad \Rightarrow \frac{1}{\eta_1(s_*-\epsilon)}\leq \frac{1}{\eta_1(0)} + \int_0^{s_*} \varpi(s)\mathrm ds\leq \frac{1}{\eta_1(0)} + 2Cs_*.
\end{split}\]
This contradicts the assumption that $\lim_{\epsilon\to 0+}\frac{1}{\eta_1(s_*-\epsilon)}=+\infty$.

We have now proved $\eta_1(s,t,y')>0$ on $[0,s_0)\times \mathring V_\mr$. This in turn shows 
\[ \partial_{y_1}\phi_\mr>0  \ \text{ in } \ \mathring J_\mr \]
considering \eqref{e4yb}.

\appendix

\section{Moment-zero approximations}\label{mz}
 Moment zero approximations are introduced to deal with the difficulty presented by the fact that primitives in $\theta$ of functions that decay as $|\theta|\to \infty$ do not themselves necessarily decay at all as $|\theta|\to \infty$; see Remark \ref{rmza}.\footnote{We first saw moment-zero approximations used in the context of pulses in \cite{altermanrauch}.}
 
 \begin{defn}\label{mz1}
 If $\sigma(t,x,\theta)$ is integrable in $\theta$, we call $\int^\infty_{-\infty}\sigma(t,x,\theta)\,\mathrm d\theta=(2\pi)^{1/2}\hat\sigma(t,x,0)$, the zeroth moment, or just \emph{the moment}, of $\sigma$.  Here $\hat\sigma(t,x,k)=\frac{1}{2\pi}\int e^{-ik\theta}\sigma(t,x,\theta)\,\mathrm d\theta.$
 \end{defn}

This motivates the following definitions.  
\begin{defn}\label{mz1a}
If $u(\theta)$ is a tempered distribution (that is, $u\in\cS'(\bR)$),  we say $u$ has \emph{moment zero} if $\hat u(k)$ vanishes in a neighborhood of $k=0$.  
\end{defn}

When $u\in \cS'(\bR)$ has moment zero, we can define its unique moment zero primitive $u^*\in \cS'(\bR)$ by 
$\widehat{u^*}=\frac{\hat u}{ik}$.

\begin{defn}[Moment zero approximations]\label{mz2}

For $\omega >0$ let $\chi_\omega(k)\in C^\infty(\bR)$  be such that $0\leq \chi_\omega\leq 1$ and 
$$\chi_\omega(k)=
\begin{cases}0 &\text{ on }|k|\leq \omega, \\ 
1 &\text{ on }|k|\geq 2\omega.
\end{cases}$$
\begin{enumerate}
    \item[(i)] For $u\in\cS'(\bR)$ define the moment zero approximation $u_\omega\in \cS'(\bR)$ of $u$  by $\widehat{u_\omega}=\chi_\omega \hat u$.

    \item[(ii)] Suppose $\sigma(t,x,\theta)\in H^s(\Omega_T\times \bR)$ for some $s\in \bN_0$.
\end{enumerate}   
Define the \emph{moment-zero approximation} $\sigma_\omega(t,x,\theta)$ by 
\begin{align*}
\hat\sigma_\omega(t,x,k) \coloneqq \chi_\omega(k)\hat\sigma(t,x,k).
\end{align*}
\end{defn}

\begin{rem}\label{rmza}
(i) The function $\chi_\omega\to 1$ pointwisely a.e. on $\bR$ as $\omega\to 0$.

(ii) If $u\in\cS'(\bR)$, then the unique moment zero primitive of $u_\omega$, $u_\omega^*\in\cS'(\bR)$, is given by  $\widehat{u_\omega^*}=\chi_\omega(k)\frac{\hat{u}}{ik}$.    Moreover if $v=\partial_\theta u=\partial_\theta \tilde u$ for some $u, \tilde u\in \cS'(\bR)$, we have 
\begin{align*}
u_\omega=\tilde u_\omega=v_\omega^*.
\end{align*}

(iii) If $\sigma\in H^s(\Omega_T\times \bR)$, where $s\geq 0$, then a primitive in $\theta$ of $\sigma$ need not lie in $L^2(\Omega_T\times \bR)$.   However, the unique moment-zero primitive in $\theta$ of $\sigma_\omega$, $\sigma^*_\omega$,  is given by 
\begin{align}\label{mzz}
\widehat{\sigma^*_\omega}(t,x,k) \coloneqq \chi_\omega(k)\frac{\hat{\sigma}(t,x,k)}{ik}
\end{align}
and satisfies  Proposition \ref{a10z}(iii) below.

(iv) 
If $\sigma$ is smooth in $\theta$ and compactly supported in $\theta$,   
then $\sigma_\omega(t,x,\theta)$ and $\sigma^*_\omega(t,x,\theta)$ are smooth and rapidly decaying in $\theta$.\footnote{The latter applies to 
the $W^{\rho,\sigma}_\omega$ in \S \ref{npfs} and the $W^{\mu,\rho,\sigma}_{\bullet,\omega}$, $\bullet=\mi, \mr$, in \S \ref{diff}. }

\end{rem}

The following proposition  is used repeatedly in \S\S \ref{npfs}--\ref{diff}.

\begin{prop}\label{a10z}
Suppose $\sigma(t,x,\theta)\in H^s(\Omega_T\times \bR)$ for some $s\in \bN_0$ and let $\omega>0$.   Then
\begin{enumerate}
    \item[(i)] $\|\sigma_\omega\|_{H^s(\Omega_T\times \bR)}\leq \|\sigma\|_{H^s(\Omega_T\times \bR)}$;

    \item[(ii)] $\lim_{\omega\to 0}\|\sigma_\omega-\sigma\|_{H^s(\Omega_T\times \bR)}=0$;

    \item[(iii)]\label{a10z3} $\|\sigma^*_\omega\|_{H^s(\Omega_T\times \bR)}\lesssim \frac{\|\sigma\|_{H^s(\Omega_T\times \bR)}}{\omega}$.
\end{enumerate}
\end{prop}

\begin{proof}
Part (i) is immediate from Definition \ref{mz2}, and part (ii) follows from Remark \ref{rmza}(i) and the dominated convergence theorem.  Part (iii) follows from the formula \eqref{mzz}.
\end{proof}

\begin{rem}
\textup{Proposition \ref{a10z}(iii) implies, for example, that 
 $$\lim_{\eps\to 0}\sqrt{\eps}\|\sigma^*_\omega\|_{H^s(\Omega_T\times \bR)}=0,$$
   if one takes $\omega=\eps^d$ for some $0<d<1/2$.}
\end{rem}

\section{Some technical tools}\label{tools}

We gather here a few of the technical tools we need, starting with Friedrichs's lemma.

\begin{lem}[Friedrichs's lemma, \cite{CP}]\label{fl}  
  Let $P$ be a differential operator of order $m\in \bN_0$ on $\bR^n$ with $C^\infty$ coefficients that are constant outside some compact set. Let $\delta_\rho$ be a $C^\infty$ approximate identity with support contained in $\{x\in \bR^n \mid |x|\leq \rho\}$, and define $R^\rho u=\delta_\rho * u$.    %Writing $\|u\|_{H^s(\bR^n)}=\|u\|_s$ where $s\in \bN_0$, 
  Then we have:
  \begin{enumerate}
      \item[(i)] $\|R^\rho u\|_{H^s(\RR^n)}\leq C\|u\|_{H^s(\RR^n)}$ and $\|R^\rho u-u\|_{H^s(\RR^n)}\to 0$ as $\rho\to 0$. 

      \item[(ii)] $\|[P,R^\rho]u\|_{H^s(\RR^n)}\leq C\|u\|_{H^{s+m-1}(\RR^n)}.$

      \item[(iii)] $[P,R^\rho]u\to 0$ in $H^s(\RR^n)$ as $\rho\to 0$, if $u\in H^{s+m-1}(\RR^n).$
  \end{enumerate}
\end{lem}
 
 The following  trace estimate is used, for example, in the proof of Proposition \ref{a11z}.
 
\begin{prop}[Trace estimate]\label{trace}
 Let $a(x,y)\in H^1(x,L^2(y))$.  
 Then $$\|a(0,y)\|_{L^2(y)}\lesssim \|a(x,y)\|_{H^1(x,L^2(y))}.$$  
\end{prop}
  
\begin{proof}
  We have%\footnote{Here $\mathrm d'\xi=(2\pi)^{-1/2}\,\mathrm d\xi$.}
\begin{align*}
  \begin{split}
  &a(x,y)=\frac{1}{2\pi}\int\hat a(\xi,y)e^{ix\xi}\,\md \xi\\
  & \qquad \Rightarrow a(0,y)= \frac{1}{2\pi}\int\hat a(\xi,y)\,\md \xi= \frac{1}{2\pi}\int \langle\xi\rangle \hat a(\xi,y)\langle\xi\rangle^{-1}\,\md \xi\\
  &\qquad \Rightarrow |a(0,y)|\leq C\|\langle \xi\rangle \hat a(\xi,y)\|_{L^2(\xi)} \\
  & \qquad \Rightarrow \|a(0,y)\|_{L^2(y)}\leq C\|a(x,y)\|_{H^1(x,L^2(y))}.
  \end{split}
\end{align*}
This is the desired inequality.
\end{proof}

The following Sobolev interpolation inequality is used in step \textbf{3} of the proof of Theorem \ref{mta}.
%Let $H^s=H^s(\bR^n)$.
\begin{prop}\label{interpolate}
Let $0\leq m\leq k$ and set $\theta=m/k$.  We have
\begin{align}\label{inter}
\|u\|_{H^m(\RR^n)}\leq \|u\|^\theta_{H^k(\RR^n)}\|u\|_{L^2(\RR^n)}^{1-\theta}.
\end{align}
\end{prop}

\begin{proof}
We can suppose $0<m<k$, since the cases $m=0$, $m=k$ are clear. By H\"older's inequality we have 
\begin{align*}
\begin{split}
\int |\hat u(\xi)|^{2}\langle \xi\rangle^{2m} \,\mathrm d\xi & =\int \left(|\hat u(\xi)|^{2m/k}\langle\xi\rangle^{2m}\right)\left(|\hat u(\xi)|^{2(1-\frac{m}{k})}\right) \,\mathrm d\xi\\
& \leq \left\||\hat u(\xi)|^{2m/k}\langle\xi\rangle^{2m}\right\|_{L^p(\RR^n)} \left\||\hat u(\xi)|^{2(1-\frac{m}{k})}\right\|_{L^q(\RR^n)}, 
\end{split}
\end{align*}
where $p=\frac{k}{m}, q=\frac{k}{k-m}$,
which is equivalent to \eqref{inter}.
\end{proof}

The next lemma estimates the regularization of a nonlinear function of $W$ and is used, for example, in step \textbf{4} of the proof of Theorem \ref{fnp}.
Recall the notation $\tilde\Omega_T=[0,T]_t\times \bR^n_x$.   For an approximate identity $\rho=(\rho_x,\rho_\theta)$ we define the regularization $g^\rho$ of $g$ by $g^\rho \coloneqq R^{\rho_x}(R^{\rho_\theta}g)$.\footnote{Recall, for example, that $R^{\rho_\theta}g=\delta_{\rho_\theta}*g$, where $\delta_{\rho_\theta}(\theta)$ is an approximate identity.}
\begin{lem}\label{rl}
Let $f$ and $\phi$ be as in Theorem \ref{fnp}.
\begin{enumerate}
    \item[(i)] If $W\in H^1(\tilde\Omega_T\times\bR_\theta)$, then there exists $M>0$ such that for $\epsilon$ small
\begin{align*}
\sqrt{\eps}\left\|f(t,x,0,\epsilon^{-\frac12}W\nabla_{t,x} \phi)\right\|_{H^1(\tilde\Omega_T\times \bR_\theta)}\leq M.
\end{align*}

    \item[(ii)] For each $\rho=(\rho_x,\rho_\theta)$
\begin{align}\label{toold}
\sqrt{\eps}f(t,x,0,\epsilon^{-\frac12}W\nabla_{t,x} \phi)^\rho\in H^1([0,T],H^\infty(\bR^n_x\times \bR_\theta)).
\end{align}
\end{enumerate}
\end{lem}

\begin{proof}
\textbf{(i)} Boundedness of the $L^2$ norm follows immediately from the Lipschitz property of $f$.  
%Letting $\partial$ denote the partial with respect to $t$ or $x$, 
Defferentiate with respect to $(t,x)$ we obtain
\begin{align*}%\label{care2}
\begin{split}
&\nabla_{t,x}\left[\sqrt{\eps}f\left(t,x,0,\epsilon^{-\frac12}W\nabla_{t,x} \phi\right)\right] \\ 
& \qquad \qquad =\sqrt{\eps}(\nabla_{t,x} f)\left(t,x,0,\epsilon^{-\frac12}W\nabla_{t,x} \phi\right)+f_\zeta\left(t,x,0,\epsilon^{-\frac12}W\nabla_{t,x} \phi\right)\nabla_{t,x}(W\nabla_{t,x}\phi).
\end{split}
\end{align*}
Since $\nabla_{t,x} f$ is Lipschitz in $\zeta$ and $f_\zeta$ is bounded, the $L^2$ norm of the right side is bounded; recall Assumption \ref{amr1z}.  A similar estimate
for $\partial_\theta$ yields part (i).

\textbf{(ii)} By item (i)
\begin{align*}
\sqrt{\eps}f(t,x,0,\epsilon^{-\frac12} W\nabla_{t,x} \phi)\in H^1([0,T],L^2((\bR^n_x\times \bR_\theta)) \text{ uniformly for }\eps\text{ small}.
\end{align*}
Since
\begin{align*}
\sqrt{\eps}f(t,x,0,\epsilon^{-\frac12}W\nabla_{t,x} \phi)^\rho=\delta_{\rho_x}*_x\left(\delta_{\rho_\theta}*_\theta\sqrt{\eps}f(t,x,0,\epsilon^{-\frac12}W\nabla_{t,x} \phi)\right),
\end{align*}
\eqref{toold} follows.
\end{proof}

The next lemma estimates the effect of regularization in $x_1$ on traces at $x_1=0$.   Of course, a close analogue applies as well to the effect of regularization in $t$ on traces at $t=-T.$   We apply it in the proofs of Theorems~\ref{mta} and~\ref{fnp}.

\begin{lem}\label{tl}
For $m\in \bN$ let $f,g\in H^1((\overline{\bR_+})_{x_1},H^1(\bR^m_{x'}))$, and suppose $f(0,x')=g(0,x')$.  Let $\delta_\sigma(x_1)\geq 0$ be an approximate identity with compact support in $x_1\leq 0$, and set\footnote{Observe then that  for $x_1\geq 0$, $f^\sigma(x)$ depends only on $f_{|x_1\geq 0}.$}
$$f^\sigma(x) \coloneqq (\delta_\sigma *_{x_1} f) (x)=\int f(x_1-y_1,x')\delta_\sigma(y_1) \,\mathrm dy_1.$$
Then $\|f^\sigma(0,x')-g^\sigma(0,x')\|_{H^1(\bR^m)}=o_\sigma(1).$
\end{lem}

\begin{proof}
We have 
\begin{align}\label{tl2} 
\begin{split}
\|f^\sigma(0,x')-g^\sigma(0,x')\|_{H^1(\bR^m_{x'})}
& =\left\|\int[f(-y_1,z)-g(-y_1,z)]\delta_\sigma(y_1) \,\mathrm dy_1\right\|_{H^1(\bR^m)}\\
& \leq \int \|f(-y_1,z)-g(-y_1,z)\|_{H^1(\bR^m)}\delta_\sigma(y_1) \,\mathrm dy_1.
\end{split}
\end{align}
Fix $\eps>0$.  Since $f-g\in H^1((\overline{\bR_+})_{x_1},H^1(\bR^m_{x'}))\subset C((\overline{\bR_+})_{x_1},H^1(\bR^m_{x'}))$ and $f=g$ on $x_1=0$, the right side of \eqref{tl2} is less than or equal to $ \eps \int\delta_\sigma(y_1)\,\mathrm dy_1=\eps$ for $\sigma$ small enough.
\end{proof}

\begin{rem}\label{tl3}
(i) The conclusion of Lemma \ref{tl}  holds by the same proof if we just assume $f,g\in L^2((\overline{\bR_+})_{x_1},H^1(\bR^m_{x'}))$ and there exists $\alpha>0$ such that $f,g\in H^1([0,\alpha]_{x_1},H^1(\bR^m_{x'}))$.

(ii) For $f\in H^1((\overline{\bR_+})_{x_1},H^1(\bR^m_{x'}))$ a similar proof shows
\begin{align*}
\|f(0,x')-f^\sigma(0,x')\|_{H^1(\bR^m)}=o_\sigma(1).
\end{align*}
\end{rem}

The final lemma allows us to estimate the nonlinear interaction term $f_{\mathrm{nc}}$ in \eqref{ead3}.
\begin{lem}\label{nc}
Write $x=(x_1,x_2,x'')\in \bR^m$ and 
suppose $\phi_1$, $\phi_2$ are   $C^1$ characteristic phases for which the change of variables 
\begin{align}\label{nc1}
\begin{split}
&(x'',\kappa_1,\kappa_2)=F^{-1}(x) \coloneqq (x'',\phi_1(x),\phi_2(x)), \\
&(x_1,x_2,x''))=F(x'',\kappa_1,\kappa_2)=(x_1(x'',\kappa_1,\kappa_2),x_2(x'',\kappa_1,\kappa_2),x'')
\end{split}
\end{align}
is a $C^1$ diffeomorphism in an open set $\cO\subset \bR^m$.  Assume $W_1,W_2$ are in $H^\infty(\bR^m_x\times \bR_\theta)$ with $x$-support in a compact $K\subset\cO$. 
Then the term $f_{\mathrm{nc}}\left(\epsilon^{-\frac12}W_1,\epsilon^{-\frac12}W_2\right)=o_\eps(1)$ in $L^2(\Omega_T)$, where $f_{\mathrm{nc}}$ is defined in \eqref{nc1b}.

\end{lem}

\begin{proof}
\textbf{1. }Let $\frac{1}{2}<\beta<1$.  We estimate
\begin{align*}
\begin{split}
&\int_K |f_{\mathrm{nc}}|^2 \,\mathrm dx \mathrm dt\leq \int_{K_i} |f_{\mathrm{nc}}|^2+\int_{K_{ii}} |f_{\mathrm{nc}}|^2+\int_{K_{iii}} |f_{\mathrm{nc}}|^2 \eqqcolon A+B+C, 
\end{split}
\end{align*}
where
\[ K_i=\{x\in K \mid |\phi_1|\leq\eps^\beta, \ |\phi_2|\leq \eps^\beta\}, \ K_{ii}=\{x\in K \mid |\phi_1|\geq \eps^\beta\}, \ K_{iii}=\{x\in K \mid |\phi_2|\geq \eps^\beta\}. \]
Since $ |f_{\mathrm{nc}}|^2\leq \frac{1}{\eps}|(W_1,W_2)|^2$ and $\mathrm{vol}(K_i)\lesssim \eps^{2\beta}$,  we obtain $A=O(\eps^{2\beta-1})$.

\textbf{2. Estimate the term $B$.} Write 
\begin{align}\label{nc1b}
f_{\mathrm{nc}}=\left[f\left(\epsilon^{-\frac12}W_1,\epsilon^{-\frac12}W_2\right)-f\left(0,\epsilon^{-\frac12}W_2\right)\right]-f\left(\epsilon^{-\frac12}W_1,0\right).
\end{align}
We estimate $B$ in two parts.  First,
\begin{align*}
\int_{K_{ii}}\left|f\left(\epsilon^{-\frac12}W_1,0\right)\right|^2 \,\mathrm dx\lesssim \int_{K_{ii}}\epsilon^{-1}|W_1(x,\phi_1/\eps)|^2 \,\mathrm dx =o_\eps(1)
\end{align*}
by the argument of step {\bf 3} of the proof of Proposition \ref{a11z}.     
Moreover, 
 \begin{align*}
 \left\|f\left(\epsilon^{-\frac12}W_1,\epsilon^{-\frac12}W_2\right)-f\left(0,\epsilon^{-\frac12}W_2\right)\right\|_{L^2(K_{ii})}\lesssim \left\|\epsilon^{-\frac12}W_1\right\|_{L^2(K_{ii})}=o_\eps(1)
 \end{align*}
by the same argument.   The estimate of $C$ is similar.
\end{proof}

\bibliographystyle{alpha}
\bibliography{bib2}

\begin{thebibliography}{JMR00}

\bibitem[AR02]{ar2}
D.~Alterman and J.~Rauch.
\newblock Nonlinear geometric optics for short pulses.
\newblock {\em J. Diff. Eqns}, 178:437--465, 2002.

\bibitem[AR03]{altermanrauch}
D.~Alterman and J.~Rauch.
\newblock Diffractive nonlinear geometric optics for short pulses.
\newblock {\em SIAM J. Math. Analysis}, 34:1477--1502, 2003.

\bibitem[Che96]{cheverry1996}
Christophe Cheverry.
\newblock Propagation d'oscillations pr\`es d'un point diffractif.
\newblock {\em J. Math. Pures Appl.}, 75:419--467, 1996.

\bibitem[CP82]{CP}
J.~Chazaran and A.~Piriou.
\newblock {\em Introduction to the Theory of Linear Partial Differential Equations}.
\newblock North Holland, 1982.

\bibitem[CR04]{cr}
R.~Carles and J.~Rauch.
\newblock Focusing of spherical nonlinear pulses {III}: sub and supercritical cases.
\newblock {\em Tohuku Math. J.}, 56:393--410, 2004.

\bibitem[CW13]{CW1}
J.-F. Coulombel and M.~Williams.
\newblock Nonlinear geometric optics for reflecting uniformly stable pulses.
\newblock {\em J. Diff. Eqns}, 225:1939--1987, 2013.

\bibitem[CW14]{CW2}
J.-F. Coulombel and M.~Williams.
\newblock Amplification of pulses in nonlinear geometric optics.
\newblock {\em J. of Hyperbolic Diff. Eqns.}, 11:749--793, 2014.

\bibitem[Dum02]{dumas2002}
\'{E}ric Dumas.
\newblock Propagation of oscillations near a diffractive point for a semilinear and dissipative {K}lein--{G}ordon equation.
\newblock {\em Communications in Partial Differential Equations}, 27(5--6):953--978, 2002.

\bibitem[JMR95]{jmr1995tams}
Jean-Luc Joly, Guy M\'etivier, and Jeffrey Rauch.
\newblock Focusing at a point and absorption of nonlinear oscillations.
\newblock {\em Transactions of the A.M.S.}, 347:3921--3969, 1995.

\bibitem[JMR96]{jmr1996cpam}
Jean-Luc Joly, Guy M\'etivier, and Jeffrey Rauch.
\newblock Nonlinear oscillations beyond caustics.
\newblock {\em Communications on Pure and Applied Mathematics}, 49:443--527, 1996.

\bibitem[JMR00]{jmr2000mams}
Jean-Luc Joly, Guy M\'etivier, and Jeffrey Rauch.
\newblock Caustics for dissipative semilinear oscillations.
\newblock {\em Memoirs of the American Mathematical Society}, 144(685), 2000.

\bibitem[Mel75]{mel1975}
Richard~B. Melrose.
\newblock Microlocal parametrices for diffractive boundary value problems.
\newblock {\em Duke Math. J.}, 42:605--635, 1975.

\bibitem[MS78]{melsjo1978}
Richard~B. Melrose and Johannes Sj\"{o}strand.
\newblock Singularities of boundary value problems. {I}.
\newblock {\em Communications on Pure and Applied Mathematics}, 31:593--617, 1978.

\bibitem[Sak78]{sakamoto3}
R.~Sakamoto.
\newblock {\em Hyperbolic boundary value problems}.
\newblock Cambridge University Press, 1978.

\bibitem[Tay76]{taylor1976cpam}
Michael~E. Taylor.
\newblock Grazing rays and reflection of singularities of solutions to wave equations.
\newblock {\em Communications on Pure and Applied Mathematics}, 29:1--38, 1976.

\bibitem[Wil15]{willig}
C.~Willig.
\newblock {\em Nonlinear geometric optics for reflecting and evanescent pulses}.
\newblock UNC Chapel Hill PhD thesis, 2015.

\bibitem[Wil22]{williams2022}
Mark Williams.
\newblock Solving eikonal equations by the method of characteristics.
\newblock \url{https://markwilliams.web.unc.edu/wp-content/uploads/sites/19674/2022/01/eikonal.pdf}, 2022.

\bibitem[Wil24]{williams5}
M.~Williams.
\newblock Reflection of conormal pulse solutions to large variable-coefficient semilinear hyperbolic systems.
\newblock {\em Journal of Differential Equations}, 401:93--147, 2024.

\bibitem[WW25]{wangwil}
J.~Wang and M.~Williams.
\newblock Transport of nonlinear oscillations along rays that graze a convex obstacle to any order.
\newblock {\em Annals of PDE}, 11(1):1--74, 2025.

\bibitem[WW26]{wangwil2}
J.~Wang and M.~Williams.
\newblock Convex waves grazing convex obstacles to high order.
\newblock {\em Pure and Applied Analysis}, 8(1):153--188, 2026.

\end{thebibliography}

\end{document}